\documentclass[11pt]{article}
\usepackage{amsmath,enumerate,amsfonts,color,amssymb, amsthm,soul}

\usepackage{hyperref}

\usepackage{tikz}

\def\ZZ{{\mathbb Z}}
\def\Z{{\mathbb Z}}
\def\RR{{\mathbb R}}

\def\Sphere{{\mathbb S}}

\def\mcC{{\mycal C}}
\def\mcA{{\mycal A}}

\def\mcF{{\mycal F}}

\def\mcN{{\mycal N}}
\def\mcP{{\mycal P}}
\def\mcPP{{\Pi}}

\def\nbdry{n}

\def\bN{{\mathbb N}}

\def\eps{{\varepsilon}}

\newtheorem{theorem} {\sc  Theorem\rm} [section]

\newtheorem{lemma} [theorem] {\sc  Lemma\rm}
\newtheorem{proposition} [theorem] {\sc  Proposition\rm}

\newtheorem{definition}[theorem]{\sc  Definition\rm}

\newtheorem{remark}[theorem]{\sc  Remark\rm}

\newcounter{marnote}

\DeclareFontFamily{OT1}{rsfs}{}
\DeclareFontShape{OT1}{rsfs}{m}{n}{ <-7> rsfs5 <7-10> rsfs7 <10-> rsfs10}{}
\DeclareMathAlphabet{\mycal}{OT1}{rsfs}{m}{n}

\def\dist{{\rm dist}}
\def\tr{{\rm tr}}

\def\mcS{{\mycal{S}}}

\def\tilw{{\tilde w}}

\def\bN{{\mathbb N}}
\def\Ss{{\mathbb S}}

\def\be{\begin{equation}}
\def\ee{\end{equation}}

\def\dist{{\rm dist}}
\def\tr{{\rm tr}}

\def\mcS{{\mycal{S}}}

\def\tilw{{\tilde w}}

\def\mcE{{\mycal E}}

\def \f {\varphi}
\def\mcFRdom{{\mcF^R}}

\def\be{\begin{equation}}
\def\ee{\end{equation}}
\def\bea#1\eea{\begin{align}#1\end{align}}
\newcommand{\wkr}{$k$-fold $SO(2)$-symmetric  }
\newcommand{\okr}{$k$-fold $O(2)$-symmetric  }

\def\tiw{\tilde{w}}
\def\haw{\hat{w}}

\def\non{\nonumber}

\def\bw{{\mathbf{w}}}

\numberwithin{equation}{section}

\begin{document}

\title{Symmetry and multiplicity of solutions in a two-dimensional Landau-de Gennes model for liquid crystals}
\author{Radu Ignat\thanks{Institut de Math\'ematiques de Toulouse \& Institut Universitaire de France, UMR 5219, Universit\'e de Toulouse, CNRS, UPS IMT, F-31062 Toulouse Cedex 9, France.
Email: Radu.Ignat@math.univ-toulouse.fr}~, Luc Nguyen\thanks{Mathematical Institute and St Edmund Hall, University of Oxford, Andrew Wiles Building, Radcliffe Observatory Quarter, Woodstock Road, Oxford OX2 6GG, United Kingdom. Email: luc.nguyen@maths.ox.ac.uk}~, Valeriy Slastikov\thanks{School of Mathematics, University of Bristol, University Walk, Bristol, BS8 1TW, United Kingdom. Email: Valeriy.Slastikov@bristol.ac.uk}~ and Arghir Zarnescu\thanks{IKERBASQUE, Basque Foundation for Science, Maria Diaz de Haro 3,
48013, Bilbao, Bizkaia, Spain.}\,  \thanks{BCAM,  Basque  Center  for  Applied  Mathematics,  Mazarredo  14,  E48009  Bilbao,  Bizkaia,  Spain.
(azarnescu@bcamath.org)}\, \thanks{``Simion Stoilow" Institute of the Romanian Academy, 21 Calea Grivi\c{t}ei, 010702 Bucharest, Romania.}}
\date{}

 \maketitle
 
 \begin{abstract}
We consider a variational two-dimensional Landau-de Gennes
model in the theory of nematic liquid crystals in a disk of radius $R$. We prove that under a symmetric boundary condition carrying a topological defect of degree $\frac{k}{2}$ for some given {\bf even} non-zero integer $k$, there are exactly two minimizers for all large enough $R$. We show that the minimizers do not inherit the full symmetry structure of the energy functional and the boundary data. We further show that there are at least five symmetric critical points.
 \end{abstract} 
 
 \tableofcontents

 \section{Introduction} \label{sec:intro}

The questions of symmetry and stability of critical points for the Landau-de Gennes energy functional on two dimensional domains have been recently raised in the mathematical liquid crystal community \cite{BaumanParkPhillips,FRSZ, hu2016disclination, INSZ_AnnIHP,INSZ_CVPDE,KRSZ}. The particular focus of these works was on analyzing special symmetric critical points and investigating their stability properties depending on multiple parameters of the problem.

In this paper we continue the study of the symmetry, stability and multiplicity of critical points of the Landau-de Gennes energy using the same mathematical setting. The main result we establish is  the uniqueness (up to reflection) of the global minimizer in the most relevant physical regime of small elastic constant under the strong anchoring boundary condition which has a topological degree $\frac{k}{2}$ with even nonzero $k$ (see \eqref{BC1}-\eqref{def:n} below). As a consequence of this uniqueness, the minimizers satisfy a $k$-fold $O(2)$-symmetry (see Definition \ref{Def:kfO2}) which has not been identified earlier. Additionally, we prove the existence of two other $k$-fold $O(2)$-symmetric critical points which are not minimizing.

We recall the (non-dimensional) Landau-de Gennes energy functional in the disk $B_R  \subset \RR^2$ of radius $R \in (0,\infty)$ centered at the origin:
\begin{equation}
 \mcF[Q;B_R] = \int_{B_R} \Big[ \frac{1}{2}|\nabla Q|^2 + f_{\rm bulk}(Q)\Big]\,dx, \qquad Q \in H^1(B_R, \mcS_0),
	\label{Eq:LdG}
\end{equation}
where $\mcS_0$ is the set of  {\it $Q$-tensors}:
\begin{equation}
\mcS_0:=\{Q\in \mathbb{R}^{3\times 3}\,:\,tr(Q)=0,\,Q=Q^t\}.
	\label{Def:S0}
\end{equation}

The nonlinear bulk potential is given by 
\[
f_{\rm bulk}(Q) = - \frac{a^2}{2} \tr(Q^2) - \frac{b^2}{3} \tr(Q^3) + \frac{c^2}{4} (\tr(Q^2))^2 - f_*,
\] 
where $a^2\geq 0$, $b^2, c^2 > 0$ are appropriately scaled parameters and the normalizing constant $f_*$ is chosen such that the minimum value of $f_{\rm bulk}$ over $\mcS_0$ is zero. A direct computation gives\footnote{It is sometimes useful to note that the function $s \mapsto -\frac{a^2}{3} s^2 - \frac{2b^2}{27} s^3 + \frac{c^2}{9} s^4$ is minimized at $s = s_+$.}
\begin{equation}
f_* = -\frac{a^2}{3} s_+^2 - \frac{2b^2}{27} s_+^3 + \frac{c^2}{9} s_+^4
	\label{Eq:f*def}
\end{equation}
with 
\be\label{def:s_+}
s_+=\frac{ b^2 + \sqrt{b^4+24 a^2 c^2}}{4 c^2}>0.
\ee
The set of minimizers of 
$f_{\rm bulk}$, which we call the {\it limit manifold}, is given by the following set of uniaxial  $Q$-tensors
\begin{equation} \label{def:limmanifold}
\mcS_*:= \{Q\in\mcS_0: f_{\rm bulk}(Q) = 0\} = \left\{ Q=s_+\left(v \otimes v -\frac13 I_3 \right), \ \ v \in \Sphere^2\right\},
\ee
where $I_3$ is the $3 \times 3$ identity matrix.

The Euler-Lagrange equations satisfied by the critical points of $\mcF$ read
\begin{equation}
\Delta Q = - a^2\,Q - b^2 \big[Q^2 - \frac{1}{3}\tr(Q^2)I_3] + c^2\,\tr(Q^2)\,Q \quad \text{ in } B_R.
	\label{eq:EL}
\end{equation}

The Landau-de Gennes energy describes the pattern formation in liquid crystal systems, in particular, the so-called  defect patterns. A well-studied limit, relating the defects in the Landau-de Gennes framework with those in the Oseen-Frank framework,  is that of small elastic constant (after a suitable non-dimensionalisation -- see \cite{gartlandscalings}), considered, for instance, in \cite{ alama2015weak, BaumanParkPhillips, canevari2015biaxiality, golovaty2014minimizers} in $2D$ and \cite{canevari2017line,Ma-Za,NZ13-CVPDE} in $3D$. Qualitative properties of defects and their stability are studied, for example, in the case of one elastic constant in  $2D$ domains in  \cite{FRSZ, INSZ_AnnIHP, INSZ_CVPDE} and in $3D$ domains in \cite{canevari2016radial, INSZ3, lamyuniaxial}. Numerical explorations of the defects in $2D$ domains and several elastic constants are available in \cite{an2017equilibrium, golovaty2019phase, KRSZ}.

\bigskip
We couple the system \eqref{eq:EL} with the following strong anchoring boundary condition:
\be\label{BC1}
Q(x) = Q_b(x) \text{ on } \partial B_R
\ee
where the map $Q_b :\RR^2\setminus \{0\} \to \mcS_0$ is defined, for some fixed $k \in \ZZ \setminus \{0\}$, by
\begin{align}
Q_b(x) 
	&:=s_+ \left( n(x) \otimes n(x) -\frac{1}{3} I_3 \right) , \qquad x \in \RR^2 \setminus \{0\},
	\label{Qbdef}\\
n(r\cos \varphi, r\sin \varphi) 
	&:= (\cos \frac{k\varphi}{2}, \sin\frac{k\varphi}{2}, 0), \qquad r > 0, \,0 \leq \varphi < 2\pi.  
 \label{def:n} 
 \end{align}
Note that $Q_b$ has image in $\{s_+\left(v\otimes v -\frac13 I_3 \right)\,:\,  v \in \Sphere^1\} \cong \mathbb{R}P^1$ and, as a map from $\partial B_R \cong \mathbb{S}^1$ into $\mathbb{R}P^1$ (see for instance  formula $(8.16)$ in \cite{BrezisCoronLieb}),  has $\frac{1}{2}\ZZ$-valued topological degree $\frac{k}{2}$.

It is worth pointing out the difference between the cases when $k$ is even and odd. If $k$ is even, the vector $n$ defined in \eqref{def:n} is continuous at $\varphi=2\pi$, however, if  $k$ is odd there is a jump discontinuity at $\varphi=2\pi$. Nevertheless the boundary data $Q_b$  defined in terms of $n$ in \eqref{Qbdef} is continuous for any $k\in\ZZ$, but its topological features as a map from $ \partial B_R \cong \mathbb{S}^1$ into $\mathbb{R}P^1$ will depend on the parity of $k$, see Ball and Zarnescu \cite{BallZar}, Bethuel and Chiron \cite{BethuelChiron07-CM}, Brezis, Coron and Lieb \cite{BrezisCoronLieb}, Ignat and Lamy \cite{IgnatLamy19-CVPDE}. In particular, this leads to major qualitative differences in the properties of the critical points; see  \cite{BallZar, FRSZ, INSZ_AnnIHP, INSZ_CVPDE} for analytical studies in $2D$ domains which involve only one elastic constant, and \cite{golovaty2019phase,KRSZ} for numerical studies for several elastic constants. Moreover, in the limit of small elastic constant the minimal Landau-de Gennes energy becomes infinite in the case  of odd $k$ (see \cite{canevari2015biaxiality,golovaty2014minimizers}) and is finite in the case of even $k$ (see \cite{DRSZ2}).  This phenomenon leads to significant differences in the structure and distribution of defects depending on the parity of $k$ (see Appendix~\ref{sec:appendixoddk}).

\bigskip

\noindent {\bf Two group actions on the space $H^1(B_R, \mcS_0)$}. In the following we consider two types of symmetries induced by two group actions on the space  $H^1(B_R, \mcS_0)$ which keep invariant both the energy functional $\mcF$ as well as the boundary condition \eqref{BC1}-\eqref{def:n}.

\medskip

$\bullet$ $k$-{\bf fold} $O(2)$-{\bf symmetry}. 
For $k \in \ZZ \setminus \{ 0 \}$, we introduce the following group action of $O(2)$ on $H^1(B_R, \mcS_0)$. We identify $O(2)\sim \{0,1\} \times \mathbb{S}^1\sim \{0,1\} \times [0,2\pi)$ and define the action of $O(2)$ on $H^1(B_R, \mcS_0)$ by 
\be
\label{Eq:fullyk_rad}
(\alpha, \psi, Q)\in \{0,1\} \times [0,2\pi) \times H^1(B_R, \mcS_0) \mapsto Q_{\alpha,\psi} \in H^1(B_R, \mcS_0).
\ee
Here $Q_{\alpha,\psi}$ is defined  as
\begin{equation}
Q_{\alpha,\psi}(x):= L^\alpha {\mathcal R}_k^t (\psi) Q\bigg(P_2\big(L^\alpha {\mathcal R}_2 (\psi)  \tilde x\big)\bigg) {\mathcal R}_k (\psi) L^\alpha \quad \textrm{for almost  every } x=(x_1,x_2)\in B_R 
	\label{Eq:Qapdef}
\end{equation}
with $\tilde x=(x_1,x_2,0)$, $P_2:\RR^3\to \RR^2$ given by projection $P_2(x_1,x_2,x_3)=(x_1,x_2)$, 
\begin{equation}
\label{eq: R_k }
{\mathcal R}_k (\psi) := \left(\begin{array}{ccc}\cos(\frac{k}{2}\psi) & -\sin(\frac{k}{2}\psi) & 0 \\\sin(\frac{k}{2}\psi) & \cos(\frac{k}{2}\psi) & 0 \\0 & 0 & 1\end{array}\right)
\end{equation}
representing an in-plane rotation about  $e_3=(0,0,1)$ by angle $\frac{k}{2}\psi$, and
\be\label{def:L}
L := \left(\begin{array}{ccc} 1 & 0 & 0 \\0 & -1 & 0 \\0 & 0 & 1\end{array}\right)
\ee
defining the reflection with respect to the plane perpendicular to the  $(0,1,0)$-direction.

\begin{definition}\label{Def:kfO2}
Let $k \in \Z\setminus\{0\}$. The subset of $H^1(B_R, \mcS_0)$ that is invariant under the group action \eqref{Eq:fullyk_rad} is called the set of $k$-{\bf fold} $O(2)$-{\bf symmetric} maps. Such a map $Q\in H^1(B_R, \mcS_0)$ is therefore characterized by
\begin{equation}
Q=Q_{\alpha,\psi} \quad \textrm{ in } B_R\,  \textrm{ for every } (\alpha,\psi) \in \{0,1\} \times [0,2\pi).
	\label{Eq:04IV19-Qsym}
\end{equation} 
\end{definition}

Sometimes when $k$ is clear (uniquely determined) from the context, we will omit ``$k$-fold'' and simply call the above property as $O(2)$-symmetry. The following proposition provides a characterization of $k$-fold $O(2)$-symmetric maps in the case of even $k$. Its proof is postponed until Section~\ref{Sec:WkRSS}.

\begin{proposition}\label{Prop:okrChar}
Let $k \in 2\ZZ\setminus \{0\}$. A map $Q \in H^1(B_R,\mcS_0)$ is \okr if and only if
\begin{equation}
Q(x) =w_0(r)E_0+w_1(r)E_1+w_3(r)E_3 \quad \text{ for a.e. } x = (r\cos \varphi, r \sin\varphi) \in B_R,
	\label{Eq:Qw0w1w3}
\end{equation}
where
\begin{equation}
E_0 = \sqrt{\frac{3}{2}}( e_3 \otimes e_3 - \frac{1}{3} I_3), 
	E_1 = \sqrt{2}(n \otimes n - \frac{1}{2} I_2),
	E_3 = \frac{1}{\sqrt{2}}(n \otimes e_3 + e_3 \otimes n)
	\label{Eq:E013def}
\end{equation}
with $n$ given by \eqref{def:n}, $e_3 = (0,0,1)$, $I_2 = I_3 - e_3 \otimes e_3$, $w_0 \in H^1((0,R);r\,dr)$ and $w_1,  w_3\in H^1((0,R);r\,dr) \cap L^2((0,R);\frac{1}{r}\,dr)$.
\end{proposition}

When $k$ is odd, $k$-fold $O(2)$-symmetric maps are of the form $Q(x)=w_0(r)E_0+w_1(r)E_1$, i.e., $w_3=0$ in \eqref{Eq:Qw0w1w3}. See Remark \ref{rem:kOdd}.

\medskip

$\bullet$ $\Z_2$-{\bf symmetry}. We introduce the group action of $\Z_2$ on $H^1(B_R, \mcS_0)$: 
\be
\label{Eq:k_rad}
(\alpha, Q)\in \Z_2 \times H^1(B_R, \mcS_0) \mapsto J^\alpha Q J^\alpha \in H^1(B_R, \mcS_0)
\ee
where $J$ stands for the reflection with respect to the plane perpendicular to the  $(0,0,1)$-direction.
\be\label{def:J}
J := \left(\begin{array}{ccc} 1 & 0 & 0 \\0 & 1 & 0 \\0 & 0 & -1\end{array}\right).
\ee 

\begin{definition}\label{Def:Z2S}
The subset of $H^1(B_R, \mcS_0)$ that is invariant under the group action \eqref{Eq:k_rad} is called the set of $\Z_2$-{\bf symmetric} maps. Such a map $Q\in H^1(B_R, \mcS_0)$ is therefore characterized by
$$Q=JQJ \quad \textrm{ in } B_R.$$ 
\end{definition}

We will see in Proposition \ref{Prop:Z2Char} that a map $Q\in H^1(B_R, \mcS_0)$ is 
$\Z_2$-symmetric if and only if $e_3 =(0,0,1)$ is an eigenvector of $Q(x)$ for almost all $x\in B_R$.
We note an important difference between the definitions of $k$-fold $O(2)$-symmetry and $\Z_2$-symmetry:
the $O(2)$-action on $H^1(B_R, \mcS_0)$ applies to both the domain and the target space while the $\Z_2$-action applies only to the target space.

It is clear that if $Q$ is a minimizer (or a critical point) of $\mcF$ under the boundary condition \eqref{BC1} then the elements of its orbit under the $k$-fold $O(2)$-action as well as the $\Z_2$-action are also minimizers (or critical points, respectively). A natural question therefore arises: 
do minimizers/critical points of $\mcF$ (under \eqref{BC1}) have $k$-fold $O(2)$-symmetry, or $\Z_2$-symmetry, or both, or maybe none? Some partial answers are available in the literature. In a work of Bauman, Park and Phillips \cite{BaumanParkPhillips}, which is not directly related to symmetry issues, it was shown that, for $|k| \neq 0,1$ and as $R \rightarrow \infty$, there exist none-$O(2)$-symmetric critical points. Their results might tempt one to extrapolate a lack of symmetry in general. This intuition would be also apparently supported by the numerical simulations in Hu, Qu and Zhang \cite{hu2016disclination} which observed lack of symmetry for a certain radius. However, in \cite{KRSZ} the $k$-fold $O(2)$-symmetry was numerically observed for a minimizer in the case of  even $k$ and large enough radius $R$ (probably larger than in the examples explored numerically in \cite{hu2016disclination}).

\begin{definition}\label{Def:kSym}
For $k \in \ZZ \setminus \{ 0 \}$, a map $Q\in H^1(B_R, \mcS_0)$ is called ({\bf $k$-fold}) $\Z_2 \times O(2)$-{\bf symmetric} if $Q$ is both ($k$-fold) $O(2)$-symmetric and $\Z_2$-symmetric. 
\end{definition}

We will see later (in Section \ref{Sec:WkRSS}) that all $\Z_2\times O(2)$-symmetric maps are of the form \footnote{In particular, in view of Remark \ref{rem:kOdd}, if $k$ is odd, all \okr maps are $\Z_2\times O(2)$-symmetric.}
\begin{equation}
Q(x) = w_0(|x|) E_0 + w_1(|x|) E_1 \textrm{ for a.e. }x\in B_R.
	\label{Eq:Qkrs}
\end{equation}
It is known from \cite{INSZ_AnnIHP} that all $\Z_2\times O(2)$-symmetric critical points of $\mcF$ coincide with the so-called $k$-radially symmetric critical points. See Section \ref{Sec:WkRSS} for more details. Note that the boundary data $Q_b$ defined in \eqref{Qbdef} is $\Z_2\times O(2)$-symmetric on $\partial B_R$. However, we will prove that the minimizers of $\mcF[\cdot; B_R]$ under the boundary condition \eqref{Qbdef} do not satisfy this symmetry (namely they are not $\Z_2$-symmetric).

 The structure and stability properties of $\Z_2\times O(2)$-symmetric critical points were investigated in \cite{FRSZ, INSZ_AnnIHP, INSZ_CVPDE}. In particular, it was proved that 
\begin{itemize}
\item when $b=0$ and $R<\infty$ they are  minimizers of the Landau-de Gennes energy for all $k \in \ZZ \setminus\{0\}$ (see \cite{FRSZ});
\item when $b \neq 0$, $\bN \ni |k|>1$ and $R$ is large enough they are unstable (see \cite{INSZ_AnnIHP});
\item when $b \neq 0$ and $k=\pm 1$ they  are locally stable for all $R \leq \infty$ under suitable condition on $w_0$ and $w_1$  in \eqref{Eq:Qkrs} (see \cite{INSZ_CVPDE}).\footnote{For $b^4 \leq 3a^2c^2$, the condition reduces to $w_0 < 0$ and $w_1 > 0$. }
\end{itemize}

In this paper we focus on the case
$$
k\in 2\mathbb{Z} \setminus \{0\},
$$ 
where the $k$-fold $O(2)$-symmetry does not imply in general the $\Z_2$-symmetry (some remarks on the case $k$ odd are provided in Appendix~\ref{sec:appendixoddk}). Our main result states that for large enough radius $R$ the Landau-de Gennes energy \eqref{Eq:LdG} under the boundary condition \eqref{BC1} has exactly two minimizers and these minimizers are $k$-fold $O(2)$-symmetric and $\Z_2$-conjugate to each other.

\begin{theorem} \label{thm:MinSymmetry}
Let $a^2 \geq 0, b^2, c^2 >0$ be any fixed constants and $k \in 2\ZZ \setminus \{0\}$. There exists some $R_0 = R_0(a^2, b^2, c^2, k) > 0$ such that for all $R > R_0$, there exist exactly two global minimizers $Q^\pm_R$ of $\mcF[\cdot;B_R]$ subjected to the boundary condition \eqref{BC1} and these minimizers are $k$-fold $O(2)$-symmetric (but not $\Z_2\times O(2)$-symmetric). The minimizers $Q^\pm_R$ are $\Z_2$-conjugate to one another, namely, $Q^\pm_R=JQ^\mp_R J\neq Q^\mp_R$ and have the form 
\be
\label{pag7}
Q^\pm_R (x) = w_0(|x|) E_0 + w_1(|x|) E_1 \pm w_3(|x|)  E_3 \quad \textrm{for every } x\in B_R,
\ee
where $E_0$, $E_1$ and $E_3$ are given by \eqref{Eq:E013def} and $w_3 > 0$ in $(0,R)$.
\end{theorem}

It is clear that the Euler-Lagrange equation \eqref{eq:EL} for $Q_R^\pm$ then reduces to a system of ODEs for $(w_0, w_1, 0, \pm w_3, 0)$ with the boundary condition $w_0(R) = -\frac{s_+}{\sqrt{6}}$, $w_1(R) = \frac{s_+}{\sqrt{2}}$ and $w_3(R) = 0$. See Remark \ref{Rem:ODESys}.

The idea of the proof of Theorem~\ref{thm:MinSymmetry}  is presented in Section~\ref{SSec:Heu}. An assumption in the above theorem concerns the radius of the domain which is taken to be large enough. This is a physically relevant assumption, capturing the most interesting physical regime of small elastic constant (as explained in \cite{gartlandscalings} and studied,  for instance, in \cite{alama2015weak, BaumanParkPhillips, canevari2015biaxiality, DRSZ2, golovaty2014minimizers} in $2D$ and \cite{canevari2017line,Ma-Za,NZ13-CVPDE} in $3D$).

We would like to draw the attention to our related uniqueness  results in a Ginzburg-Landau settings \cite{INSZ18_CRAS,INSZ_AnnENS} where the bulk potential satisfies a suitable global convexity assumption. In these articles, we established a link between the so-called non-escaping phenomenon and uniqueness of minimizers. In the context of $Q$-tensors, a non-escaping phenomenon would mean the existence of $O(2)$-symmetric critical point $Q$ such that $Q \cdot E_3$ does not change sign. While it is not hard to prove the existence of such critical points for large $R$ (see the last paragraph in the proof of Theorem \ref{thm:MinSymmetryRes}), the method in \cite{INSZ18_CRAS,INSZ_AnnENS} does not apply to the present setting as our bulk potential $f_{\rm bulk}$ does not satisfy the relevant global convexity. In a sequel to the present article, we will apply the method developed here to prove a similar uniqueness result for minimizers of a Ginzburg-Landau type energy functional where the bulk potential satisfies only a local convexity property near the limit manifold.

Our second result concerns the multiplicity of   $k$-fold $O(2)$-symmetric  critical points of $\mcF$. This is coherent with
the numerical simulations in \cite[Section $3.2$]{KRSZ} for $k=2$ and \cite[Section $2.2$]{hu2016disclination}, which observed, for large enough $R$, that there can be several distinct solutions, corresponding to  boundary conditions \eqref{BC1}.

\begin{theorem}\label{thm:FiveSol}
Let $a^2 \geq 0, b^2, c^2 >0$ be any fixed constants and $k \in 2\ZZ \setminus \{0\}$. There exists some $R_1 = R_1(a^2, b^2, c^2, k) > 0$ such that for all $R > R_1$, there exist at least five  $k$-fold $O(2)$-symmetric critical points of $\mcF[\cdot;B_R]$ subjected to the boundary condition \eqref{BC1}. At least four of these solutions are {\bf not} $\Z_2 \times O(2)$-symmetric.
\end{theorem}

The rough idea of proving Theorem~\ref{thm:FiveSol} is the following:  Theorem \ref{thm:MinSymmetry} gives us two global minimizers $Q_R^\pm$. By the mountain pass theorem, there is a mountain pass critical point, denoted  $Q^{mp}_R$ that connects these two ($\Z_2$-conjugate) minimizers $Q_R^\pm$. The main point in the proof of Theorem \ref{thm:FiveSol} is to show that the mountain pass solution $Q^{mp}_R$ does not coincide with the $k$-radially symmetric critical point $Q^{str}_R$ constructed in \cite{INSZ_AnnIHP}. This is done by an energy estimate showing in particular the existence of paths between $Q_R^\pm$ for which the energy is uniformly bounded with respect to $R$, see \eqref{Eq:MPPathEnergy}. As the maps $Q_R^\pm$, after suitably rescaled, converge to two $\mcS_*$-valued minimizing harmonic maps of  different topological nature, this highlights the difficulty of constructing that path; see Section \ref{sec:multiplesol} for a more detailed discussion. Moreover, we show that the mountain pass critical point is not $\Z_2$-symmetric, thus its $\Z_2$-conjugate $\tilde Q^{mp}_R$ is also a critical point, thus yielding  five different critical points.

In Table \ref{Table1}, we summarize the properties of the critical points from Theorem~\ref{thm:FiveSol}.\footnote{See Appendix~\ref{sec:appendixoddk} for related remarks regarding the case when $k$ is odd.}

\begin{table}[ht]\label{Table1}
\caption{Properties of critical points for even $k\not=0$ and large radius $R$} 
\vskip 0.2cm
\centering 
\begin{tabular}{c c c c c}
\hline
\hline   Critical Point & Stability & $w_3$ in \eqref{Eq:Qw0w1w3} & Symmetry & Energy as $R\to \infty$\\  [0.5ex]
\hline
 $Q_R^\pm$ & Yes & $w_3^+ = - w_3^- > 0$ & $O(2)$-symmetry & $ 4\pi s_+^2 |k|  + o_k(1) $ \\ 
 $Q_R^{str}$ & No & $w_3 \equiv 0$ & $\Z_2\times O(2)$-symmetry & $ \frac{\pi k^2 s_+^2}{2} \ln R + o_k(\ln R)$ \\
$Q_R^{mp}$ and $\tilde Q_R^{mp}$ & No & $w_3 = - \tilde w_3 \not\equiv 0$ &  $O(2)$-symmetry & $O_k(1)$  \\ [1ex]
\hline
\end{tabular}
\end{table}

The two global minimizers $Q_R^\pm$ as well as the two mountain pass critical points $Q_R^{mp}$ and $\tilde Q_R^{mp}$ are $k$-fold $O(2)$-symmetric but not $\Z_2$-symmetric. The map $Q_R^{str}$ is 
a minimizer among $\Z_2 \times O(2)$-symmetric maps.\footnote{It is an open problem if the $w_0$ and $w_1$ components of $Q_R^{str}$ satisfy $w_0 < 0$ and $w_1 > 0$. See \cite[Open problem 3.2]{INSZ_AnnIHP}.} The subscript $k$ in the little $o$-terms indicates that the rate of convergence may depend on $k$. The subscript $k$ in the big $O$-terms indicates that the implicit constant may depend on $k$. The energy of $Q_R^\pm$ is bounded from above by and converges as $R\to \infty$ to the Dirichlet energy of the $\mcS_*$-valued minimal harmonic map(s) on $B_1$, which is $4\pi s_+^2|k|$; see \eqref{Eq:UnifEnergyBnd}. The asymptotic behavior of the energy of $Q_R^{str}$ as $R\to \infty$ is proved in Lemma \ref{Lem:StrongLogDiv}. The estimate for the energy of the mountain pass solutions is given in \eqref{Eq:MPEnergy}. In addition to these solutions, we also have the non-$O(2)$-symmetric solutions constructed in  \cite{BaumanParkPhillips}, which have energy 
$O(|k| \ln R)$ for large $R$, see \cite[Theorem $B$]{BaumanParkPhillips}.
 
To dispel confusion, we note that the Ginzburg-Landau counterpart for our model is the $2D-3D$ Ginzburg-Landau model (see \cite[Theorem 1.1]{INSZ_AnnENS}). In particular, the minimal energy remains bounded as $R \rightarrow \infty$, which is contrary to the $2D-2D$ Ginzburg-Landau case where the minimal energy grows like $\ln R$ as $R \rightarrow \infty$ (see e.g. the seminal book of B\'ethuel, Brezis and H\'elein \cite{vortices} or \cite{SS}). In the $2D-3D$ case, it was shown in \cite{INSZ_AnnENS} that, for every $k \in 2\ZZ\setminus\{0\}$ and under the boundary condition \eqref{def:n}, there exists $R_* > 0$ such that the Ginzburg-Landau energy functional has a unique critical point for $R \leq R_*$ and has exactly two minimizers which `escape in the third dimension' for $R > R_*$.

\medskip

The paper is organized as follows: In Section~\ref{Sec:WkRSS} we present some basic facts about the two types of symmetry induced by the $O(2)$- and $\Z_2$-group actions, and, in particular, about $k$-fold $SO(2)$-symmetric minimizers of the Landau-de Gennes energy. Section~\ref{Sec:MinSym} contains the main part of the paper, namely, the proof of Theorem~\ref{thm:MinSymmetry}.  The overall idea and main mathematical set-up of the proof are described in the Sections~\ref{SSec:Heu} -- \ref{SSec:ELpsiP}. Sections~\ref{ssec:LinHM} -- \ref{ssec:unique} contain formulations and proofs of the auxiliary results used in Sections~\ref{SSec:ProofUniq} -- \ref{SSec:ProofUniqThm1.5} to prove Theorem~\ref{thm:MinSymmetry}. In Section~\ref{sec:multiplesol} we prove the existence of multiple critical points for large enough domains, namely Theorem~\ref{thm:FiveSol}. In Appendix~\ref{sec:appendixoddk} we provide a couple of remarks on the minimal energy and the symmetry properties of minimizers  of $\mcF[\cdot;B_R]$ for odd $k$. Finally, in the Appendices~\ref{App:Tech}, \ref{App:LpaInv}, \ref{app:Calc} we put some technical details required to prove our results.

\section{Structure of symmetric maps. Proof of Proposition \ref{Prop:okrChar} }\label{Sec:WkRSS}

We work with a moving (i.e., $x$-dependent) orthonormal basis of the space $\mcS_0$ (defined in \eqref{Def:S0}), which is compatible with the boundary condition \eqref{BC1}. We use polar coordinates in $\RR^2$, i.e., $x=(r \cos \varphi, r \sin \varphi)$ with $r>0$ and $\varphi\in [0, 2\pi)$. Let $\{e_i\}_{i=1}^3$ be the standard basis of $\RR^3$, and let
\begin{equation}
n(x) = (\cos\frac{k\varphi}{2},\sin\frac{k\varphi}{2},0), \quad m(x) =  (-\sin\frac{k\varphi}{2},\cos\frac{k\varphi}{2}, 0), \quad x\in \RR^2.
	\label{Eq:03V19-mndef}
\end{equation}
We endow $\mcS_0$ with the Frobenius scalar product of symmetric matrices $Q \cdot P = \tr(Q\,P)$ and the induced norm $|Q|=(Q\cdot Q)^{1/2}$. We define, for $x\in \RR^2$, the following orthonormal basis of $\mcS_0$:
 \begin{align} \label{basisS}
 E_0 
 	&= \sqrt{\frac{3}{2}} (e_3 \otimes e_3 - \frac{1}{3}I_3), \ 
	E_1 = \sqrt{2}(n \otimes n - \frac{1}{2} I_2), \ E_2 = \frac{1}{\sqrt{2}} (n \otimes m + m \otimes n),\\
E_3
	&= \frac{1}{\sqrt{2}}(n \otimes e_3 + e_3 \otimes n), \ E_4	= \frac{1}{\sqrt{2}}(m \otimes e_3 + e_3 \otimes m) \non.
 \end{align}
Recall that $I_3$ is the $3 \times 3$ identity matrix and $I_2 = I_3 -e_3\otimes e_3$. It should be noted that this choice of basis elements for $\mcS_0$ differs slightly from \cite{INSZ_AnnIHP, INSZ_CVPDE} where both even and odd values of $k$ are considered. This is due to the fact that $E_3$ and $E_4$ are continuous when we identify $\varphi = 0$ with $\varphi = 2\pi$ if and only if $k$ is even.

We identify a map $Q: B_R \rightarrow \mcS_0$ with a map $\bw = (w_0, \ldots, w_4): B_R \rightarrow \RR^5$ via $Q = \sum_{i = 0}^4 w_i\,E_i$. Then $|Q|^2= |\bw|^2$,
\begin{align}
\tr(Q^3)
	&=  \frac{\sqrt{6}}{12}\big[2w_0^3 - 6w_0(w_1^2 + w_2^2) + 3w_0(w_3^2 + w_4^2) \nonumber\\
			&\qquad\qquad\qquad\qquad+ 3\sqrt{3}w_1(w_3^2 - w_4^2) + 6\sqrt{3}w_2 w_3 w_4\big],\nonumber\\
|\nabla Q|^2 
	&= |\partial_r \bw|^2 + \frac{1}{r^2}\Big[|\partial_\varphi w_0|^2 + |\partial_\varphi w_1 - k w_2|^2 + |\partial_\varphi w_2 + k w_1|^2\nonumber\\
		&\qquad\qquad  + |\partial_\varphi w_3 - \frac{k}{2} w_4|^2 + |\partial_\varphi w_4 + \frac{k}{2} w_3|^2\Big],
		\label{Eq:GQ2}
\end{align}
where we have used the following identities for even $k$
\[
\partial_\varphi E_1 = k E_2,\quad \partial_\varphi E_2 = -k E_1, \quad\partial_\varphi E_3 = \frac{k}{2} E_4,\quad \partial_\varphi E_4 = -\frac{k}{2} E_3.
\]
The Landau-de Gennes energy \eqref{Eq:LdG} becomes
\begin{align}
\mcF[Q;B_R]
	= I[\bw] 
		&:= \int_{B_R} \Big\{\frac{1}{2}|\partial_r \bw|^2 + \frac{1}{2r^2}\Big[|\partial_\varphi w_0|^2 + |\partial_\varphi w_1 - k w_2|^2 + |\partial_\varphi w_2 + k w_1|^2\nonumber\\
		&\qquad\qquad  + |\partial_\varphi w_3 - \frac{k}{2} w_4|^2 + |\partial_\varphi w_4 + \frac{k}{2} w_3|^2\Big]\nonumber\\
 		&\qquad\qquad + \big(-\frac{a^2}{2} + \frac{c^2}{4}|\bw|^2\big)|\bw|^2\nonumber\\
		&\qquad\qquad - \frac{b^2\sqrt{6}}{36}\big[2w_0^3 - 6w_0(w_1^2 + w_2^2) + 3w_0(w_3^2 + w_4^2) \nonumber\\
			&\qquad\qquad\qquad\qquad+ 3\sqrt{3}w_1(w_3^2 - w_4^2) + 6\sqrt{3}w_2 w_3 w_4\big] - f_*\Big\} r\,dr\, d\varphi.
			\label{Eq:LdGW}
\end{align}

The boundary condition \eqref{BC1} becomes
\begin{equation}
\bw(x) = (-\frac{s_+}{\sqrt{6}}, \frac{s_+}{\sqrt{2}},0,0,0) \text{ on } \partial B_R.
	\label{Eq:BC1W}
\end{equation}

In the introduction, we defined a group action of $O(2)$ on $H^1(B_R,\mcS_0)$. There we viewed $O(2)$ as a direct product of $\{0,1\}$ and $SO(2)$. This naturally induces two group actions of $\{0,1\}\cong \Z_2$ and of $SO(2)$, as subgroups of $O(2)$, on $H^1(B_R,\mcS_0)$.

\begin{definition}\label{Def:kfSO2}
Let $k \in \Z\setminus\{0\}$. A map $Q \in H^1(B_R, \mcS_0)$ is said to be $\{0,1\}$-{\bf symmetric} if
\begin{equation}
Q = Q_{1,0}  \quad \textrm{ in } B_R.
	\label{Eq:05IV19-Q01sym}
\end{equation} 
A map $Q \in H^1(B_R, \mcS_0)$ is said to be $k$-{\bf fold} $SO(2)$-{\bf symmetric} if
\begin{equation}
Q=Q_{0,\psi} \quad \textrm{ in } B_R\,  \textrm{ for every } \psi \in [0,2\pi).
	\label{Eq:05IV19-QSOsym}
\end{equation} 
Here $Q_{\alpha,\psi}$ is defined by \eqref{Eq:Qapdef}.
\end{definition}

Note that the groups $\Z_2$ and $\{0,1\}$ are isomorphic, but we have deliberately distinguished the notations to avoid confusion with the $\Z_2$-action defined in the introduction. Moreover, the nature of the two group actions are somewhat different. The $\Z_2$-action is related to the reflection along $e_3$ direction of the target, while the $\{0,1\}$-action is related to the reflection along the $e_2$ direction in both the domain and the target.

\begin{definition}[{\cite[Definition 1.1]{INSZ_AnnIHP}}]\label{Def:kfZ2SO2}
Let $k \in \Z\setminus\{0\}$. A map $Q \in H^1(B_R, \mcS_0)$ is said to be $k$-{\bf radially symmetric} (or equivalently $\Z_2 \times SO(2)$-{\bf symmetric}) if $Q$ is $\Z_2$-symmetric and $k$-fold $SO(2)$-symmetric.
\end{definition}

The $k$-fold $SO(2)$-symmetry is exactly condition ${\rm \bf (H2)}$ in \cite[Definition 1.1]{INSZ_AnnIHP}. We have the following characterization.
\begin{proposition}\label{Prop:QWRSRep}
Let $R \in (0,\infty]$ and $k \in 2\ZZ\setminus \{0\}$. A map $Q \in H^1(B_R,\mcS_0)$ is \wkr if and only if it can be represented for a.e. $x = r(\cos\varphi,\sin\varphi) \in B_R$ as
\[
Q(x) = \sum_{i=0}^4 w_i(r)\,E_i \text{ for a.e. } x = r(\cos\varphi,\sin\varphi) \in B_R
\]
where $E_i$'s are given by \eqref{basisS}, $w_i = Q \cdot E_i$, $w_0 \in H^1((0,R);r\,dr)$ and $w_1, w_2, w_3, w_4 \in H^1((0,R);r\,dr) \cap L^2((0,R);\frac{1}{r}\,dr)$.
\end{proposition}

\begin{proof}
Suppose that $Q \in H^1(B_R,\mcS_0)$ is \wkr. By \cite[Proposition 2.1]{INSZ_AnnIHP}, there exist $w_0 \in H^1((0,R);r\,dr)$ and $w_1, w_2, \tilw, \haw \in H^1((0,R);r\,dr) \cap L^2((0,R);\frac{1}{r}\,dr)$ such that
\begin{align*}
Q(x)
	& = \sum_{i=0}^2 w_i(r)\,E_i 
		+ (\tiw(r)\cos \frac k 2\f+\haw(r) \sin \frac k 2\f)\frac{1}{\sqrt{2}}(e_1 \otimes e_3+e_3\otimes e_1)\\
		&\qquad +(-\haw(r)\cos \frac k 2\f+\tiw(r) \sin \frac k 2\f)\frac{1}{\sqrt{2}}\left(e_2\otimes e_3+e_3\otimes e_2\right)\\
	& = \sum_{i=0}^2 w_i(r)\,E_i 
		+ \tiw(r) E_3 - \haw(r) E_4
		\text{  for a.e. } x \in B_R.
\end{align*}
This gives the desired representation.

Consider the converse. Suppose that $Q(x) = \sum_{i=0}^4 w_i(r) E_i$. A direct check shows that the basis elements $E_0, \ldots, E_4$ are $k$-fold $SO(2)$-symmetric. Hence $Q$ is $k$-fold $SO(2)$-symmetric. If we have further that $w_0 \in H^1((0,R);r\,dr)$ and $w_1, \ldots, w_4 \in H^1((0,R);r\,dr) \cap L^2((0,R);\frac{1}{r}\,dr)$, then, by \eqref{Eq:GQ2}, $|\nabla Q|$ is square integrable over $B_R$. It follows that $Q \in H^1(B_R, \mcS_0)$.
\end{proof}

\begin{remark}\label{Rem:ODESys}
Let $Q$ be a $SO(2)$-symmetric map. $Q$ is a critical point of $\mcF$ if and only if its components $w_0, \ldots, w_4$ satisfy
\begin{align}
w_0'' + \frac{1}{r}w_0' 
	&= w_0\Big(-a^2 + c^2 |\bw|^2 -\frac{b^2}{\sqrt{6}}w_0\Big) + \frac{b^2}{\sqrt{6}}(w_1^2 + w_2^2) - \frac{b^2}{2\sqrt{6}}(w_3^2 + w_4^2),\label{Eq:w0}\\
w_1'' + \frac{1}{r}w_1' - \frac{k^2}{r^2}\,w_1 
	&= w_1\Big(-a^2 + c^2 |\bw|^2 + \frac{2b^2}{\sqrt{6}}w_0\Big)  - \frac{b^2}{2\sqrt{2}}(w_3^2 - w_4^2),\label{Eq:w1}\\
w_2'' + \frac{1}{r}w_2' - \frac{k^2}{r^2}\,w_2 
	&= w_2\Big(-a^2 + c^2 |\bw|^2 + \frac{2b^2}{\sqrt{6}}w_0\Big)  - \frac{b^2}{\sqrt{2}}w_3 \, w_4,\label{Eq:w2}\\
w_3'' + \frac{1}{r}w_3' - \frac{k^2}{4r^2}\,w_3 
	&= w_3\Big(-a^2 + c^2 |\bw|^2 - \frac{b^2}{\sqrt{6}}w_0 - \frac{b^2}{\sqrt{2}}w_1\Big)  - \frac{b^2}{\sqrt{2}} w_2\,w_4,\label{Eq:w3}\\
w_4'' + \frac{1}{r}w_4' - \frac{k^2}{4r^2}\,w_4 
	&= w_4\Big(-a^2 + c^2 |\bw|^2 - \frac{b^2}{\sqrt{6}}w_0 + \frac{b^2}{\sqrt{2}}w_1\Big)  - \frac{b^2}{\sqrt{2}} w_2\,w_3.\label{Eq:w4}
\end{align}
\end{remark}

 On the other hand, the $\{0,1\}$-symmetry imposes that the $w_2$ and $w_4$ components are odd in the $x_2$ variable. More precisely we have the following:

\begin{proposition}
Let $R \in (0,\infty]$ and suppose that $Q \in H^1(B_R, \mcS_0)$. Then $Q$ is $\{0,1\}$-symmetric if and only if
\begin{equation}
w_i(x_1,-x_2)=-w_i(x_1,x_2), \,\textrm{ for a.e. }  (x_1,x_2)\in B_R, i\in\{2,4\} 
	\label{Eq:01Char}
\end{equation}
where $w_i=Q \cdot E_i$ and the $E_i$'s are defined in \eqref{basisS}.
\end{proposition}
\begin{proof}

The $\{0,1\}$-symmetry is equivalent to $Q(x)=LQ(P_2(L\tilde x))L$ for a.e. $x\in B_R$, where $P_2$, $L$ and $\tilde x$ are as in \eqref{Eq:Qapdef}. This means
\[
\sum_{i=0}^4 w_i(x_1,x_2)E_i(x_1,x_2)=\sum_{i=0}^4 w_i(x_1,-x_2)LE_i(x_1,-x_2)L.
\] 
Noting that $Lm(x_1,-x_2) = -m(x_1,x_2)$, $L n(x_1,-x_2) = n(x_1,x_2)$,
\begin{align*}
LE_i(x_1,-x_2)L
	&=E_i(x_1,x_2) \text{ for } i \in \{0,1,3\},\\
LE_j(x_1,-x_2)L
	&=-E_j(x_1,x_2)\text{ for } j \in \{2,4\},
\end{align*}
we obtain the conclusion.
\end{proof}

\begin{proof}[Proof of Proposition \ref{Prop:okrChar}]
The $k$-fold $SO(2)$-symmetry implies, by the Proposition~\ref{Prop:QWRSRep}, that the components $w_i, i\in\{0,\dots,4\}$ are radial, hence, in particular, we have $w_2(x_1,-x_2)=w_2(x_1,x_2)$ and $w_4(x_1,-x_2)=w_4(x_1,x_2)$. Since we also have $\{0,1\}$-symmetry, the relations \eqref{Eq:01Char} hold, which lead to  $w_2=w_4\equiv 0$, as claimed.
\end{proof}

\begin{remark}\label{rmk:kradsym}
We recall \cite[Corollary 2.2]{INSZ_AnnIHP} that, for $k \in \ZZ \setminus \{0\}$, $k$-radially symmetric maps (i.e., $k$-fold $\Z_2\times SO(2)$-symmetric maps) are of the form
\[
Q(x) =w_0(r)E_0+w_1(r)E_1+w_2(r)E_2.
\]
For $k\in 2\ZZ\setminus \{0\}$, note the difference between the non-zero components of $O(2)$-symmetric maps (i.e., $\{0,1\}\times SO(2)$-symmetric maps) and $k$-radially symmetric maps: $O(2)$-symmetric maps have components $w_0, w_1$ and $w_3$ while $k$-radially symmetric maps have components $w_0$, $w_1$ and $w_2$.
\end{remark}

\begin{remark}\label{rem:kOdd}
Let $R \in (0,\infty]$, $k$ be an odd integer and $Q \in H^1(B_R,\mcS_0)$. Then $Q$ is $k$-fold $O(2)$-symmetric if and only if
\[
Q(x) =w_0(r)E_0+w_1(r)E_1
\]
where $w_0 \in H^1((0,R);r\,dr)$ and $w_1 \in H^1((0,R);r\,dr) \cap L^2((0,R);\frac{1}{r}\,dr)$. Indeed, by \cite[Proposition 2.1]{INSZ_AnnIHP}, the $k$-fold $SO(2)$-symmetry implies $Q = w_0(r)E_0+w_1(r)E_1 + w_2(r)E_2$. Then $\{0,1\}$-symmetry implies by \eqref{Eq:01Char} that $w_2 \equiv 0$. The converse is clear (cf. \eqref{Eq:GQ2}).
\end{remark}

We also have the following characterization of $\Z_2$-symmetry:

\begin{proposition}\label{Prop:Z2Char}
Let $R \in (0,\infty]$ and suppose that $Q\in H^1(B_R, \mcS_0)$. Then the following conditions are equivalent.
\begin{enumerate}
\item $Q$ is $\Z_2$-symmetric.
\item $e_3=(0,0,1)$ is an eigenvector of $Q(x)$ for almost all $x \in B_R$.\footnote{This is the assumption {\rm \bf (H1)} in \cite[Definition 1.1]{INSZ_AnnIHP}}
\item $Q(x) = \sum_{i=0}^2 w_i(x) E_i$ for almost all $x \in B_R$.
\end{enumerate}
\end{proposition}
\begin{proof} We know that any $Q \in H^1(B_R,\mcS_0)$ can be represented as $Q(x)=\sum_{i=0}^4 w_i(x) E_i$. 

\medskip
\noindent {\it Step 1. We prove ($1 \implies 2$)}. Suppose $Q$ is $\Z_2$-symmetric. Then 
$$
Q(x) = \sum_{i=0}^4 w_i(x) E_i = J Q(x) J =\sum_{i=0}^4 w_i(x) J E_i J = \sum_{i=0}^2 w_i(x) E_i - w_3(x) E_3 - w_4(x) E_4.
$$
Therefore we obtain $w_3=w_4=0$, $Q(x) = \sum_{i=0}^2 w_i(x) E_i$ and hence $e_3$ is an eigenvector of $Q(x)$ for a.e. $x \in B_R$.  

\medskip
\noindent {\it Step 2. We prove ($2 \implies 3$)}. Assume now that $e_3$ is an eigenvector of $Q(x)$ for a.e. $x \in B_R$. Therefore there exists $\lambda(x)$ such that 
$$
\lambda(x) e_3 = Q(x) e_3 = \sum_{i=0}^4 w_i(x) E_i e_3 =  \sqrt{\frac23} w_0(x) e_3 +  \frac{1}{\sqrt{2}} w_3(x) n +  \frac{1}{\sqrt{2}}  w_4(x) m.
$$
Since $e_3, n$ and $m$ form an orthonormal basis of $\RR^3$, it is clear that $w_3(x)=w_4(x)=0$ a.e. $x \in B_R$.

\medskip
\noindent {\it Step 3. We prove ($3 \implies 1$)}. Assume that $Q(x) = \sum_{i=0}^2 w_i(x) E_i$ then it is straightforward to check that $Q = J Q J$, i.e. $Q$ is $\Z_2$-symmetric.
\end{proof}

We now give a characterization of $\Z_2 \times O(2)$-symmetric maps.

\begin{proposition}\label{prop:Z2O2}
Let $R \in (0,\infty]$ and $k \in 2\ZZ\setminus \{0\}$. A map $Q \in H^1(B_R,\mcS_0)$ is $\Z_2 \times O(2)$-symmetric if and only if
\[
Q(x) =w_0(r)E_0+w_1(r)E_1 \text{ for a.e. } x = r(\cos\varphi,\sin\varphi) \in B_R
\]
where $w_0 \in H^1((0,R);r\,dr)$ and $w_1\in H^1((0,R);r\,dr) \cap L^2((0,R);\frac{1}{r}\,dr)$.

\end{proposition}

\begin{proof}
By Proposition \ref{Prop:okrChar}, the $O(2)$-symmetry implies
 \[
Q(x) =w_0(r)E_0+w_1(r)E_1+w_3(r)E_3
\]
where $w_0 \in H^1((0,R);r\,dr)$ and $w_1,  w_3\in H^1((0,R);r\,dr) \cap L^2((0,R);\frac{1}{r}\,dr)$.  By Proposition~\ref{Prop:Z2Char} the $\Z_2$-symmetry implies that $w_3 \equiv 0$. The converse is clear.
\end{proof}

For $k \in 2\ZZ\setminus\{0\}$, we next note a connection of $SO(2)$-symmetric and $O(2)$-symmetric minimizers: Under an $O(2)$-symmetric boundary condition, in particular \eqref{BC1}, $SO(2)$-symmetric minimizers of $\mcF$ are in fact $O(2)$-symmetric. We do not know however if this remains true for all $SO(2)$-symmetric critical points. See \cite[Proposition 1.3 and Remark 1.4]{INSZ_AnnIHP} for a related statement that $\Z_2 \times SO(2)$ critical points are in fact $\Z_2 \times O(2)$-symmetric.

\begin{proposition}\label{wrmin}
Let $R \in (0,\infty)$ and $k \in 2\ZZ\setminus \{0\}$. If $Q \in H^1(B_R,\mcS_0)$ is a minimizer of $\mcF[\cdot; B_R]$ in the set of all \wkr $Q$ tensors satisfying an $O(2)$-symmetric boundary condition, then $Q$ satisfies \eqref{eq:EL} and
\[
Q(x) = w_0(|x|)\,E_0 + w_1(|x|)\,E_1 + w_3(|x|)\,E_3 \text{ for all } x \in B_R.
\]
In addition, if the boundary data is $\Z_2 \times O(2)$-symmetric \footnote{Note that, unlike in \eqref{BC1}, we are not assuming that the boundary data be uniaxial in this statement.}, then $\tilde Q =JQJ = w_0\,E_0 + w_1\,E_1 - w_3\,E_3$ is also a minimizer of $\mcF[\cdot; B_R]$ in the same set of competitors.
\end{proposition}

\begin{proof} Write $Q(x) = \sum_{i=0}^4 w_i(r)\,E_i$ as in Proposition \ref{Prop:QWRSRep} and let $\bw = (w_0, \ldots, w_4)$. We will only consider the case  where $w_1(R) \geq 0$ and $w_3(R) \geq 0$. (Note that the $O(2)$-symmetry of the boundary data implies that $w_2(R) = w_4(R) = 0$.) The other cases are treated similarly.

Observe that
\[
w_1(w_3^2 - w_4^2) + 2w_2 w_3 w_4 \leq \sqrt{w_1^2 + w_2^2} \sqrt{(w_3^2 -w_4^2)^2 + (2w_3 w_4)^2} = \sqrt{w_1^2 + w_2^2} (w_3^2 + w_4^2).
\]
Therefore, by \eqref{Eq:LdGW}, 
\[
I[\bw] \geq I[w_0, \sqrt{w_1^2 + w_2^2}, 0, \sqrt{w_3^2 + w_4^2}, 0].
\]
As $\bw$ is a minimizer for $I$, we have equality in the above inequalities, which leads to
\begin{align}
|\partial_r w_{j}|^2 + |\partial_r w_{j+1}|^2 &= \Big|\partial_r \sqrt{w_{j}^2 + w_{j+1}^2}\Big|^2 \text{ for } j \in \{1,3\},
	\label{Eq:Sat1}\\
(w_1,w_2) \text{ and } (w_3^2 - w_4^2, 2w_3w_4) &\text{ are colinear with a non-negative colinear factor}.
	\label{Eq:Sat2}
\end{align}

From \eqref{Eq:Sat1}, we deduce that there exist constant unit vectors $(\cos \alpha,\sin\alpha)$ and $(\cos\beta,\sin\beta)$, $\alpha,\beta \in [0,2\pi)$, and scalar functions $\lambda$ and $\mu$ such that
\[
(w_1,w_2) = \lambda(r)(\cos \alpha,\sin\alpha) \text{ and } (w_3,w_4) = \mu(r)(\cos \beta,\sin\beta).
\]
We recall that boundary conditions are $k$-fold $O(2)$-symmetric and therefore we have $w_2(R)=w_4(R)=0$.

\medskip
\noindent\underline{Case 1:} $w_1(R) > 0$ and $w_3(R) > 0$. In this case, we have $\lambda(R) \neq 0$ and $\mu(R) \neq 0$, which implies that $\sin\alpha = \sin\beta = 0$, and hence $w_2 \equiv w_4 \equiv 0$. 

\medskip
\noindent\underline{Case 2:} $w_1(R) > 0$ and $w_3(R) = 0$.\footnote{For example, \eqref{BC1} falls into this case.} This implies that $\lambda(R) \neq 0$ which leads to $\sin\alpha = 0$, and $w_2 \equiv 0$. We need to show that $w_4 \equiv 0$. Since $w_2 = 0$, we have $I[\bw] = I[w_0, w_1, 0, w_3, |w_4|]$ and so $(w_0, w_1, 0, w_3, |w_4|)$ is $I$-minimizing. Thus, we may assume without loss of generality that $w_4 \geq 0$. As $w_2 \equiv 0$, equation \eqref{Eq:w4} reduces to
\[
w_4'' + \frac{1}{r}w_4' - \frac{k^2}{4r^2}\,w_4 
	= w_4\Big(-a^2 + c^2 |\bw|^2 - \frac{b^2}{\sqrt{6}}w_0 + \frac{b^2}{\sqrt{2}}w_1\Big).
\]
By the strong maximum principle we thus have either $w_4 \equiv 0$ or $w_4 > 0$ in $(0,R)$. Assume by contradiction that $w_4 > 0$ in $(0,R)$. As $w_1(R) > 0$, there is some $R' < R$ such that $w_1 > 0$ in $(R',R)$. By \eqref{Eq:Sat2}, we have $w_1 w_3 w_4 \equiv 0$, and so $w_3 \equiv 0$ in $(R',R)$. But this implies that $(w_1,0)$ is not positively colinear to $(-w_4^2,0)$ in $(R',R)$ which contradicts \eqref{Eq:Sat2}.

\medskip
\noindent\underline{Case 3:} $w_1(R) = 0$ and $w_3(R) > 0$. This implies that $\mu(R) \neq 0$, $\sin\beta = 0$, and $w_4 \equiv 0$. We need to show that $w_2 \equiv 0$. By \eqref{Eq:Sat2}, we have $w_2 w_3 \equiv 0$. As $w_4 \equiv 0$, we have $I[\bw] = I[w_0, w_1, |w_2|, w_3, 0]$ and so $(w_0, w_1, |w_2|, w_3, 0)$ is $I$-minimizing. Thus, we may assume without loss of generality that $w_2 \geq 0$. Also as $w_4 \equiv 0$, equation \eqref{Eq:w2} reduces to
\[
w_2'' + \frac{1}{r}w_2' - \frac{k^2}{r^2}\,w_2
	= w_2\Big(-a^2 + c^2 |\bw|^2 + \frac{2b^2}{\sqrt{6}}w_0\Big),
\]
which implies, in view of the strong maximum principle, that $w_2 \equiv 0$ or $w_2 > 0$ in $(0,R)$. If the latter holds, then as $w_2 w_3 \equiv 0$, we would have $w_3 \equiv 0$, which would contradict the fact that $w_3(R) > 0$. We thus have that $w_2 \equiv 0$.

\medskip
\noindent\underline{Case 4:} $w_1(R) = 0$ and $w_3(R) = 0$. We have
\begin{align*}
&2w_0^3 - 6w_0(w_1^2 + w_2^2) + 3w_0(w_3^2 + w_4^2) + 3\sqrt{3} w_1(w_3^2 - w_4^2) + 6\sqrt{3} w_2 w_3 w_4 \\
	&\qquad = 2w_0^3 - 6w_0 \lambda^2 + 3w_0 \mu^2 + 3\sqrt{3} \lambda \mu^2 \cos(\alpha - 2\beta)
	 \leq g(w_0,|\lambda|,\mu)
\end{align*}
where
\[
g(x,y,z)= 2x^3 - 6x y^2 + 3xz^2 + 3\sqrt{3} yz^2, \qquad (x,y,z) \in \RR^3.
\]
By Lemma \ref{Lem:gMin} in Appendix \ref{app:Calc}, we have $g(x,y,z) \leq 2(x^2 + y^2 + z^2)^{3/2}$. It thus follows that
\[
I[\bw] \geq I[|\bw|,0,0,0,0].
\]
Since $\bw$ is $I$-minimizing, we hence have $I[\bw] = I[|\bw|,0,0,0,0]$, which implies $w_1 \equiv w_2 \equiv w_3 \equiv w_4 \equiv 0$.
\end{proof}

In Table \ref{Table2}, we summarize the characterization of various symmetries that we introduced for maps in $H^1(B_R, \mcS_0)$ (in particular, critical points or minimizers of $\mcF$) in terms of components $w_0, \dots, w_4$ and $k\neq 0$ even.
\begin{table}[ht]\label{Table2}
\caption{Characterization of symmetries in the components $w_0, \dots, w_4$ and $k\in 2\ZZ\setminus\{0\}$} 
\vskip 0.2cm
\centering 
\begin{tabular}{cc}
\hline
\hline                        
Symmetries in $H^1(B_R, \mcS_0)$ & Radial components \\  [0.5ex]
\hline
 $SO(2)$-symmetric map & $w_0, \dots, w_4$ \\ 
 $O(2)$-symmetric map & $w_0, w_1, w_3$ \\ 
 $\Z_2 \times SO(2)$-symmetric (i.e., $k$-radially symmetric) map & $w_0, w_1, w_2$ \\
   $\Z_2\times O(2)$-symmetric map & $w_0, w_1$ \\ 
 $SO(2)$-symmetric minimizer & $w_0, w_1, w_3$ \\ 
 $\Z_2 \times SO(2)$-symmetric ($k$-radially symmetric) critical point & 
 $w_0, w_1$ \\  [1ex]
\hline
\end{tabular}
\end{table}

\section{Minimizers with $k$-fold $O(2)$-symmetry on large disks}\label{Sec:MinSym}
In this section, we provide the proof of Theorem \ref{thm:MinSymmetry}. 
Instead of working directly with the functional $\mcF[\cdot;B_R]$ defined in \eqref{Eq:LdG} we rescale the domain $B_R$ to the unit disc $D \equiv B_1$. We work with a new parameter $\eps=\frac{1}{R}$ and the following rescaled Landau-de Gennes energy functional 
\be\label{def:mcF}
\mcF_\eps[Q] := \int_{D} \Big[\frac{1}{2}|\nabla Q|^2 + \frac{1}{\eps^2} f_{\rm bulk}(Q)\Big]\,dx, 
\ee
defined on the set
\[
H_{Q_b}^1(D,\mcS_0) = \Big\{Q \in H^1(D,\mcS_0): Q = Q_b \text{ on } \partial D\Big\}
\]
with $Q_b$ given by \eqref{Qbdef}. Throughout the section $k$ is an even non-zero integer.

The Euler-Lagrange equation for $\mcF_\eps$ reads
\begin{equation}
\eps^2 \Delta Q = -a^2 Q - b^2[Q^2 - \frac{1}{3}\tr(Q^2)I_3] + c^2 \tr(Q^2)Q \text{ in } D.
	\label{Eq:1.6res}
\end{equation}

The statement on the uniqueness up to $\Z_2$-conjugation for minimizers of $\mcF[\cdot;B_R]$ in Theorem \ref{thm:MinSymmetry} is equivalent to the following.

\begin{theorem} \label{thm:MinSymmetryRes}
Let $a^2 \geq 0, b^2, c^2 >0$ be any fixed constants and $k \in 2\ZZ \setminus \{0\}$. There exists some $\eps_0 = \eps_0(a^2, b^2, c^2, k) > 0$ such that for all $\eps \in (0,\eps_0)$, there exist exactly two minimizers $Q_\eps^\pm$ of $\mcF_\eps$ in $H_{Q_b}^1(D,\mcS_0)$ and these minimizers are $k$-fold $O(2)$-symmetric but not $\Z_2$-symmetric. Moreover, they are $\Z_2$-conjugate, namely, $Q_\eps^\pm=JQ_\eps^\mp J\neq Q_\eps^\mp$ with $J$ as defined in \eqref{def:J}. 
\end{theorem}

\subsection{Towards the proof of Theorem \ref{thm:MinSymmetryRes}}\label{SSec:Heu}

Using standard arguments it is straightforward to show that as $\eps \to 0$ the minimizers of $\mcF_\eps$ converge, along subsequences, in $H_{Q_b}^1(D,\mcS_0)$ to the minimizers of the harmonic map problem
\begin{equation}\label{HMPQ}
\mcF_*[Q] = \int_D \frac{1}{2}|\nabla Q|^2\,dx, \quad Q \in H_{Q_b}^1(D,\mcS_*), 
\end{equation}
where
\[
H_{Q_b}^1(D,\mcS_*) = \Big\{Q \in H^1(D,\mcS_*): Q = s_+(n \otimes n - \frac{1}{3} I_3) \text{ on } \partial D\Big\}
\]
and $\mcS_*$ defined in \eqref{def:limmanifold} is the set of global minimizers of $f_{\rm bulk}(Q)$, see e.g. \cite{BBH, Ma-Za}.

Due to the explicit form of $Q_b$ and the fact that $k$ is even, the minimizers of $\mcF_*$ in $H_{Q_b}^1(D,\mcS_*) $ can be written in the form $s_+(n_* \otimes n_* - \frac{1}{3}I_3)$ (see \cite{BallZar}), where $n_*$ minimizes the problem 
\begin{equation}
\mcF_{OF}[v] = \int_D \frac{1}{2}|\nabla v|^2\,dx, \qquad v \in H^1_{\nbdry}(D, \Sphere^2) = \{v \in H^1(D,\Sphere^2): v = \nbdry \text{ on } \partial D\}.
	\label{Eq:OFMinP}
\end{equation}
It is well known, see e.g.  \cite[Lemma A.2]{BrezisCoron83}, that minimizers of \eqref{Eq:OFMinP} are conformal and have images in either the upper or lower hemisphere. The compositions of these minimizers with the stereographic projections of the upper and lower hemispheres of $\Sphere^2$ onto the unit disk are the complex maps $z \mapsto z^{k/2}$ or $z \mapsto \bar z^{k/2}$, respectively. Therefore $\mcF_{OF}$ has exactly two minimizers in $H^1_n(D,\Ss^2)$ which are given by
\begin{equation}
n_*^\pm(r\cos\varphi, r\sin\varphi) =\left(\frac{2r^{\frac{k}{2}}\cos (\frac{k}{2}\varphi)}{1+r^k}, \frac{2r^{\frac{k}{2}}\sin(\frac{k}{2}\varphi)}{1+r^k}, \pm\frac{1-r^k}{1+r^k}\right).\label{Eq:n*Def}
\end{equation}
The corresponding minimizers of $\mcF_*$ are 
\begin{equation}
Q_*^\pm = s_+ (n_*^\pm \otimes n_*^\pm - \frac{1}{3}I_3).
	\label{Eq:Q*Def}
\end{equation}
We note that $Q_*^\pm$ are smooth and $O(2)$-symmetric but not $\Z_2$-symmetric and we can explicitly write $Q_*^\pm$ in terms of the basis tensors $\{E_i\}$ (see \eqref{basisS}) as
$$
Q_*^\pm = w_0^*(r) E_0 + w_1^*(r) E_1 \pm w_3^*(r)E_3,
$$
where
$$
w_0^*(r)= s_+\frac{2(1-r^k)^2-4r^k}{\sqrt{6}(1+r^k)^2}, \ w_1^*(r)= \frac{4s_+r^k}{\sqrt{2}(1+r^k)^2}, \ w_3^* (r) =  \frac{4s_+r^\frac{k}{2} (1-r^k)}{\sqrt{2}(1+r^k)^2}.
$$
It is possible to show that from any sequence of minimizers $Q_{\eps_k}$ of $\mcF_{\eps_k}$ in $H_{Q_b}^1(D,\mcS_0)$ one can extract a subsequence which converges in $C^{1,\alpha}(\bar D)$ and $C^j_{loc}(D)$ for any $j \geq 2$ to either $Q_*^+$ or $Q_*^-$ (the reasoning requires straightforward modifications of the arguments in \cite{BBH, NZ13-CVPDE}). Using the energy representation \eqref{Eq:LdGW} one can observe that if $Q_{\eps} = \sum_{i = 0}^4 w_{i,\eps} E_i$ is a minimizer of $\mcF_{\eps}$ in $H_{Q_b}^1(D,\mcS_0)$, then the $\Z_2$-conjugate $\tilde Q_\eps = JQ_\eps J = \sum_{i = 0}^2 w_{i,\eps} E_i - \sum_{i = 3}^4 w_{i,\eps} E_i$ of $Q_\eps$ is also a minimizer of $\mcF_{\eps}$ in $H_{Q_b}^1(D,\mcS_0)$. Also, if $Q_{\eps_k'} \rightarrow Q_*^+$, then $\tilde Q_{\eps_k'} \rightarrow Q_*^- = JQ_*^+ J$ and vice versa. Thus both $Q_*^+$ and $Q_*^-$ can appear as limits of the sequences of minimizers  of $\mcF_\eps$.

Now, restrict $\mcF_\eps$ to the set of \okr tensors
\be\label{def:Ars}
\mcA^{rs}= \Big\{Q \in H_{Q_b}^1(D,\mcS_0): Q \text{ is $O(2)$-symmetric}\Big\}.
\ee
Arguing as before but restricting $\mcF_\eps$ to $\mcA^{rs}$ it is straightforward to show that any sequence of minimizers $Q_{\eps_k}^{rs}$ of $\mcF_{\eps_k}$ in $\mcA^{rs}$ has a subsequence which converges in $C^{1,\alpha}(\bar D)$ and $C^j_{loc}(D)$ for any $j \geq 2$ to a minimizer of $\mcF_*$ in the set of $O(2)$-symmetric tensors in $H_{Q_b}^1(D,\mcS_*)$, which clearly must be either $Q_*^+$ or $Q_*^-$ as these are $O(2)$-symmetric.

Based on the above we can construct two sequences of critical points of $\mcF_\eps$ converging to $Q_*^+$ (or $Q_*^-$): one consisting of minimizers of $\mcF_\eps$ in $H_{Q_b}^1(D,\mcS_0)$, not as yet guaranteeing $k$-fold $O(2)$-symmetry, and another consisting of minimizers of $\mcF_\eps\big|_{\mcA^{rs}}$, which are $k$-fold $O(2)$-symmetric. Therefore, to prove Theorem \ref{thm:MinSymmetryRes} we need to show that if $\eps$ is small enough these sequences coincide, in particular, the minimizers of $\mcF_\eps$ are $O(2)$-symmetric. One possible approach is to employ the contraction mapping theorem or the implicit function theorem to show that 
\[
\parbox{.8\textwidth}{there are ``neighborhoods'' $\mcN^\pm$ of $Q_*^\pm$ such that when $\eps$ is small enough $\mcF_\eps$ admits at most one critical point in each of $\mcN^\pm$.}
\tag{$\dagger$}
\label{dagger}
\]

In this approach typically the neighborhoods $\mcN^\pm$ are set up in relatively stronger norms than the energy-associated norm. In addition, the norm and thus the neighborhood are dependent on $\eps$. A delicate point is the competition between the size of the neighborhood where one can prove uniqueness and the rate of convergence of minimizing sequences to the limit $Q_*^\pm$ (so that one can squeeze all minimizers into the designed neighborhood).

Below we present a roadmap to the proof of Theorem~\ref{thm:MinSymmetryRes}. Since $Q_*^\pm$ are equivalent up to a $\Z_2$-conjugation, it suffices to construct one such neighborhood, say $\mcN^+$ of $Q_*^+$. For simplicity, we will in the sequel drop the superindex $+$, so that $Q_* = Q_*^+$, $n_*=n_*^+$, etc.

In  Subsection~\ref{ssec:BasisEl}, we will provide a parameterization of suitable neighbourhoods $\mcN^\pm$ where we have a decomposition 
\[
Q = Q_\sharp + \eps^2 P = s_+\Big(\frac{n_* + \psi}{|n_* + \psi|} \otimes \frac{n_* + \psi}{|n_* + \psi|} - \frac{1}{3}I_3\Big) + \eps^2 P
\] with $\eps^2 P$ being  a ``transversal component'' of $Q$ and $Q_\sharp$ being a (non-orthogonal) ``projection'' onto the limit manifold $\mcS_*$. 

In Subsection~\ref{SSec:ELpsiP} we employ the above parameterization to obtain a new representation of the Euler-Lagrange equations \eqref{Eq:1.6res} in terms of the variables $\psi$ and $P$. In particular, we will derive a coupled  system of equations for $\psi$ and $P$ with the following properties:

\begin{enumerate}
\item One equation is of the form
\begin{equation}
L_{\parallel}\psi= \text{Lagrange multiplier terms} + F[\eps,\psi,P],
	\label{Eq:Lparal}
\end{equation} where the operator $L_{\parallel} = -\Delta - |\nabla n_*|^2$ is the linearized harmonic map operator at the minimizer $n_*$ of the problem \eqref{Eq:OFMinP}. See \eqref{Eq:psiEq} for the exact equation.

\item The other equation is of the form
\be \label{eq:Lepsperp}
L_{\eps,\perp} P = \text{Lagrange multiplier terms} +  s_+ \Delta(n_* \otimes n_*) + G[\eps,\psi,P],
\ee  with the linear operator $L_{\eps,\perp} P = -\eps^2\Delta P + b^2\,s_+\,P + 2 (c^2\,s_+ - b^2)\,(Pn_* \cdot n_*) Q_*$. See \eqref{Eq:PEq} for the exact form.
\end{enumerate}
We will see that, although the nonlinear operators $F$ and $G$ are second ordered in the fields $\psi$ and $P$, they are `super-linear' and are `small' when $\psi$ and $P$ are suitably `small'.
Subsection~\ref{ssec:LinHM} is devoted to study the operator $L_{\parallel}$ and in Subsection~\ref{ssec:LPInverse} we concentrate on the operator $L_{\eps,\perp}$.

In Subsection~\ref{ssec:psiP}, using previously derived properties of $L_\parallel$, we revisit \eqref{Eq:Lparal} and study the dependence of its  solution $\psi_\eps$ (with zero Dirichlet boundary condition) as a map of $P$. In order to balance the rate of convergence of the sequences of minimizers and the size of the neighborhoods $\mcN^\pm$ it will be convenient to measure the size of $P$ with respect to an $\eps$-dependent $H^2$ norm, specifically defined as:
\begin{equation}
\|P\|_{\eps} := \|P\|_{L^2(D)} + \eps\|\nabla P\|_{L^2(D)} + \eps^2\|\nabla^2 P\|_{L^2(D)}.
	\label{Eq:H2epsNorm}
\end{equation}

We first show that, for $P\in H^1_0\cap H^2(D, \mcS_0)$ with $\|P\|_{L^2(D)} = O(1)$ and $\|\nabla^2 P\|_{L^2(D)} = o(\eps^{-2})$, one can solve \eqref{Eq:Lparal} for $\psi_\eps = \psi_\eps(P) \in H^1_0\cap H^2(D,\RR^3)$. Furthermore, when $P$ is measured with respect to the norm $\|\cdot\|_\eps$ above, $\psi_\eps$ is Lipschitz with respect to $P$ with Lipschitz constant $O(\eps)$ (see Proposition \ref{Prop:psiofP}). Using the Lipschitz estimate above, we show that the map $P \mapsto L_{\eps,\perp}^{-1} G[\eps,\psi_\eps(P),P]$ is contractive. This proves the uniqueness statement formulated informally above in relation \eqref{dagger}. See Proposition \ref{Prop:ResUniq} in Subsection \ref{ssec:unique}.
In Subsection~\ref{SSec:ProofUniq}, using convergence results from \cite{NZ13-CVPDE} and the results presented above we prove Theorem~\ref{thm:MinSymmetryRes}. Finally, the proof of Theorem \ref{thm:MinSymmetry} is done in Subsection \ref{SSec:ProofUniqThm1.5}.

\subsection{A parametrization in small $H^2$-neighborhoods of $Q_*$}\label{ssec:BasisEl}

In this section we show that every $Q \in H^1_{Q_b}(D,\mcS_0) \cap H^2(D,\mcS_0)$ sufficiently close to a minimizer  $Q_*$ of $\mcF_*$ can be decomposed in a special way (see \eqref{eq:Qdecomp}) that takes into account the geometry of the limit manifold $\mcS_*$ and the way it embeds into the space $\mcS_0$ of $Q$-tensors.

Since $\mcS_*$ is a smooth compact submanifold of $\mcS_0$, we can find a neighborhood $N(\mcS_*)$ of $\mcS_*$ in $\mcS_0$, such that for every $B \in N(\mcS_*)$, there exists a unique $B_{ort} \in \mcS_*$ such that $|B - B_{ort}| = \dist(B, \mcS_*)$ where $|\cdot|$ stands for the norm associated to the Frobenius scalar product. Furthermore the projection $B \mapsto B_{ort}$ is a smooth map from $N(\mcS_*)$ onto $\mcS_*$. See Figure \ref{Fig1}.

Although the above orthogonal projection suffices for many purposes, it is somewhat more convenient in our current setting to work with a different projection which is more adapted to $Q_*$.  Let $n_* = n_*^+$ be as in \eqref{Eq:n*Def}  and $Q_* = Q_*^+ = s_+(n_* \otimes n_* - \frac{1}{3}I_3)$. We denote by $T_{Q_*}\mcS_* = T_{Q_*(x)}\mcS_*$ the tangent space to the {\it limit manifold} $\mcS_*$ at $Q_*(x)$ and by $(T_{Q_*}\mcS_*)^\perp$ its orthogonal complement in $\mcS_0 \approx \RR^5$, which is normal to $\mcS_*$ at $Q_*(x)$. It is known that $(T_{Q_*}\mcS_*)^\perp$ consists of all matrices in $\mcS_0$ commuting with $Q_*$; see \cite[Eq. (3.2)]{NZ13-CVPDE}. In particular, all matrices in $(T_{Q_*}\mcS_*)^\perp$ admits $n_*$ as an eigenvector. \footnote{Recall that the eigenspace of $Q_*$ corresponding to the eigenvalue $\lambda_*=\frac23s_+$ is of dimension one and generated by $n_*$; therefore, if $AQ_*=Q_*A$, then $An_*$ belongs to this eigenspace, i.e., $An_*$ is parallel to $n_*$.}

We want to show that every $Q$ in a ``sufficiently small neighborhood" of $Q_*$ decomposes as
\be\label{eq:Qdecomp}
Q(x):=\underbrace{s_+\left(v(x)\otimes v(x)-\frac{1}{3}I_3\right)}_{\text{belongs to $\mcS_*$}}+ \text{ part transversal to $\mcS_*$}
\ee 
so that $Q(x)-s_+\left(v(x)\otimes v(x)-\frac{1}{3} I_3\right) \in (T_{Q_*}\mcS_*)^\perp$, which will be useful later. See Figure \ref{Fig1}. We specify the result in the following lemma.

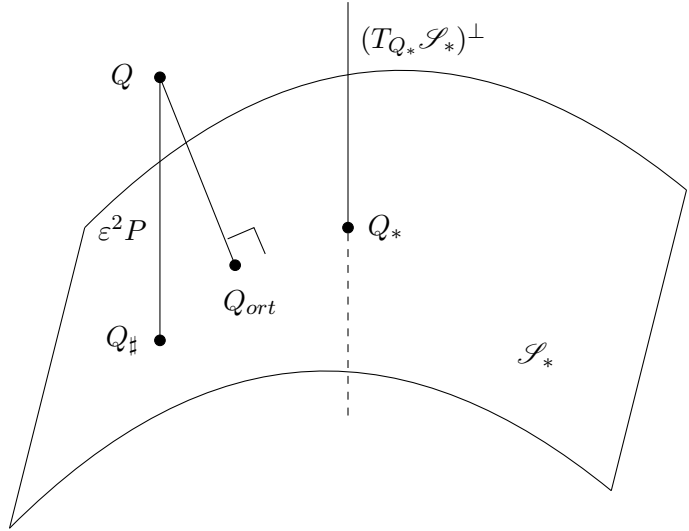
\begin{figure}[h]
\begin{center}
\begin{tikzpicture}
\draw (0,0) .. controls (2.7,2.7) and (5.3,2.7) .. (8,0.5);
\draw (1,4) .. controls (3.7,6.7) and (6.3,6.7) .. (9,4.5);
\draw (0,0) -- (1,4);
\draw (8,0.5) -- (9,4.5);

\draw (4.5,4) -- (4.5,7);
\draw[dashed] (4.5,1.5)--(4.5,4);
\filldraw[black] (4.5,4) circle [radius = 2pt];
\draw (2,2.5)--(2,6);
\filldraw[black] (2,2.5) circle [radius = 2pt];
\filldraw[black] (2,6) circle [radius = 2pt];
\draw(2,6)--(3,3.5);
\filldraw[black] (3,3.5) circle [radius = 2pt];
\draw (2.9,3.85)--(3.25,4)--(3.4,3.65);

\draw (5.5,6.5) node {$(T_{Q_*}\mcS_*)^\perp$}
(7,2.3) node {$\mcS_*$}
(5,4) node {$Q_*$}
(1.5,4) node {$\eps^2 P$}
(1.5,2.5) node {$Q_\sharp$}
(1.5,6) node {$Q$}
(3.2,3) node {$Q_{ort}$};
\end{tikzpicture}
\end{center}
\caption{The decomposition in Lemma \ref{lem:TubNbhd} vs orthogonal projection.}
\label{Fig1}
\end{figure}

\begin{lemma}\label{lem:TubNbhd}
Let $n_* = n_*^+$ and $Q_* = Q_*^+$ be as in \eqref{Eq:n*Def} and  \eqref{Eq:Q*Def}. There exist $\gamma > 0$ and some large $C_0>0$ such that for every $\eps > 0$ and every $Q \in H^1_{Q_b}(D,\mcS_0) \cap H^2(D,\mcS_0)$ with $\|Q - Q_*\|_{H^2(D)} \leq \frac{1}{C_0}$ we can uniquely write 
\begin{equation}
Q = Q_\sharp + \eps^2 P
	\label{Eq:QQsPDecomp}
\end{equation}
where $Q_\sharp$ and $P$ satisfy
\begin{itemize}
\item $Q_\sharp \in H^1_{Q_b}(D,\mcS_*) \cap H^2(D, \mcS_*)$,
\item $P\in H^1_0(D, \mcS_0) \cap H^2(D,\mcS_0)$ with $P(x)\in  (T_{Q_*}\mcS_*)^\perp$ for $x\in D$,
\item $\|Q_\sharp - Q_*\|_{H^2(D)} \leq \gamma\|Q - Q_*\|_{H^2(D)}$,
\item and $\eps^2\|P\|_{H^2(D)} \leq \gamma\,\|Q - Q_*\|_{H^2(D)}$.
\end{itemize}
Furthermore, there exists a unique $\psi \in H^1_0(D,\RR^3) \cap H^2(D,\RR^3)$ with $\psi \cdot n_* = 0$ a.e. in $D$ such that $\|\psi\|_{H^2(D)} \leq \gamma\|Q - Q_*\|_{H^2(D)}$ and
\begin{equation}
Q_\sharp = s_+\Big(\frac{n_* + \psi}{|n_* + \psi|} \otimes \frac{n_* + \psi}{|n_* + \psi|} - \frac{1}{3}I_3\Big).
	\label{QsharpX}
\end{equation}
\end{lemma}

\begin{remark}
In the above lemma, we have deliberately written the ``transversal'' component of $Q$ as $\eps^2 P$ even though $\eps$ plays no role at the moment. In \cite{NZ13-CVPDE}, it is shown that, in a similar setting, if $Q_\eps$ is a minimizer for $\mcF_\eps$, then its ``transversal'' contribution is of size $\eps^2$ in some appropriate topology. In the setting of the present paper, we will show this holds in the $L^2(D,\mcS_0)$-topology; see \eqref{bddPeps} below. This rate of convergence however does not hold in the $H^2(D,\mcS_0)$-topology.\footnote{For such rate of convergence would imply in view of Proposition \ref{Prop:psiofP} below that $\|Q_\eps - Q_*\|_{H^2(D,\mcS_0)} = O(\eps^2)$, which would further imply that the limit of $\eps^{-2} (Q_\eps - Q_*)$ has zero trace on $\partial D$, which would contradict \cite[Theorem 2, Eq. (2.14)]{NZ13-CVPDE}. See also \cite{BBH} for a similar statement in the Ginzburg-Landau setting.} This is related to the comment we made earlier on the fact that the sets $\mcN^\pm$ in \eqref{dagger} are $\eps$-dependent. 
\end{remark}

\begin{remark}
It should be noted that the map $\psi$ appearing in the representation of $Q_\sharp$ belongs to a linear space (as $\psi$ is orthogonal to $n_*$) as opposed to $Q_\sharp$ that belongs to a nonlinear set (as its values being constrained in $\mcS_*$).
\end{remark}

\begin{proof} Since $\mcS_*$ is a smooth submanifold of $\mcS_0$, there exists for every point $B_* \in \mcS_*$ a neighborhood $U_{B_*}$ of $B_*$ in $\mcS_0$ such that $\mcS_* \cap U_{B_*}$ is a graph over the tangent plane $T_{B_*}\mcS_*$. We then select local Cartesian-type coordinates $\{x_1, \ldots, x_5\}$ of $\mcS_0 \approx \RR^5$ such that $B_*$ corresponds to the origin, $T_{B_*}\mcS_*$ coincides with $\{(x_1,x_2,0,0,0): x_1,x_2\in \RR\}$ and $\mcS_* \cap U_{B_*}$ is given by $\{(x_1,x_2,u_1(x_1,x_2),u_2(x_1,x_2),u_3(x_1,x_2)): (x_1,x_2)\in\widetilde{\mathcal{U}}\}$ for  some open set $\widetilde{\mathcal{U}}\subset\RR^2$ and some smooth function $u=(u_1,u_2,u_3):\widetilde{\mathcal{U}}\to\RR^3$ with $u(0) = 0$ and $\nabla u(0) = 0$. Define a projection $\mcP_{B_*}$ from $U_{B_*}$ to $\mcS_*$ by
\[
\mcP_{B_*}(x_1, x_2, x_3, x_4, x_5) = (x_1,x_2,u(x_1,x_2)).
\]
One can check that $\mcP_{B_*}(B)$ is well-defined (i.e. independent of local charts) and smooth as a function of two variables $B\in\mcS_0$ and $B_*\in\mcS_*$. Furthermore, $\mcP_{B_*}$ is the unique projection with the property $B - \mcP_{B_*}(B) \in (T_{B_*} \mcS_*)^\perp$.

As $D$ is two dimensional, maps in $H^2(D,\mcS_0)$ are continuous. Thus, there exists some large constant $C_0 > 0$ such that whenever  
\be\label{ass:closeQQstar}
\|Q - Q_*\|_{H^2(D)} \leq \frac{1}{C_0},
\ee there holds $Q(x) \in U_{Q_*(x)}$ for all $ x \in D$. The decomposition \eqref{Eq:QQsPDecomp} is achieved by
\[
Q_\sharp(x) = \mcP_{Q_*(x)}(Q(x)) \text{ and } P(x)= \eps^{-2}(Q(x) - Q_\sharp(x)) \text{ for } x \in D.
\]

We now proceed to check the desired properties of $Q_\sharp$ and $P$. First, we have
\[
Q_\sharp(x)-Q_*(x)=\mcP_{Q_*(x)}(Q(x))-\mcP_{Q_*(x)}(Q_*(x)).
\] 
Using the smoothness of $\mcP$ in both variables, we obtain the claimed control of $\|Q_\sharp-Q_*\|_{H^2(D)}$ in terms of $\|Q-Q_*\|_{H^2(D)}$. Furthermore we have $Q-Q_\sharp=Q-Q_*+Q_*-Q_\sharp$ which also provides the control of the $H^2$-norm of $\eps^2 P=Q-Q_\sharp$ in terms of the $H^2$-norm of $Q-Q_*$, as claimed.

We turn to the second part of the lemma. Note that $Q_\sharp \in H^2(D,\mcS_*)$ is continuous. As $D$ is simply connected and $\mcS_*$ can be topologically identified with a projective plane, a standard result in topology about 
covering spaces implies that there is a unique continuous function $v \in C^0(D, \Sphere^2)$ such that 
\[
Q_\sharp = s_+(v \otimes v - \frac{1}{3}I_3) \text{ in } D \text{ and } v = \nbdry \text{ on }\partial D.
\]
Furthermore, by \cite[Theorem 2]{BallZar}, we have $v \in W^{1,p}(D,\Sphere^2)$ for any $p \geq 2$. 

Note that $\nabla_k (Q_\sharp)_{ij} = s_+ (\nabla_k v_i\,v_j + \nabla_k v_j\,v_i)$, and so, as $|v| = 1$, 
\[
\nabla_k v_i = \frac{1}{s_+} \nabla_{k} (Q_\sharp)_{ij}\,v_j,
\]
from which one can easily deduce that $v \in H^2(D,\Sphere^2)$.

Next note that, as $Q_\sharp - Q_* = s_+(v \otimes v - n_* \otimes n_*)$, we have
\begin{equation} \label{eq:rel1}
|Q_\sharp - Q_*|^2 = 2s_+^2(1 - (v \cdot n_*)^2)
\quad \text{ and }\quad
(Q_\sharp - Q_*)n_* 
	= s_+[(v \cdot n_*) v - n_*] .
\end{equation}
Taking $C_0$ large enough in \eqref{ass:closeQQstar} and using equality \eqref{eq:rel1}, we obtain $(v\cdot n_*)^2\ge \frac{1}{4}$. Since both $v$ and $n_*$ are continuous and coincide at the boundary we deduce $v \cdot n_* \geq \frac12$. Therefore we can define
\[
\psi = \frac{1}{v \cdot n_*}\,v - n_*.
\]
Observe that the above is equivalent to $\big($ $\psi \cdot n_* =0$ and $v = \frac{n_* + \psi}{|n_* + \psi|}$ $\big)$, which gives the uniqueness of $\psi$. On the other hand, one has
$$
\psi={ (v \cdot n_*)v -n_* \over (v \cdot n_*)^2} + {n_* (1- (v \cdot n_*)^2) \over (v \cdot n_*)^2}.
$$
Recalling relations \eqref{eq:rel1} we can represent $\psi = G( n_*, Q_\sharp - Q_*)$, where $G$ is a smooth map provided that $|Q_\sharp - Q_*|$ is small. We hence obtain $\| \psi \|_{H^2(D)} \leq C \|Q_\sharp - Q_*\|_{H^2(D)} \leq \gamma\|Q - Q_*\|_{H^2(D)}$. This concludes the proof.
 \end{proof}

\subsection{The Euler-Lagrange equations}\label{SSec:ELpsiP}

In this subsection we rewrite the Euler-Lagrange equations \eqref{Eq:1.6res} for $\mcF_\eps$ in terms of the variables $\psi$ and $P$ introduced in Lemma \ref{lem:TubNbhd}. This new form of the  Euler-Lagrange equations will be used in the subsequent analysis. 

\begin{lemma}\label{lemma:EL}
Let  $Q \in  H^1_{Q_b} (D,\mcS_0) \cap H^2(D,\mcS_0)$ be a critical point of $\mcF_\eps$ for some $\eps > 0$, and $n_*$ and $Q_*$ be given by \eqref{Eq:n*Def} and \eqref{Eq:Q*Def}. Suppose that $\|Q - Q_*\|_{H^2(D)}$ is sufficiently small and let $Q_\sharp$, $P$ and $\psi$ be as in Lemma \ref{lem:TubNbhd}. Then $\psi$ and $P$ satisfy the following equations
\begin{eqnarray}
-\Delta \psi - |\nabla n_*|^2\,\psi
	&=&  \lambda_\eps \,n_* +  A[\psi] + \eps^2\,B_\eps[\psi,P] ,
		\label{Eq:psiEq}\\ 
-\eps^2 \Delta P + b^2\,s_+\,P + 2 (c^2\,s_+ - b^2)\,(Pn_* \cdot n_*) Q_* 
	&=&   F_\eps + s_+ \Delta( n_* \otimes n_* ) \nonumber\\
	&&\hspace{-0.5cm} + C_\eps[\psi,P]-\frac{1}{3}\tr(C_\eps[\psi,P])I_3 ,
		\label{Eq:PEq} \\
\psi  \cdot n_* =0, \quad  P &\in& (T_{Q_*}\mcS_*)^\perp \quad  \hbox{ in }  D,
		\label{Eq:psiPCons}
\end{eqnarray}
where $\lambda_\eps(x)$ is a Lagrange multiplier accounting for the constraint $\psi \cdot n_* = 0$, $F_\eps(x) \in T_{Q_*}\mcS_*$ is a Lagrange multiplier accounting for the constraint $P \in (T_{Q_*}\mcS_*)^\perp$, and maps $A, B_\eps, C_\eps$ are defined in equations \eqref{def:A}, \eqref{def:B} and \eqref{def:C} in Appendix \ref{App:Tech}.
\end{lemma}

In Lemma~\ref{lemma:EL} above, we do not provide exact form of $A, B_\eps, {C}_\eps$ nor indicate their explicit dependence on $x$ as we show later (see the proof of Lemma~\ref{lemma:EL}) that  these are lower order terms that do not play a role in our analysis. We will only use their properties summarized in the following proposition.

\begin{proposition}\label{prop:ABCest}
Let $\eps \in (0,1)$, $n_*$ and $Q_*$ be given by \eqref{Eq:n*Def} and \eqref{Eq:Q*Def}, and let $A,B_\eps$ and $C_\eps$ be the operators appearing in Lemma~\ref{lemma:EL}, defined in  \eqref{def:A},\eqref{def:B},\eqref{def:C} in Appendix \ref{App:Tech}. Then, for $\psi \in H^1_0(D,\RR^3) \cap H^2(D,\RR^3)$, $P\in H^1_0(D, \mcS_0) \cap H^2(D,\mcS_0)$ satisfying $\psi \cdot n_* = 0$ and $P \in (T_{Q_*}\mcS_*)^\perp$ in $D$, we have  the following:
\begin{align}
A[0]
	&= 0,
	\label{est:Aest0}\\
\|A[\psi] - A[\tilde\psi]\|_{L^2(D)}
	&\leq C(\|\psi\|_{H^2(D)}+\|\tilde\psi\|_{H^2(D)}) \nonumber\\
		&\qquad  \times (1+\|\psi\|_{H^2(D)}+\|\tilde\psi\|_{H^2(D)})\|\psi - \tilde\psi\|_{H^2(D)},
	\label{est:Aest}\\
\|B_\eps[0,P]\|_{L^2(D)}
	&\leq C\|P\|_{H^1(D)},
	\label{est:Best0}\\
\|B_\eps[\psi,P] - B_\eps[\tilde \psi, P]\|_{L^2(D)}
	&\leq C\Big[ \|\nabla^2 P\|_{L^2(D)} + (1 + \eps^2\|P\|_{H^2(D)}) \|P\|_{L^4(D)}^2 \Big]\|\psi-\tilde\psi\|_{H^2(D)},
	\label{est:Best1}
\end{align}
\vspace{-0.5cm}
\begin{multline}
\|B_\eps[\psi,P] - B_\eps[ \psi, \tilde P]\|_{L^2(D)}
	\leq C\|P-\tilde P\|_{H^1(D)}
		 + C\|\psi\|_{H^2(D)}\Big(\|P-\tilde P\|_{H^2(D)}\\
			+(\|P\|_{H^2(D)}+\|\tilde P\|_{H^2(D)}) (1 + \eps^2\|P\|_{H^2(D)} + \eps^2\|\tilde P \|_{H^2(D)})\|P-\tilde P\|_{L^2(D)}\Big),
\label{est:Best2}
\end{multline}
\vspace{-0.5cm}
\begin{align}
&\|C_\eps(\psi,P)-C_\eps(\tilde\psi,\tilde P)\|_{L^2(D)} 
	 \leq C(1 + \|\psi\|_{H^2(D)} + \|\tilde\psi\|_{H^2(D)})^2\|\psi - \tilde\psi\|_{H^2(D)}\nonumber\\
&\qquad\quad +C(\|P\|_{L^2(D)}+\|\tilde P\|_{L^2(D)}) (1 + \eps^2(\|P\|_{H^2(D)}+\|\tilde P\|_{H^2(D)})) \|\psi - \tilde\psi\|_{H^2(D)}\nonumber\\
&\qquad\quad+C(\|\psi\|_{H^2(D)} +\|\tilde\psi\|_{H^2(D)})  \|P - \tilde P\|_{L^2(D)}\nonumber\\
&\qquad\quad+ C\eps^2(\|P\|_{L^4(D)} + \|\tilde P\|_{L^4(D)}) (1 + \eps^2(\|P\|_{H^2(D)}+\|\tilde P\|_{H^2(D)}))	\|P - \tilde P\|_{H^1(D)} ,
\label{est:Cest}
\end{align} with $C$ denoting various constants independent of $\eps$ and the functions appearing in the inequalities.

\end{proposition}

The proofs of Lemma \ref{lemma:EL} and Proposition \ref{prop:ABCest} are lengthy though elementary. We postpone them to Appendix \ref{App:Tech}.

\subsection{The linearized harmonic map problem}\label{ssec:LinHM}

In this subsection we briefly study the properties of the operator $L_{\parallel} = -\Delta - |\nabla n_*|^2$ appearing on the left hand side of \eqref{Eq:psiEq}, i.e. the linearized harmonic map operator at $n_*$ given by \eqref{Eq:n*Def}, as well as its inverse $L_\parallel^{-1}$.

\begin{proposition}\label{cor:LInverse}
For every $f \in L^2(D,\RR^3)$, the minimization problem
\[
\min \Big\{\int_D [|\nabla \zeta|^2 - |\nabla n_*|^2\,\zeta^2 - f\cdot \zeta]\,dx: \zeta \in H_0^1(D,\RR^3), \zeta \cdot n_* = 0 \text{ a.e. in } D\Big\}
\]
admits a minimizer which is the unique solution to the problem
\begin{equation}
\left\{\begin{array}{rcll}
L_\parallel \zeta \equiv -\Delta \zeta - |\nabla n_*|^2\zeta  &=& f + \lambda(x)\,n_* &\text{ in }D,\\
\zeta \cdot n_*	&=& 0 & \text{ in } D,\\
\zeta &=& 0 &\text{ on } \partial D,
\end{array}\right.
\label{Eq:LBVP}
\end{equation}
where $\lambda$ is a Lagrange multiplier.\footnote{The expression of the Lagrange multiplier $\lambda$ is given in \eqref{pg26}.}
\end{proposition}
Using Proposition~\ref{cor:LInverse} we can define the inverse operator $L_\parallel^{-1}$.
\begin{definition}\label{def:Lpar}
For $f \in L^2(D,\RR^3)$, we define  $L_{\parallel}^{-1}f \in H^1_0(D,\RR^3)$ to be the unique solution to \eqref{Eq:LBVP}. 
\end{definition}

The proof of Proposition~\ref{cor:LInverse} is a standard argument using Lax-Milgram's theorem and the strict stability of $n_*$. For completeness, we give the proof in Appendix \ref{App:LpaInv}.

In the following lemma we prove some useful properties of $L_\parallel^{-1}$ required for our analysis.

\begin{lemma} \label{Lem:LINorm}
The range of the operator $L_{\parallel}^{-1}$ over $L^2$-data is
\begin{equation}
X  := \{\psi \in H^1_0(D,\RR^3) \cap H^2(D,\RR^3): \psi \cdot n_* = 0 \text{ in } D\}.
	\label{Eq:XSpaceDef}
\end{equation}
Furthermore, there exists some positive constant $C$ such that, for $f \in L^2(D, \RR^3)$,
\[
\|L_{\parallel}^{-1}f\|_{H^1(D)} \leq C\|f\|_{H^{-1}(D)} \text{ and }  \|L_{\parallel}^{-1}f\|_{H^2(D)} \leq C\|f\|_{L^2(D)}.
\]
\end{lemma}

\begin{proof} Let $f \in L^2(D,\RR^3)$ and $\zeta \in H_0^1(D,\RR^3)$ be the solution of \eqref{Eq:LBVP}. We first show that $\zeta \in H^2(D,\RR^3)$. 
Let us fix some $\xi \in C_c^\infty(D)$. Testing \eqref{Eq:LBVP} against $\xi\,n_*$ and noting that $n_* \cdot \zeta = 0 = \Delta n_* \cdot \zeta $, we obtain by integration by parts
\begin{align*}
\int_D \xi(f \cdot n_* + \lambda)\,dx
	&= \int_D \nabla \zeta \cdot \nabla (\xi n_*) \,dx= -\int_D \zeta \cdot \Delta (\xi n_*) \,dx\\
	&= -2 \int_D \zeta \cdot (\nabla n_*(\nabla\xi)) \,dx= 2\int_D \xi \nabla \zeta \cdot \nabla n_*\,dx.
\end{align*}
Since this is true for all $\xi \in C_c^\infty(D)$, it follows that
\be
\label{pg26}
\lambda = 2 \nabla \zeta \cdot \nabla n_*  - f \cdot n_* \in L^2(D).
\ee
By elliptic regularity for \eqref{Eq:LBVP}, we conclude that $\zeta \in H^2(D,\RR^3)$.

We next turn to estimating $\zeta$. Testing \eqref{Eq:LBVP} against $\zeta$ and using Lemma \ref{Lem:StrStab}, we obtain
\[
c_0\|\zeta\|_{H^1(D)}^2 \leq \int_D [|\nabla \zeta|^2 - |\nabla n_*|^2\,|\zeta|^2]\,dx = \int_D f \cdot \zeta\,dx \leq \|f\|_{H^{-1}(D)}\,\|\zeta\|_{H^1(D)}.
\]
This implies $\|L_{\parallel}^{-1}f\|_{H^1(D)} = \|\zeta\|_{H^1(D)} \leq C\,\|f\|_{H^{-1}(D)}$. Using this estimate in \eqref{pg26} we have $\|\lambda\|_{H^{-1}(D)} \leq C\,\|f\|_{H^{-1}(D)}$ and $\|\lambda\|_{L^2(D)} \leq C\|f\|_{L^2(D)}$. Employing elliptic estimates for \eqref{Eq:LBVP} we obtain  $\|\zeta\|_{H^2(D)} \leq C\|f\|_{L^2(D)}$.
\end{proof}

\subsection{The transversal linearized problem}\label{ssec:LPInverse}

In this section we study the linear operator appearing on the left hand side of \eqref{Eq:PEq}
\be \label{eq:Lperp}
L_{\eps,\perp} P = -\eps^2\Delta P + b^2\,s_+\,P + 2 (c^2\,s_+ - b^2)\,(Pn_* \cdot n_*) Q_*.
\ee  
As in the previous subsection we would like to define the inverse operator $L_{\eps,\perp}^{-1}$ and prove some properties required for our analysis.

We claim that
$P \mapsto b^2\,s_+\,P + 2 (c^2\,s_+ - b^2)\,(Pn_* \cdot n_*) Q_*$ is a monotone linear operator, namely
\be \label{eq:mon}
b^2\,s_+\,|P|^2 + 2s_+ (c^2\,s_+ - b^2)\,(Pn_* \cdot n_*)^2 \geq \min(2a^2 + \frac{b^2}{3}s_+, b^2s_+)|P|^2,  \forall P(x) \in (T_{Q_*}\mcS_*)^\perp.
\ee
Indeed, recall that $n_*$ is an eigenvector of $ P\in (T_{Q_*}\mcS_*)^\perp$. Thus, in some orthonormal basis of $\RR^3$, $P$ takes the form $\textrm{diag}(\lambda_1, \lambda_2, -\lambda_1 - \lambda_2)$ with $Pn_* = \lambda_1 n_*$. It is not hard to see that this implies $(Pn_* \cdot n_*)^2 = \lambda_1^2 \leq \frac{4}{3} (\lambda_1^2 + \lambda_2^2 + \lambda_1 \lambda_2) =\frac{2}{3}|P|^2$. The inequality \eqref{eq:mon} follows in view of the identity $-a^2 - \frac{b^2}{3} s_+ + \frac{2c^2}{3} s_+^2 = 0$.\footnote{Alternatively, one can argue that the operator $P \mapsto b^2\,s_+\,P + 2 (c^2\,s_+ - b^2)\,(Pn_* \cdot n_*) Q_*$ is an automorphism of $(T_{Q_*}\mcS_*)^\perp$ which has eigenvalues $2a^2 + \frac{b^2}{3}s_+$ along the direction parallel to $Q_*$ and $b^2s_+$ along the directions perpendicular to $Q_*$, which also gives \eqref{eq:mon}.}

Using \eqref{eq:mon} and Lax-Milgram's theorem in the Hilbert space
\[
\{P \in H_0^1(D,\mcS_0): P \in (T_{Q_*}\mcS_*)^\perp \text{ a.e. in } D\},
\]
one can easily show that, for every $q \in L^2(D,\mcS_0)$, there exists a unique solution $P \in H^1_0(D,\mcS_0)$ to the problem
\begin{equation}
\left\{\begin{array}{rcll}
L_{\eps,\perp} P &=& q + F(x) &\text{ in } D,\\
P &\in&  (T_{Q_*}\mcS_*)^\perp &\text{ a.e. in } D,\\
P &=& 0 & \text{ on } \partial D,
\end{array}\right.
	\label{Eq:LPBVP}
\end{equation}
where $F(x) \in T_{Q_*}\mcS_*$ is a Lagrange multiplier accounting for the constraint $P \in  (T_{Q_*}\mcS_*)^\perp$ a.e. in $D$. Therefore we have the following definition.

\begin{definition}\label{def:Lperp}
For $q \in  L^2(D,\mcS_0)$, we define  $L_{\eps,\perp}^{-1} q \in  H^1_0(D,\mcS_0)$ to be the unique solution to \eqref{Eq:LPBVP}. 
\end{definition}

We summarise properties of the operator $L_{\eps,\perp}^{-1}$ in the following lemma.

\begin{lemma}\label{Lem:LPInverseEst}
For every $\eps > 0$, the range of the operator $L_{\eps,\perp}^{-1}: L^2(D,\mcS_0) \rightarrow H^1_0(D,\mcS_0)$ is
\begin{equation}
Y : = \{P \in H^1_0(D,\mcS_0) \cap H^2(D,\mcS_0): P \in  (T_{Q_*}\mcS_*)^\perp \text{ in } D\}.
	\label{Eq:YSpaceDef}
\end{equation}
Furthermore, there exists $C > 0$ such that, for every $0 < \eps < 1$,
\[
\|L_{\eps,\perp}^{-1}q\|_\eps
 \leq C\|q\|_{L^2(D)} \text{ for all } q \in L^2(D,\mcS_0),
\]
where $\|\cdot\|_\eps$ was defined in \eqref{Eq:H2epsNorm}.
\end{lemma}

\begin{proof} In the proof $C$ will denote some generic constant which varies from line to line but is always independent of $\eps$.

Let us fix some $q \in L^2(D,\mcS_0)$ and let $P$ be the solution of \eqref{Eq:LPBVP}. Testing \eqref{Eq:LPBVP} against $P$ and using \eqref{eq:mon}, we obtain
\begin{equation}
\|P\|_{L^2(D)} + \eps \|\nabla P\|_{L^2(D)} \leq C\|q\|_{L^2(D)}.
	\label{Eq:P1.EstX}
\end{equation}
Next, we would like to show that $P \in H^2(D,\mcS_0)$. Let $\mcPP = \mcPP_x$ be the orthogonal projection of $\mcS_0$ onto $T_{Q_*(x)} \mcS_*$. Then, the first equation of \eqref{Eq:LPBVP} is equivalent to
\[
F(x) = \mcPP_x(L_{\eps,\perp} P(x) - q(x)).
\]
Here, we naturally extended $\mcPP$ to distributions, in particular to $\Delta P \in H^{-1}$, by defining $\langle \mcPP (\Delta P), \zeta \rangle := \langle \Delta P, \mcPP(\zeta) \rangle$ for every test function $\zeta \in C_c^\infty(D,\mcS_0)$. Therefore, to show that $F \in L^2(D,\mcS_0)$, it is enough to show that 
\begin{equation}
\text{$\mcPP(\Delta P) \in L^2(D,\mcS_0)$ for any $P \in H^1_0(D,\mcS_0)$, $P(x) \in  (T_{Q_*} \mcS_*)^\perp$ a.e. in $D$}.
\label{Eq:ProjBoost}
\end{equation}
In fact, we show
\begin{equation}
\|\mcPP(\Delta P)\|_{L^2(D)} \leq C\|P\|_{H^1(D)}
	\label{Eq:PBX}
\end{equation}
for all $P \in C_c^\infty(D,\mcS_0)$ satisfying $P(x) \in (T_{Q_*} \mcS_*)^\perp$ in $D$. To this end we use the following formula for $\mcPP$ which was computed in \cite[Eq. (3.4)]{NZ13-CVPDE}:\footnote{The brackets below indeed commute thanks to $Q_*^2-\frac13s_+Q_*-\frac29 s_+^2I_3=0$.  }
\begin{align*}
\mcPP(A) 
	&= A + \frac{2}{s_+^2}\left(\frac{1}{3}s_+ A - AQ_* - Q_* A\right)\left(Q_* - \frac{1}{6}s_+\,I_3\right)\\
	&= A + \frac{2}{s_+^2}\left(Q_* - \frac{1}{6}s_+\,I_3\right)\left(\frac{1}{3}s_+ A - AQ_* - Q_* A\right), \qquad \text{ for all } A \in \mcS_0.
\end{align*}
Since $P(x) \in (T_{Q_*} \mcS_*)^\perp$ in $D$, $\mcPP(P) = 0$ in $D$ and so $\Delta \mcPP(P) = 0$ in $D$. On the other hand, as $Q_*$ is smooth, it follows from the above formula for $\Pi$, applied to $P$ and $\Delta P$, that
\[
|\Delta \mcPP(P) - \mcPP(\Delta P)| \leq C(|\nabla P| + |P|).
\]
Combining the above two facts, we obtain \eqref{Eq:PBX} and hence \eqref{Eq:ProjBoost}.

It follows that $F \in L^2(D,\mcS_0)$ and $P \in H^2(D,\mcS_0)$. Moreover, for $0 < \eps < 1$,
\be \label{eq:est2}
\|F\|_{L^2(D)} \leq C(\eps^2 \|\nabla P\|_{L^2(D)} + \|P\|_{L^2(D)} + \|q\|_{L^2(D)}).
\ee
Using previously established estimates \eqref{Eq:P1.EstX} and \eqref{eq:est2} we obtain that $\| F \|_{L^2(D)} \leq C \|q\|_{L^2(D)}$. Returning to the first equation in \eqref{Eq:LPBVP}, elliptic estimates yield $\eps^2\|\nabla^2 P\|_{L^2(D)} \leq C\,\|q\|_{L^2(D)}$.
\end{proof}

\subsection{Solution of \eqref{Eq:psiEq} for given $P$}\label{ssec:psiP}

In this section we solve equation \eqref{Eq:psiEq} for given $P$. The properties of the map $P \mapsto \psi_\eps (P)$ obtained in this section will be used later in proving uniqueness of the critical point of $\mcF_\eps$ in a small neighbourhood of $Q_*$ using fixed point arguments.

\vskip 0.3cm

We define the following set 
\begin{equation}
\mathcal{U}_{\eps,C_1,C_2} := \Big\{P \in Y :\eps^2\|\nabla^2 P\|_{L^2(D)}  \leq \frac{1}{C_2}, \|P\|_{L^2(D)} \leq C_1\Big\},
\label{Eq:UeCCDef}
\end{equation}
where $Y$ is given in \eqref{Eq:YSpaceDef}. Note that, by integration by parts,
\begin{equation}
\eps\|\nabla P\|_{L^2(D)} \leq C_1^{1/2}\,C_2^{-\frac{1}{2}}\text{ for all } P \in \mathcal{U}_{\eps,C_1,C_2}.
	\label{Eq:11V19-UeCCInter}
\end{equation}

\begin{proposition}\label{Prop:psiofP}
Let $X$ be defined by \eqref{Eq:XSpaceDef}. For every $C_1 > 0$, there exist large $C_2 > 1$ and small $\eps_0 > 0$ such that, for every $0 < \eps < \eps_0$,
\begin{enumerate}[(i)]
\item For every $P \in \mathcal{U}_{\eps,C_1,C_2}$, there exists a unique $\psi_\eps(P) \in X$ satisfying simultaneously equation \eqref{Eq:psiEq} and $\|\psi_\eps(P)\|_{H^2(D)} \leq \frac{1}{C_2}$.
\end{enumerate}
Furthermore, there exists $C > 0$ (depending on $C_1, C_2$) such that:
\begin{enumerate}[(i)]
\setcounter{enumi}{1}
\item For every $P \in \mathcal{U}_{\eps,C_1,C_2}$, $\|\psi_\eps(P)\|_{H^2(D)} \leq C \eps^2\|P\|_{H^1(D)}$. In particular, $\psi_\eps(0) = 0$.
\item For every $P, \tilde P \in \mathcal{U}_{\eps,C_1, C_2}$, 
\[
\|\psi_\eps(P) - \psi_\eps(\tilde P)\|_{H^2(D)} 
	\leq C\eps \|P -\tilde P\|_{\eps},
\]
where $\| \cdot \|_\eps$ is defined in \eqref{Eq:H2epsNorm}.
\end{enumerate}
\end{proposition}

\begin{proof} Let us fix some $C_1 > 0$. In this proof $C$ will denote some generic constant which may depend on $C_1, a^2, b^2, c^2$ but is independent of $\eps$ (and $C_2$ and $P$ which will appear below). 

For $P \in Y$, define an operator $K_{\eps,P}: X \rightarrow X$ by
\[
K_{\eps,P}(\psi) = L_{\parallel}^{-1}(A[\psi] + \eps^2B_\eps [\psi,P]),
\]
where $L_{\parallel}^{-1}$ is given in Definition~\ref{def:Lpar}, and $A$ and $B_\eps$ are the operators appearing on the right hand side of \eqref{Eq:psiEq}.

\medskip\noindent\underline{Proof of (i):} It suffices to show that, for sufficiently large $C_2$ and all $P \in \mathcal{U} := \mathcal{U}_{\eps,C_1,C_2}$, $K_{\eps,P}$ is a contraction on the set $\mathcal{O} =\mathcal{O}_{C_2} := \{\psi \in X: \|\psi\|_{H^2(D)} \leq \frac{1}{C_2}\}$. 
 
Observe that, in view of \eqref{Eq:11V19-UeCCInter} and Poincar\'e's inequality in $H^1_0(D,\mcS_0)$, one has for all sufficiently large $C_2$ that 
\[
\eps\|P\|_{L^4(D)} \leq  C\eps\| P\|_{H^1(D)} \leq \frac{C}{C_2^{1/2}} < 1 \text{ for all } P \in  \mathcal{U}.
\]
Estimates \eqref{est:Aest} and \eqref{est:Best1} imply, for $\psi,\tilde\psi \in \mathcal{O}$ and $P \in \mathcal{U}$,
\begin{align*}
\|A[\psi] - A[\tilde\psi]\|_{L^2(D)}
	&\leq \frac{C}{C_2}\|\psi - \tilde\psi\|_{H^2(D)},\\
\|B_\eps[\psi,P] - B_\eps[\tilde \psi, P]\|_{L^2(D)}
	&\leq \frac{C\eps^{-2}}{C_2}\|\psi - \tilde \psi\|_{H^2(D)}.
\end{align*}
Therefore, by Lemma \ref{Lem:LINorm}, we have for $\psi,\tilde\psi \in \mathcal{O}$ and $P \in \mathcal{U}$ that
\[
\|K_{\eps,P}(\psi) - K_{\eps,P}(\tilde\psi)\|_{H^2(D)} \leq \frac{C}{C_2}\,\|\psi - \tilde\psi\|_{H^2(D)}.
\]
Also, by  \eqref{est:Aest0}, \eqref{est:Best0} and Lemma \ref{Lem:LINorm},
\begin{equation}
\|K_{\eps,P}(0)\|_{H^2(D)} 
	\leq C\eps^2\|B_\eps[0,P]\|_{L^2(D)} 
		\leq C \eps^2\|P\|_{H^1(D)} \leq 
		 \frac{C\eps}{C_2^{1/2}} \text{ for all } P \in \mathcal{U}.
	\label{Eq:10V19-KeP0}
\end{equation}
From the above two estimates, we deduce that there exist a large constant $C_2>1$ and a small constant $\eps_0 > 0$ such that, for every $\eps \in (0, \eps_0)$ and for every $P \in \mathcal{U}$, $K_{\eps,P}$ is a contraction from $\mathcal{O}$ into $\mathcal{O}$ and so has a unique fixed point $\psi_\eps(P) \in \mathcal{O}$.

\medskip\noindent\underline{Proof of (ii) and (iii):}  We now fix $C_2$ so that $\psi_\eps$ is defined on $\mathcal{U}$ as above and $K_{\eps,P}$ is a contraction from $\mathcal{O}$ into itself.

It follows from \eqref{Eq:10V19-KeP0} and Lemma \ref{Lem:ContractProperty} (see below) that the unique fixed point $\psi_\eps(P) \in \mathcal{O}$ of $K_{\eps,P}$ satisfies
$$
\|\psi_\eps(P)\|_{H^2(D)} \leq C \eps^2\|P\|_{H^1(D)},
$$ 
which proves (ii).

Next, Lemma \ref{Lem:LINorm} and estimate \eqref{est:Best2} imply that 
\begin{align*}
&\|K_{\eps,P}(\psi) - K_{\eps,\tilde P}(\psi)\|_{H^2(D)} \\
	&\qquad\leq C\eps^2\|\psi\|_{H^2(D)}\big\{\|P -\tilde P\|_{H^2(D)} +  (\|P\|_{H^2(D)} + \|\tilde P\|_{H^2(D)})\|P -\tilde P\|_{L^2(D)}\big\}\\
		&\qquad\qquad + C\eps^2\|P - \tilde P\|_{H^1(D)}
		\text{ for all } \psi \in \mathcal{O} \text{ and } P, \tilde P \in \mathcal{U}.
\end{align*}
Taking $\psi = \psi_\eps(P)$ and using (ii), we find that
\begin{align*}
&\|\psi_\eps(P) - K_{\eps,\tilde P}(\psi_\eps(P))\|_{H^2(D)} \\
	&\qquad\leq C\eps^4\|P\|_{H^1(D)}\big\{\|P -\tilde P\|_{H^2(D)} +  (\|P\|_{H^2(D)} + \|\tilde P\|_{H^2(D)})\|P -\tilde P\|_{L^2(D)}\big\}\\
		&\qquad\qquad + C\eps^2\|P - \tilde P\|_{H^1(D)}\\
	&\qquad \leq C\eps\|P - \tilde P\|_\eps
	\text{ for all } P, \tilde P \in \mathcal{U}.
\end{align*}
Applying Lemma \ref{Lem:ContractProperty} (see below) to $K_{\eps,\tilde P}$ and $b = \psi_\eps(P)$, we obtain
\begin{align*}
\|\psi_\eps(P) - \psi_\eps(\tilde P)\|_{H^2(D)} \leq  C\eps\|P - \tilde P\|_\eps \text{ for all } P, \tilde P \in \mathcal{U}.
\end{align*}
This proves (iii) and completes the proof.
\end{proof}

We used the following simple lemma whose proof is omitted.

\begin{lemma}\label{Lem:ContractProperty}
If $(M,d)$ is a complete metric space and $K: M \rightarrow M$ is a $\lambda$-contraction ($0\leq \lambda < 1$) with a fixed point $a \in M$, then $d(a, b) \leq \frac{1}{1 - \lambda} d(K(b),b)$ for any $b \in M$.
\end{lemma}

\subsection{Uniqueness of critical points in a neighborhood of $Q_*$}\label{ssec:unique}

In this subsection we show the uniqueness of critical points of $\mcF_\eps$ in a small neighbourhood of $Q_*\in \{Q_*^\pm\}$ given in \eqref{Eq:Q*Def}. In particular, we prove the following version of the informal statement \eqref{dagger} formulated in Subsection~\ref{SSec:Heu}:
\begin{proposition}\label{Prop:ResUniq}
For every $C_1 > 0$, there exist large $C_2 > 1$ and small $\eps_0 > 0$ such that, for all $0 < \eps \leq \eps_0$, $\mcF_\eps$ has at most one critical point $Q_\eps$, represented by $(\psi_\eps, P_\eps)$ as in Lemma~\ref{lemma:EL}, with $\|\psi_\eps\|_{H^2(D)} \leq \frac{1}{C_2}$, $\eps^2 \|P_\eps\|_{H^2(D)} \leq \frac{1}{C_2}$, and $\|P_\eps\|_{L^2(D)} \leq C_1$.
\end{proposition}

\begin{remark}
In the proof of Proposition \ref{Prop:ResUniq}, the exact form of $n_*$ is used only to have the tubular neighborhood representation (Lemma \ref{lem:TubNbhd}) and the stability inequality (Lemma \ref{Lem:StrStab}). Therefore, provided these are true, the statement of Proposition~\ref{Prop:ResUniq} will hold 
for more general domains and boundary conditions.
\end{remark}

\begin{proof} 
Let $X$ and $Y$ be defined by \eqref{Eq:XSpaceDef} and \eqref{Eq:YSpaceDef}. Let $L_{\eps,\perp}^{-1}$ be as in Definition~\ref{def:Lperp} and
\[
\theta_\eps := L_{\eps,\perp}^{-1}(s_+ \Delta(n_* \otimes n_*) ) \in Y.
\]
By Lemma \ref{Lem:LPInverseEst}, as $n_*$ is smooth, we have for every $\eps\in (0,1)$:
$$
\|\theta_\eps\|_\eps \leq C_0
$$ 
for some constant $C_0$ independent of $\eps$, and where $\|\cdot\|_\eps$ is as defined in \eqref{Eq:H2epsNorm}.

Fix some $C_1 > 0$ and let $\eps_0\in (0,1)$ and $C_2$ be as in Proposition \ref{Prop:psiofP}. By shrinking $\eps_0$ if necessary, we have for $0 < \eps \leq \eps_0$ that the solution $\psi_\eps(P)$ to \eqref{Eq:psiEq} is defined for all given $P \in \mathcal{U} := \mathcal{U}_{\eps, C_1, C_2}$ (see \eqref{Eq:UeCCDef}) and 
\begin{equation}
\|\psi_\eps(P)\|_{H^2(D)} \leq C\eps^2\,\|P\|_{H^1(D)} \leq C\eps C_2^{-1/2} < 1,
	\label{Eq:11V19-2dot}
\end{equation}
where we have used \eqref{Eq:11V19-UeCCInter}. Here and below, $C$ denotes some constant which may depend on $C_1$, $C_2$, $a^2, b^2, c^2$ but is always independent of $\eps$. For $P \in \mathcal{U} $ we define
\[
K_{\eps,\perp}(P) := L_{\eps,\perp}^{-1}\Big(s_+ \Delta(n_* \otimes n_*) + \mathring{C_\eps}[\psi_\eps(P),P]\Big) = \theta_\eps + L_{\eps,\perp}^{-1}(\mathring{C_\eps}[\psi_\eps(P),P]),
\]
where $\mathring{C_\eps}[\psi,P] = C_\eps[\psi,P] - \frac{1}{3}\tr(C_\eps[\psi,P]) I_3$  and $C_\eps$ is the operator appearing on the right hand side of \eqref{Eq:PEq}. It should be clear that if $P$ is a fixed point of $K_{\eps,\perp}$, then $(\psi_\eps(P),P)$ solves \eqref{Eq:psiEq}-\eqref{Eq:psiPCons}, and so the map $Q_\eps$ corresponding to $(\psi_\eps(P),P)$ in the representation Lemma \ref{lem:TubNbhd} is a critical point of $\mcF_\eps$. Therefore, to reach the conclusion, it suffices to show that for all sufficiently small $\eps$, the map $K_{\eps,\perp}$ has at most one fixed point in $\mathcal{U}$. In fact, we show that, for all small $\eps$, $K_{\eps,\perp}$ is contractive on $\mathcal{U}$ with respect to the norm $\|\cdot\|_\eps$.

In the following, we will use Ladyzhenskaya's inequality in two dimensions:
\[
\|\varphi\|_{L^4(D)} \leq C\|\varphi\|_{L^2(D)}^{1/2}\|\nabla \varphi\|_{L^2(D)}^{1/2} \text{ for all } \varphi \in H_0^1(D).
\]
In particular, it holds that
\begin{equation}
\|P\|_{L^4(D)} \leq C\|\nabla P\|_{L^2(D)}^{1/2} \leq \frac{C}{\eps^{1/2}} \text{ for all } P \in \mathcal{U}.
	\label{Eq:10V19-LadyI}
\end{equation}

Using the estimate \eqref{est:Cest} and inequality \eqref{Eq:10V19-LadyI}, we have
\begin{align*}
\|C_\eps[\psi,P] - C_\eps[\tilde \psi, \tilde P]\|_{L^2(D)}
	&\leq 
		C\,\|\psi - \tilde \psi\|_{H^2(D)}
			+C(\|\psi\|_{H^2(D)} +\|\tilde\psi\|_{H^2(D)})\|P - \tilde P\|_{L^2(D)} 
			\\
			&\qquad \qquad  + C\eps^2(\|\nabla P\|_{L^2(D)} + \|\nabla\tilde P\|_{L^2(D)})^{1/2}\|P - \tilde P\|_{H^1(D)}\\
	&\leq 
		C(\|\psi\|_{H^2(D)} +\|\tilde\psi\|_{H^2(D)})  \|P - \tilde P\|_{L^2(D)}
			+ C\,\|\psi - \tilde \psi\|_{H^2(D)}
			\\
			&\qquad \qquad  + C\eps^{\frac{3}{2}}
			\|P - \tilde P\|_{H^1(D)}
\end{align*}
for all $P, \tilde P \in \mathcal{U}$ and $\psi,\tilde\psi \in X$ with $\|\psi\|_{H^2(D)}, \|\tilde \psi\|_{H^2(D)} \leq 1$. Thus, by Proposition \ref{Prop:psiofP}(ii) and (iii), we get
\begin{equation}
\|C_\eps[\psi_\eps(P),P] - C_\eps[\psi_\eps(\tilde P),\tilde P]\|_{L^2(D)}
	\leq C\eps^{1/2}\|P - \tilde P\|_\eps
	 \text{ for all } P, \tilde P \in \mathcal{U}.
		\label{Eq:Ceps1}
\end{equation}
In view of Lemma \ref{Lem:LPInverseEst}, it follows that
\[
\|K_{\eps,\perp}(P) - K_{\eps,\perp}(\tilde P)\|_{\eps}
	\leq C\eps^{\frac{1}{2}}\|P - \tilde P\|_{\eps} \text{ for all } P, \tilde P \in \mathcal{U}.
\]
This implies that, for all sufficiently small $\eps$, $K_{\eps,\perp}$ has at most one fixed point in $\mathcal{U}$, which concludes the proof.
\end{proof}

\subsection{Proof of Theorem \ref{thm:MinSymmetryRes}}\label{SSec:ProofUniq}

\begin{proof}
For $\eps > 0$, let $\mcC_\eps \subset H_{Q_b}^1(D,\mcS_0)$ denote the set of minimizers of $\mcF_\eps$ in $H_{Q_b}^1(D,\mcS_0)$. Note that if $Q_\eps \in \mcC_\eps$, then $JQ_\eps J \in \mcC_\eps$ (where $J$ is given in \eqref{def:J}). 

Let $n_*^\pm$ be given by \eqref{Eq:n*Def} and $Q_*^\pm = s_+(n_*^\pm \otimes n_*^\pm - \frac{1}{3}I_3)$. It is well known that (see e.g. \cite{BBH, Chen-Lin-CAG}), if $\eps_m \rightarrow 0$ and $Q_{\eps_m} \in \mcC_{\eps_m}$, then $Q_{\eps_m}$ converges along a subsequence in $H^1(D,\mcS_0)$ to either $Q_* := Q_*^+$ or $Q_*^- = JQ_*^+J$. Thus, with $d = \frac{1}{3} \|Q_*^+ - Q_*^-\|_{H^1(D,\mcS_0)}$, it holds for all small $\eps>0$ that
\[
\mcC_\eps = \mcC_\eps^+ \cup \mcC_\eps^- \text{ where } \mcC_\eps^\pm := \mcC_\eps \cap \Big\{Q \in H_{Q_b}^1(D,\mcS_0): \|Q_*^\pm - Q\|_{H^1(D,\mcS_0)} < d\Big\}.
\]
It should be clear that $\mcC_\eps^\pm = J\mcC_\eps^\mp J$. To conclude, it is enough to show that, for all sufficiently small $\eps$, $\mcC_\eps^+$ consists of a single map which is $O(2)$-symmetric.

\medskip
\noindent {\it Step 1.} We prove that 
\begin{equation}
\sup_{Q \in\mcC_\eps^\pm} \|Q - Q_*^\pm\|_{H^2(D,\mcS_0)} \rightarrow 0 \text{ as } \eps \rightarrow 0.
	\label{Eq:10V19-distH2}
\end{equation}
In fact, it suffices to show that $\sup_{Q \in\mcC_\eps^+} \|Q - Q_*^+\|_{H^2(D,\mcS_0)} \rightarrow 0$ when $\eps \rightarrow 0$. Set $Q_*:=Q_*^+$. Arguing indirectly, suppose that there exist $\eps_m \rightarrow 0$ and $Q_{\eps_m} \in \mcC_{\eps_m}^+$ such that $\|Q_{\eps_m} - Q_*\|_{H^2(D,\mcS_0)} \geq \frac{1}{C} > 0$.
By \cite[Theorem 1]{NZ13-CVPDE}, $Q_{\eps_m}$ converges strongly to $Q_*$ in $C^{1,\sigma}(\bar D)$ for any $\sigma \in (0,1)$ and in $C^2_{\rm loc}(D)$ (note that in the cited paper the results are in $3D$ domains but one can easily check that those convergences also hold in $2D$ domains). Furthermore, by \cite[Corollary 2]{NZ13-CVPDE}, $\Delta  Q_{\eps_m}$ is bounded in $L^\infty(D)$. By Lebesgue's dominated convergence theorem
\[
\lim_{\eps \rightarrow 0} \int_D |\Delta Q_{\eps_m}|^2\,dx = \int_D |\Delta Q_*|^2\,dx,
\]
and so $\Delta Q_{\eps_m}$ converges to $\Delta Q_*$ in $L^2(D, \mcS_0)$. By elliptic estimates, we conclude that $Q_{\eps_m}$ converges to $Q_*$ in $H^2(D,\mcS_0)$, which gives a contradiction. We have thus established \eqref{Eq:10V19-distH2}.

In view of \eqref{Eq:10V19-distH2} and Lemma \ref{lem:TubNbhd}, for all sufficiently small $\eps$ and $Q_\eps \in \mcC_\eps^+$, we can represent
$$
Q_\eps= \underbrace{s_+\Big(\frac{n_* + \psi_\eps}{|n_* + \psi_\eps|} \otimes \frac{n_* + \psi_\eps}{|n_* + \psi_\eps|} - \frac{1}{3}I_3\Big)}_{=Q_{\eps,\sharp}}+\eps^2 P_\eps,
$$
where $\psi_\eps \cdot n_* = 0$ and $P_\eps \in (T_{Q_*}\mcS_*)^\perp$. We let $\widetilde\mcC_\eps^+$ denote the set of $(\psi,P)$ representing elements of $\mcC_\eps^+$ as above:
\[
\widetilde\mcC_\eps^+ = \Big\{(\psi,P): s_+\Big(\frac{n_* + \psi}{|n_* + \psi|} \otimes \frac{n_* + \psi}{|n_* + \psi|} - \frac{1}{3}I_3\Big)+\eps^2 P \in \mcC_\eps^+
	\Big\}.
\]
By \eqref{Eq:10V19-distH2} and Lemma \ref{lem:TubNbhd},
\begin{equation}
\sup_{(\psi,P) \in \widetilde \mcC_\eps^+} \Big[\|\psi\|_{H^2(D,\RR^3)} +  \eps^2 \|P\|_{H^2(D,\mcS_0)}\Big] \rightarrow 0 \text{ as } \eps \rightarrow 0.
 	\label{Eq:ReqRoC}
\end{equation}

\medskip
\noindent {\it Step 2.} In view of \eqref{Eq:ReqRoC} and Proposition \ref{Prop:ResUniq}, in order to prove that $\mcC_\eps^+$ consists of a single point for all sufficiently small $\eps$, it suffices to show that there exist $\eps_1 > 0$ and $C_1 > 0$ such that, for all $\eps \in (0,\eps_1)$,
\be
\sup_{(\psi,P) \in \widetilde \mcC_\eps^+}\| P\|_{L^2(D,\mcS_0)} \leq C_1.
	\label{bddPeps}
\ee

We recall some results from \cite{NZ13-CVPDE}. Let $Q_\eps \in \mcC_\eps^+$ and $(\psi_\eps, P_\eps) \in \widetilde \mcC_\eps^+$ be its corresponding representation as above. We consider the tensor:
\[
X_\eps := \frac{1}{\eps^2}[Q_\eps^2 - \frac{1}{3}s_+\,Q_\eps - \frac{2}{9}s_+^2 I_3].
\]
(The polynomial on the right hand side is a multiple of the minimal polynomial of matrices belonging to the limit manifold $\mcS_*$.) By \cite[Proposition 4]{NZ13-CVPDE}, 
\[
X_\eps \text{ is bounded in } C^0(\bar D).
\]
  As $Q_{\eps,\sharp}\in\mcS_*$, we have $Q_{\eps,\sharp}^2 - \frac{1}{3}s_+\,Q_{\eps,\sharp} - \frac{2}{9}s_+^2 I_3 = 0$ and thus
\begin{equation}
X_\eps = Q_{\eps,\sharp} P_\eps + P_\eps\,Q_{\eps,\sharp} - \frac{1}{3}s_+\,P_\eps + \eps^2 P_\eps^2 .
	\label{Eq:30VII19-R1}
\end{equation}
Let $End(\mcS_0)$ be the set of linear endomorphisms of $\mcS_0$ and define $\mu_\eps: D \rightarrow End(\mcS_0)$ by
\[
\mu_\eps(x)(M) = (Q_{\eps,\sharp}(x) - Q_*(x)) M + M\,(Q_{\eps,\sharp}(x) - Q_*(x)) + \eps^2 P_\eps(x) M
\]
 for all $x \in \bar D$ and $M \in \mcS_0$. Then \eqref{Eq:30VII19-R1} is equivalent to
\[
X_\eps = Q_*\,P_\eps + P_\eps\,Q_*  - \frac{1}{3}s_+\,P_\eps + \mu_\eps  P_\eps
\]
where $\mu_\eps  P_\eps$ stands for the map $x \mapsto \mu_\eps(x)(P_\eps(x))$. By definition, $P_\eps \in (T_{Q_*}\mcS_*)^\perp$ and so by \cite[Lemma 2]{NZ13-CVPDE}, $P_\eps$ commutes with $Q_*$. It follows that 
\[
X_\eps = 2s_+(n_* \otimes n_* - \frac{1}{2}I_3)\,P_\eps + \mu_\eps  P_\eps.
\]
Note that $(n_* \otimes n_* - \frac{1}{2}I_3)$ is invertible when considered as an endomorphism of $\mcS_0$ and that $\lim_{\eps \rightarrow 0} \|\mu_\eps\|_{C^0(\bar D)} = 0$ (in view of \eqref{Eq:ReqRoC}). Thus, as $X_\eps$ is bounded in $C^0(\bar D)$, we have that $P_\eps$ is also bounded in $C^0(\bar D)$, and in particular in $L^2(D)$. The assertion in \eqref{bddPeps} is established. By Proposition \ref{Prop:ResUniq}, we hence have for all sufficiently small $\eps$ that $\mcC_\eps^+$ consists of a single element and $\mcC_\eps$ consists of exactly two distinct $\Z_2$-conjugate elements.

\medskip
\noindent{\it Step 3:} Let $Q_\eps$ denote the unique element of $\mcC_\eps^+$. We show that $Q_\eps$ is $O(2)$-symmetric. Recall that we denoted in \eqref{def:Ars} by $\mcA^{rs}$ the set of $O(2)$-symmetric maps in $H^1(D,\mcS_0)$. Let $\mcC_\eps^{rs}$ denote the set of minimizers of $\mcF_\eps|_{\mcA^{rs}}$. The same argument as above shows that, for all small $\eps$, $\mcC_\eps^{rs} = \{Q_\eps^{rs,+}, Q_\eps^{rs,-} = JQ_\eps^{rs,+}J\}$, and in the representation $Q_\eps^{rs,+} \approx (\psi_\eps^{rs}, P_\eps^{rs})$, it holds that
\begin{align*}
\lim_{\eps \rightarrow 0} \Big[\|\psi_\eps^{rs}\|_{H^2(D,\RR^3)} +  \eps^2 \|P_\eps^{rs}\|_{H^2(D,\mcS_0)}\Big] = 0  \text{ and } 
\| P_\eps^{rs}\|_{L^2(D,\mcS_0)} \leq C_1'
\end{align*}
for some constant $C_1'$ independent of $\eps$. Another application of Proposition \ref{Prop:ResUniq} thus yields $ Q_\eps \equiv Q_\eps^{rs,+}$ for all small $\eps$  and so $Q_\eps$ is $O(2)$-symmetric. 

Finally, note that $Q_\eps$ and $JQ_\eps J$ are distinct as $\mcC_\eps^+\cap \mcC_\eps^-=\emptyset$, so they are not $\Z_2$-symmetric. This completes the proof.
\end{proof}

\subsection{Proof of Theorem \ref{thm:MinSymmetry}}\label{SSec:ProofUniqThm1.5}

\begin{proof} Using Theorem \ref{thm:MinSymmetryRes} and a simple scaling argument, we find $R_0 = R_0(a^2, b^2, c^2, k) > 0$ such that for all $R > R_0$, there exist exactly two global minimizers $Q^\pm$ of $\mcF[\cdot;B_R]$ subjected to the boundary condition \eqref{BC1} and these minimizers are $k$-fold $O(2)$-symmetric and are $\Z_2$-conjugate to each other. By Proposition \ref{Prop:okrChar}, we can express $Q^\pm$ in the form
\begin{align*}
Q^\pm (x) = w_0(|x|) E_0 + w_1(|x|) E_1 \pm w_3(|x|)  E_3 \quad \textrm{for every } x\in B_R.
\end{align*}
It is clear that $(w_0, w_1, 0, \pm w_3, 0)$ satisfies \eqref{Eq:w0}-\eqref{Eq:w4}.

Now, note that in view of formula \eqref{Eq:LdGW}, $\mcF[Q^\pm;B_R] = \mcF[w_0 E_0 + w_1 E_1 \pm |w_3|  E_3; B_R]$. Hence, $w_0 E_0 + w_1 E_1 \pm |w_3|  E_3$ are also minimizers of $\mcF[\cdot;B_R]$ satisfying \eqref{BC1}. By the above uniqueness up to $\Z_2$-conjugation, we may assume that $w_3 \geq 0$ in $B_R$. Also, as $Q^+ \neq Q^-$, $w_3 \not\equiv 0$. Recalling equation \eqref{Eq:w3} and noting that $w_2 = w_4 = 0$, we can apply the strong maximum principle to conclude that $w_3 > 0$. The proof is complete.
\end{proof}

\section{Mountain pass critical points}
\label{sec:multiplesol}

In this section, we give the proof of Theorem \ref{thm:FiveSol}, which asserts the existence of at least five $O(2)$-symmetric critical points satisfying the boundary condition \eqref{BC1} for $\mcFRdom := \mcF[\cdot; B_R]$ for all large enough $R$.

We denote by  $\mcA_R^{rs}$ and $\mcA_R^{str}$ the sets of $k$-fold $O(2)$-symmetric and $\Z_2 \times O(2)$-symmetric maps, respectively, satisfying the boundary conditions \eqref{BC1}:
\begin{align*}
\mcA_R^{rs}
	& = \Big\{Q \in H_{Q_b}^1(B_R,\mcS_0): \text{$Q$ is $O(2)$-symmetric}\Big\},\\
\mcA_R^{str}
	& = \Big\{Q \in H_{Q_b}^1(B_R,\mcS_0): \text{$Q$ is  $\Z_2 \times O(2)$-symmetric}\Big\}.
\end{align*}

By the characterization of symmetric maps (see Propositions \ref{Prop:okrChar} and \ref{prop:Z2O2}), we can express the sets $\mcA_R^{rs}$ and $\mcA_R^{str}$ in terms of the basis components defined in Section \ref{Sec:WkRSS} as follows:
\begin{align*}
\mcA_R^{rs}
	& = \Big\{
Q(x) =w_0(|x|)E_0+w_1(|x|)E_1+w_3(|x|)E_3\, :\\ 
   &\qquad w_0 \in H^1((0,R);r\,dr), w_1,  w_3\in H^1((0,R);r\,dr) \cap L^2((0,R);\frac{1}{r}\,dr), \\
   	&\qquad w_0(R) = -\frac{s_+}{\sqrt{6}}, w_1(R) = \frac{s_+}{\sqrt{2}}, w_3(R) = 0\Big\},\\
\mcA_R^{str}
	& = \Big\{Q(x) =w_0(|x|)E_0+w_1(|x|)E_1\, :\\ 
   &\qquad w_0 \in H^1((0,R);r\,dr), w_1\in H^1((0,R);r\,dr) \cap L^2((0,R);\frac{1}{r}\,dr), \\
   &\qquad w_0(R) = -\frac{s_+}{\sqrt{6}}, w_1(R) = \frac{s_+}{\sqrt{2}}\Big\}.
\end{align*}
A direct computation shows that critical points of $\mcFRdom$ in $\mcA_R^{rs}$ or $\mcA_R^{str}$ are in fact critical points of $\mcFRdom$ in $H_{Q_b}^1(B_R, \mcS_0)$ (cf. Remark \ref{Rem:ODESys}). To prove Theorem \ref{thm:FiveSol}, we use the fact that $\mcFRdom$ has two global minimizers in $\mcA_R^{rs}$ (due to Theorem \ref{thm:MinSymmetry}) and the mountain pass theorem. An energetic consideration is needed to show that the obtained mountain pass critical point does not coincide with critical points of $\mcFRdom$ in $\mcA_R^{str}$.

We start with an estimate for the minimal energy of $\mcFRdom$ in $\mcA_R^{str}$.

\begin{lemma}\label{Lem:StrongLogDiv}
There exists some $C > 0$ depending only on $a^2, b^2$ and $c^2$ such that, for all $\delta \in (0,1)$, $k \in \ZZ \setminus \{0\}$ and $R > \max(1,\frac{Cek^2}{\delta^2})$, there holds
\begin{equation}
\frac{\pi s_+^2 k^2}{2} \,\ln R  + C k^2  \geq \alpha_R := \min_{\mcA_R^{str}}\mcFRdom \geq  \frac{\pi s_+^2 k^2}{2(1 + 2\delta)^2} \,\left(\ln \frac{\delta^2 R}{Ck^2} - \ln \ln \frac{\delta^2 R}{Ck^2}\right).
	\label{Eq:alphaRLB}
\end{equation}
As a consequence,
\begin{equation}
\lim_{R \rightarrow \infty} \frac{\alpha_R}{\ln R} = \frac{1}{2}\pi s_+^2 k^2.
	\label{Eq:alphaRAs}
\end{equation}
\end{lemma}

\begin{remark}
In \cite{BaumanParkPhillips}, it was shown that $\mcFRdom$ has critical points whose energies are of order $k\ln R$; and these are not $SO(2)$-symmetric for $k \neq \pm 1$.
\end{remark}

\begin{proof} By \eqref{Eq:LdGW}, we have
\begin{multline*}
\alpha_R 
	= 2\pi \min \Big\{\mcE_R[w_0,w_1]: w_0 \in H^1((0,R);r\,dr), w_1 \in H^1((0,R);r\,dr) \cap L^2((0,R);\frac{1}{r}dr), \\
			 w_0(R) = -\frac{s_+}{\sqrt{6}}, w_1(R) = \frac{s_+}{\sqrt{2}}\Big\},
\end{multline*}
where
\begin{align*}
\mcE_R[w_0, w_1]
	& = \int_0^R  \Big\{\frac{1}{2}[|w_0'|^2 + |w_1'|^2] + \frac{k^2}{2r^2}|w_1|^2 + h(w_1,w_0)\Big\}rdr,\\
h(x,y)
	&= \big(-\frac{a^2}{2} + \frac{c^2}{4}[|x|^2 + |y|^2]\big)[|x|^2 + |y|^2]
		- \frac{b^2\sqrt{6}}{18}\,y(y^2 - 3x^2) - f_*,
\end{align*}
and $f_*$ is given by \eqref{Eq:f*def}.

We note (see e.g. \cite[Lemma 5.1]{INSZ_CVPDE}) that $h(x,y) \geq 0$ and equality holds if and only if $(x,y)$ belongs to the set $\{(\pm\frac{s_+}{\sqrt{2}},-\frac{s_+}{\sqrt{6}}), (0,\frac{2s_+}{\sqrt{6}})\}$. Furthermore, the Hessian of $h$ is positive definite at these critical points. In particular, one has%
\begin{equation}
h(x,y) \geq \frac{1}{C}(x - \frac{s_+}{\sqrt{2}})^2 \text{ for all } (x,y) \text{ satisfying } x^2 + y^2 \leq \frac{2}{3}s_+^2, \frac{s_+}{3\sqrt{2}} \leq x \leq \frac{s_+}{\sqrt{2}},
	\label{Eq:hQdep}
\end{equation}
where, here and below, $C$ denotes some positive constant (that may change from line to line) which depends only on $a^2$, $b^2$ and $c^2$, and in particular is always independent of $R$, $k$ and $\delta$. 

\medskip
\noindent\underline{Step 1:} Proof of the upper bound for $\alpha_R$ in \eqref{Eq:alphaRLB}.

Consider the test function $(\bar w_0, \bar w_1)$ defined by $\bar w_0(r) \equiv -\frac{s_+}{\sqrt{6}}$ and $\bar w_1(r) = \frac{s_+}{\sqrt{2}} \min(r,1)$. Then
\[
\alpha_R \leq 2\pi\mcE_R[\bar w_0, \bar w_1] = 2\pi\mcE_1[\bar w_0, \bar w_1] + 2\pi\int_1^R \frac{s_+^2 k^2}{4r} dr \leq C k^2 + \frac{1}{2}\pi s_+^2 k^2 \ln R,
\]
which provides the upper bound on $\alpha_R$, given in the left hand side of  \eqref{Eq:alphaRLB}.

\medskip
\noindent\underline{Step 2:} Proof of the lower bound for $\alpha_R$ in \eqref{Eq:alphaRLB}.
\medskip

 Let $(w_0, w_1)$ be a minimizer of $\mcE_R$ the existence of which is guaranteed by the direct method of the calculus of variations.  We fix some $\delta \in (0,1)$. Due to the fact that $w_1$ is continuous, $w_1(0) = 0$ and $w_1(R) = \frac{s_+}{\sqrt{2}}$, there exists the largest number $R_1 \in (0,R)$ such that $w_1(R_1) = \frac{s_+}{(1 + 2\delta)\sqrt{2}}$. By the same arguments, there exists the smallest number $R_2 \in (R_1,  R)$ such that $w_1(R_2) = \frac{s_+}{(1 + \delta)\sqrt{2}}$.

As $w_1(r) \geq \frac{s_+}{(1 + 2\delta)\sqrt{2}}$ in $[R_1, R]$, we have
\[
\frac{1}{2\pi}\alpha_R = \mcE_R(w_0, w_1) \geq \int_{R_1}^R \frac{k^2}{2r} |w_1|^2\,dr \geq \frac{s_+^2 k^2}{4(1 + 2\delta)^2} \ln \frac{R}{R_1}.
\]
It follows that
\begin{equation}
R_1 \geq R\,\exp\Big(-\frac{2(1 + 2\delta)^2\alpha_R}{\pi s_+^2 k^2}\Big).
	\label{Eq:LogDiv1}
\end{equation}

On the other hand, by the definition of $R_1$ and $R_2$, we have $\frac{s_+}{(1 + 2\delta)\sqrt{2}} \leq w_1 \leq \frac{s_+}{(1+\delta)\sqrt{2}}$ in $[R_1,R_2]$. Also, by \cite[Eq. (3.12)]{INSZ_AnnIHP}, $w_0^2 + w_1^2 \leq \frac{2}{3} s_+^2$. Thus, using \eqref{Eq:hQdep}, we have
\[
h(w_1,w_0) \geq \frac{1}{C}(w_1 - \frac{s_+}{\sqrt{2}})^2 > \frac{\delta^2}{C} \text{ in } [R_1, R_2].
\]
Therefore, it follows that
\[
\frac{1}{2\pi}\alpha_R = \mcE_R(w_0, w_1) \geq \int_{R_1}^{R_2} h(w_0,w_1)\,r\,dr \geq \frac{\delta^2}{C}(R_2^2 - R_1^2) \geq \frac{\delta^2}{C} R_1(R_2 - R_1),
\]
and hence, in view of \eqref{Eq:LogDiv1},
\[
\frac{R_2 - R_1}{R_1} \leq \frac{C\alpha_R}{\delta^2 R^2}\,\exp\Big(\frac{4(1 + 2\delta)^2\alpha_R}{\pi s_+^2 k^2}\Big).
\]
This leads to, by Cauchy-Schwarz' inequality,
\begin{align*}
\frac{1}{\pi}\alpha_R 
	&= 2\mcE_R(w_0, w_1) \geq R_1\int_{R_1}^{R_2} |w_1'|^2\,dr \geq \frac{R_1}{R_2 - R_1}\Big(\int_{R_1}^{R_2} w_1'\,dr\Big)^2\\
	& = \frac{\delta^2 s_+^2\,}{2(1 + \delta)^2 (1 + 2\delta)^2} \frac{R_1}{R_2-R_1} \geq \frac{\delta^4 R^2}{C\alpha_R}\,\exp\Big(-\frac{4(1 + 2\delta)^2\alpha_R}{\pi s_+^2 k^2}\Big).
\end{align*}
Rearranging, we obtain
$
\Lambda e^\Lambda  \geq \frac{\delta^2 R}{C k^2}$ for $\Lambda := \frac{2(1 + 2\delta)^2\alpha_R}{\pi s_+^2 k^2},
$
which implies
\[
\frac{2(1 + 2\delta)^2\alpha_R}{\pi s_+^2 k^2} = \Lambda \geq \ln \frac{\delta^2 R}{C k^2} - \ln \ln \frac{\delta^2 R}{C k^2} \quad \text{ provided }  \quad \ln \frac{\delta^2 R}{C k^2} \geq 1.
\]
The conclusion of the result is immediate.
\end{proof}

In order to use the mountain pass theorem to show that $\mcF^R$ has more than two critical points, we need to exhibit a path $\gamma$ connecting the two minimizers $Q_R^\pm$ of $\mcF^R$ such that 
\[
\sup_{t} \mcF^R[\gamma(t)] < \alpha_R
\]
where $\alpha_R$ is the minimal energy of $\mcF^R|_{\mcA_R^{str}}$. 

The existence of such a path is a priori not clear. Indeed, note that, as $R \rightarrow \infty$ and after a suitable rescaling, $Q_R^\pm$ tend to $Q_*^\pm$ (see equation \eqref{Eq:Q*Def} in the previous section). As maps from $D$ into $\mcS_*$, $Q_*^+$ and $Q_*^-$ belong to different homotopy classes and so cannot be connected by a continuous path in $H^1(D,\mcS_*)$. The desired path $\gamma$ must therefore necessarily leave the limit manifold $\mcS_*$. In particular, the contribution of the bulk energy potential $f_{\rm bulk}$ to $\mcF^R[\gamma(t)]$ cannot be neglected.

Our construction of the path $\gamma$ is of a completely different flavor. We exploit the conformal invariance of the Dirichlet energy in $2D$ to connect $Q_R^\pm$ to $Q_{R_0}^\pm$ for some fixed $R_0$ by using $Q_r^\pm$ (with variable $r$) and their inverted copies, and then finally connect $Q_{R_0}^+$ and $Q_{R_0}^-$. As a result we obtain a mountain path with energy $O_k(1)$ (see \eqref{Eq:MPEnergy}), which is clearly less than $\alpha_R$ for large $R$.

\begin{proof}[Proof of Theorem \ref{thm:FiveSol}]
In the proof, $C$ denotes some positive constant which is always independent of $R$.
As denoted earlier, critical points of $\mcFRdom$ in $\mcA_R^{rs}$ or $\mcA_R^{str}$ are critical points of $\mcFRdom$ in $H_{Q_b}^1(B_R, \mcS_0)$. Therefore it suffices to work with $\mcFRdom\big|_{\mcA_R^{rs}}$. To simplify the notation in what follows we still use $\mcFRdom$ instead of $\mcFRdom\big|_{\mcA_R^{rs}}$.

By Theorem \ref{thm:MinSymmetry}, there exists $R_0 > 0$ such that, for $R \geq R_0$, $\mcFRdom$ has two distinct minimizers in $\mcA_R^{rs}$ which are $O(2)$-symmetric but not $\Z_2 \times O(2)$-symmetric (in fact, they are $\Z_2$-conjugate). We label these minimizers as $Q_R^\pm$ and claim that, for any $0 < d < \|Q_R^+ - Q_R^-\|_{H^1(B_R)}$, we have
\begin{align*}
\inf \Big\{ \mcFRdom[Q]: Q \in \mcA_R^{rs}, \|Q - Q_R^+\|_{H^1(B_R)} = d\Big\} > \mcFRdom[Q_R^\pm].
\end{align*}
Assume by contradiction that there exists a sequence $\{Q_m\}_{m\in\mathbb{N}} \subset \mcA_R^{rs}$ satisfying $\|Q_m - Q_R^+\|_{H^1(B_R)} = d$ such that $\mcFRdom[Q_m] \rightarrow \mcFRdom[Q_R^\pm]$ as $m\to \infty$. Without loss of generality, we can also assume that $Q_m$ is weakly convergent in $H^1(B_R,\mcS_0)$ and strongly convergent in $L^p(B_R,\mcS_0)$ for any $p \in [1,\infty)$. The limit of $Q_m$ is then a minimizer of $\mcFRdom$, and thus, by Theorem \ref{thm:MinSymmetry} and our assumption on $d$, must coincide with $Q_R^+$. Now, as $Q_m \to Q_R^+$ in $L^4(B_R,\mcS_0)$ and $\mcFRdom[Q_m] \rightarrow \mcFRdom[Q_R^+]$, we have that $\|\nabla Q_m\|_{L^2(B_R)} \to \|\nabla Q_R^+\|_{L^2(B_R)}$, which further implies that $Q_m \to Q_R^+$ in $H^1(B_R,\mcS_0)$ as $m\to \infty$. This contradicts the fact that $\|Q_m - Q_R^+\|_{H^1(B_R)} = d>0$ for every $m$. The claim is proved. 

It is standard to check that $\mcF^R$ satisfies the Palais-Smale condition. Indeed, if $Q_m$ is a Palais-Smale sequence for $\mcF^R$, then as $f_{\rm bulk}\geq 0$, $Q_m$ is bounded in $H^1(B_R,\mcS_0)$. Now note that $D\mcFRdom(Q_m) = -\Delta Q_m + V(Q_m)$ for some nonlinear operator $V: H^1(B_R,\mcS_0) \rightarrow H^{-1}(B_R,\mcS_0)$ which, by the compact embedding theorem, maps bounded sets of $H^1(B_R,\mcS_0)$ into relatively compact sets of $H^{-1}(B_R,\mcS_0)$. Thus, up to extracting a subsequence, we may assume that $Q_m \rightharpoonup Q$ weakly in $H^1$ and $V(Q_m) \rightarrow V(Q)$ in $H^{-1}$. As $D\mcFRdom(Q_m) \rightarrow 0$ in $H^{-1}$ it follows that $-\Delta Q_m \rightarrow -\Delta Q$ in $H^{-1}$ and so $Q_m \rightarrow Q$ in $H^{1}$ as wanted.

Applying the mountain pass theorem (see e.g. \cite[Theorem 6.1]{Struwe}), we conclude for $R \geq R_0$ that $\mcFRdom$ has a mountain pass critical point in $\mcA_R^{rs}$ connecting $Q_R^\pm$, which will be denoted by $Q_R^{mp}$. 

\medskip

\noindent {\bf Claim}: There exists some $C > 0$ independent of $k$ such that
\begin{equation}
\mcFRdom[Q_R^{mp}] \leq C(R_0^2 + |k|) \text{ for all } R > R_0.
	\label{Eq:MPEnergy}
\end{equation}
To this end, it suffices to construct a continuous path $\gamma: [-2,2] \rightarrow \mcA_R^{rs}$ such that $\gamma(\pm 2) = Q_R^\pm$ and 
\begin{equation}
\mcFRdom[\gamma(t)] \leq C(R_0^2 + |k|) \text{ for all } t \in [-2,2],
	\label{Eq:MPPathEnergy}
\end{equation}
where $C$ is independent of $R$, $k$ and $t$.

\medskip

\noindent {\bf Proof of Claim}. Let $n_*^\pm$ be defined by \eqref{Eq:n*Def}. Its rescaled version to $B_R$ is given by
\[
n_{R,*}^\pm(r\cos\varphi,r\sin\varphi) = \Big(\frac{2R^{\frac{k}{2}}r^{\frac{k}{2}} \cos(\frac{k}{2}\varphi)}{R^k + r^k},\frac{2R^{\frac{k}{2}}r^{\frac{k}{2}} \sin(\frac{k}{2}\varphi)}{R^k + r^k}, \pm \frac{R^k - r^k}{R^k + r^k}\Big).
\]
We define $Q_{R,*}^\pm = s_+(n_{R,*}^\pm \otimes n_{R,*}^\pm - \frac{1}{3}I_3)$ and note that $f_{\rm bulk}(Q_{R,*}^\pm) \equiv 0$. It follows that
\begin{equation}
 \mcFRdom[Q_R^\pm] \leq \mcFRdom[Q_{R,*}^\pm] = \frac{1}{2}\int_{B_R} |\nabla Q_{R,*}^\pm|^2\,dx = s_+^2 \int_D |\nabla n_*^\pm|^2 = 4\pi |k|\,s_+^2.
	\label{Eq:UnifEnergyBnd}
\end{equation}
\medskip

\noindent {\it Step 1.} We first construct $\gamma\big|_{[-2,-1] \cup [1,2]}$. For that, let $r_1, r_2: [-2,2] \rightarrow [R_0,R]$ be given by
\begin{align*}
r_1(t)
	&=\left\{ \begin{array}{ll} 
R_0, & t\in [-1,1],\\
(R-R_0)|t|+2R_0-R, & t\in [-2,2] \setminus [-1,1],
\end{array}\right.
\\
r_2(t) 
	&= (r_1(t)R)^{1/2} .
\end{align*}
\noindent For $1 \leq t \leq 2$ we define $\gamma(\pm t): B_R \rightarrow \mcS_0$ by
\[
\gamma(\pm t)(x) = \left\{\begin{array}{ll}
Q_{r_1(t)}^{\pm}(x) &\text{ if } |x| \leq r_1(t),\\
Q_R^+(\frac{r_2(t)^2}{|x|^2} x) & \text{ if } r_1(t) < |x| < r_2(t),\\
Q_R^+(x) & \text{ if } r_2(t) \leq |x| \leq R.
\end{array}\right.
\]
To dispel confusion, we note that on the lower two cases (i.e., $r_1(t) < |x| < R$), we are using the ``plus'' minimizing branch $Q_R^+$. Since for any $r>0$ we have $Q_r^\pm(r\frac{x}{|x|}) = s_+(\nbdry \otimes \nbdry - \frac{1}{3}I_3)$, the inner and outer traces of $\gamma(t)$ at  $\partial B_{r_1(t)}$ coincide and so $\gamma(t)$ belongs to $\mcA_R^{rs}$. See Figure \ref{fig:figMPpath}. The continuity of $\gamma$ with respect to $t$ is a consequence of the uniqueness part of Theorem \ref{thm:MinSymmetry}.

\begin{figure}[h]
\begin{center}
\begin{tikzpicture}
\draw[->] (-0.5,0) -- (7,0);
\draw[->] (0,-0.5)--(0,3.5);
\draw plot[smooth] coordinates{(0,0.5) (0.5,0.6)  (1,0.7) (1.5,0.9) (2,2.5)};
\draw plot[smooth] coordinates{(2,2.5) (2.2,1.3)  (2.4,1.1) (2.6, 0.9) (3,0.9) (3.4,0.7)};
\draw plot[smooth] coordinates{(6,2.5) (5.4,1.3)  (5,1.1) (4.6, 0.9) (4, 0.9) (3.4,0.7)};

\draw[dashed] (2,0)--(2,2.5);
\draw[dashed] (3.4,0)--(3.4,0.7);
\draw[dashed] (6,0)--(6,2.5);

\draw[dashed] (0,2.5)--(6,2.5);

\draw[dashed,->] (2.7,3.3)--(2.7,1);

\draw (6.9,-0.3) node {$x$}
(6,-0.3) node {$R$}
(3.4,-0.3) node {$r_2(t)$}
(2,-0.3) node {$r_1(t)$}
(-0.7,2.5) node {$Q_b(x)$}
(-0.7,3.4) node {$\gamma(t)(x)$}
(1,1.3) node {$Q_{r_1(t)}^+$}
(5,1.5) node {$Q_R^+$}
(3,3.5) node {inverted copy of $Q_R^+$}
;
\end{tikzpicture}
\end{center}
\caption{A schematic `graph' of $\gamma(t)$ for $t \in [1,2]$.}
\label{fig:figMPpath}
\end{figure}
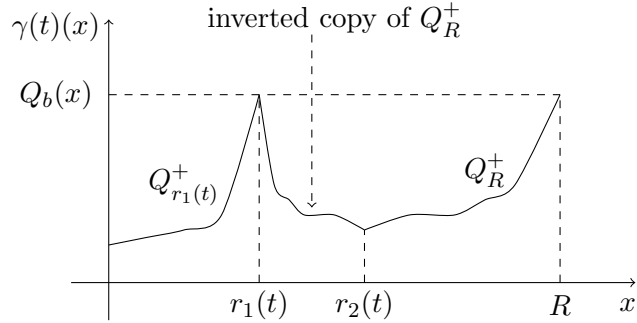

By construction, it is clear that $\gamma(\pm 2)=Q_R^{\pm}$ and $\gamma(\pm 1)\big|_{B_{R_0}}=Q_{R_0}^\pm$.

Let us check that \eqref{Eq:MPPathEnergy} holds for $1 \leq |t| \leq 2$. In view of \eqref{Eq:UnifEnergyBnd} and the fact that the integrand of $\mcF$ is non-negative, we have
$$ 
\mcF^{r_1(t)}[Q^\pm_{r_1(t)}]\leq C|k|, \quad \mcF[Q_R^+, B_R\setminus B_{r_2(t)}]\leq \mcFRdom[Q^+_R]\leq C|k|.
$$
Therefore, we only need to show that
\begin{equation}
\int_{B_{r_2(t)} \setminus B_{r_1(t)}} \Big[\frac{1}{2}|\nabla \gamma(t)|^2 + f_{\rm bulk}(\gamma(t))\Big]\,dx \leq 4\pi |k| s_+^2.
	\label{Eq:MPPEDerived}
\end{equation}
Indeed, by a change of variable $y = \frac{r_2(t)^2}{|x|^2} x$ , we have in view of \eqref{Eq:UnifEnergyBnd}:
\begin{align*}
\int_{B_{r_2(t)} \setminus B_{r_1(t)}} \Big[\frac{1}{2}|\nabla \gamma(t)|^2 + f_{\rm bulk}(\gamma(t))\Big]\,dx
	&=   \int_{B_{R} \setminus B_{r_2(t)}} \Big[\frac{1}{2}|\nabla Q_R^+|^2 + \underbrace{\frac{r_2(t)^4}{|y|^4}}_{\leq 1} f_{\rm bulk}(Q_R^+)\Big]\,dy\\
&	\leq \mcFRdom[Q_R^+] \leq 4\pi |k| s_+^2,
\end{align*}
which proves \eqref{Eq:MPPEDerived}.

\medskip

\noindent {\it Step 2. } We continue the argument by letting $\gamma\big|_{(-1,1)}$ be the linear interpolation between $\gamma(\pm 1)$, i.e. $\gamma(t) = \frac{1}{2}[(t + 1) \gamma(1) - (t - 1)\gamma(-1)]$. 

We now check \eqref{Eq:MPPathEnergy} for $|t| \leq 1$. Note that $\gamma(t) = \frac{1}{2}[(t + 1) Q_{R_0}^+ - (t - 1)Q_{R_0}^+]$ in $B_{R_0}$. A standard argument using the maximum principle (see the proof of \cite[Proposition 3]{Ma-Za}) shows that $|Q_{R_0}^{\pm}| \leq \sqrt{\frac{2}{3}}s_+$. Hence
\[
\mcF^{R_0}[\gamma(t)] \leq C R_0^2 + \frac{1}{2} \int_{B_{R_0}}|\nabla \gamma(t)|^2\,dx.
\]
This together with the convexity of the Dirichlet energy, the non-negativity of the integrand of $\mcF$, and \eqref{Eq:UnifEnergyBnd} gives
\[
\mcF^{R_0}[\gamma(t)] \leq C(R_0^2 + |k|).
\]
On the other hand, as $\gamma(t)(x) = \gamma(1)(x)$ in $B_R \setminus B_{R_0}$, we have, in view of non-negativity of the integrand of $\mcF$,
\[
\mcFRdom[\gamma(t)]=\mcF^{R_0}[\gamma(t)] +\mcF[\gamma(1), B_R\setminus B_{R_0}]
	\leq \mcF^{R_0}[\gamma_0(t)] +\mcFRdom[\gamma(1)].
\]
Recalling \eqref{Eq:MPPathEnergy} for $t = 1$, we conclude the proof of the claim.

Let us prove now that $Q_R^{mp}\notin \mcA_R^{rs}$ for sufficiently large $R$. Indeed, we take $R_1>\max(R_0,4C_1ek^2)$  such that  for all $R>R_1$ and $k\in 2\ZZ \setminus\{0\}$, we have 
$$
 \frac{\pi s_+^2 k^2}{8} \,(\ln \frac{R}{4C_1k^2} - \ln \ln \frac{ R}{4C_1k^2})>C_2(R_0^2 + |k|),
$$
where $C_1$ is the constant from \eqref{Eq:alphaRLB} corresponding to $\delta = 1/2$, and the constants $C_2$  and $R_0$ are the ones from \eqref{Eq:MPEnergy}. Then the mountain pass critical point $Q_R^{mp}$ (which belongs to $\mcA_R^{rs}$) has the energy $\mcFRdom(Q_R^{mp})$ bounded from above by $C_2|k|$ and thus does not belong to $\mcA_R^{str}$ thanks to Lemma \ref{Lem:StrongLogDiv}. 
In other words, $Q_R^{mp}$ is $O(2)$-symmetric and is not $\Z_2\times O(2)$-symmetric.

Let us now construct a second mountain pass critical point $\tilde Q_R^{mp}$. Indeed, the lack of $\Z_2\times O(2)$-symmetry implies that in the decomposition $Q_R^{mp}(x) =w_0(|x|)E_0+w_1(|x|)E_1+w_3(|x|)E_3$ we have that $w_3\not\equiv 0$. It follows that $\tilde Q_R^{mp}(x) =  w_0(|x|)E_0 + w_1(|x|)E_1 - w_3(|x|)E_3 $ is an additional critical point of $\mcFRdom$ with the same energy as $Q_R^{mp}$ (it is necessarily of mountain pass type).

Let us now construct a fifth critical point $Q_R^{str}$ that will be $k$-radially symmetric. Indeed, minimizing the energy $\mcFRdom \big|_{\mcA_R^{str}}$ one can show that $\mcFRdom$ has a critical point in $\mcA_R^{str}$, called $Q_R^{str}$. By the above energy estimates and Lemma \ref{Lem:StrongLogDiv}, it is clear that $Q_R^{str}$ differs from $Q_R^\pm$, $Q_{R}^{mp}$ and $\tilde Q_R^{mp}$. We have thus shown that $\mcFRdom$ has at least five $k$-fold $O(2)$-symmetric critical points in $\mcA_R$, at least four of which are not  $k$-fold $\Z_2 \times O(2)$-symmetric.
\end{proof}

\medskip
\noindent{\bf Acknowledgment.} The authors would like to thank the Isaac Newton Institute for Mathematical Sciences for support and hospitality during the programme {\it The mathematical design of new materials} when work on this paper was undertaken. This work was supported by EPSRC grant number EP/R014604/1.
R.I. acknowledges partial support by the ANR project ANR-14-CE25-0009-01. V.S. acknowledges support from EPSRC grant EP/K02390X/1 and Leverhulme grants RPG-2018-438. 
The work of A.Z. is supported by the Basque Government through the BERC 2018-2021
program, by Spanish Ministry of Economy and Competitiveness MINECO through BCAM
Severo Ochoa excellence accreditation SEV-2017-0718 and through project MTM2017-82184-R
funded by (AEI/FEDER, UE) and acronym ``DESFLU''. 
\appendix

\section{Remarks on minimizers for odd $k$}
\label{sec:appendixoddk}

In this appendix we would like to make some remarks on the energy bounds and the symmetry properties of minimizers of the energy
$$
\mcF_\eps[Q;\Omega]:=\int_{\Omega} \Big[ \frac{1}{2}|\nabla Q|^2 +\frac{1}{\eps^2} f_{\rm bulk}(Q)\Big]\,dx
$$ 
defined on a fixed finitely-connected, bounded $C^1$ domain $\Omega$ and subject to a given $C^1$ planar $\mcS_*$-valued boundary condition $Q_b$, i.e. $Q_b$ takes values in the set
\begin{equation}
\mcS_*^{planar} = \Big\{s_+\Big(v \otimes v - \frac{1}{3}I_3\Big)\, :\, v \in \Sphere^1\Big\} \text{ with } \Sphere^1 = \{(v_1, v_2, 0) \in \RR^3\, :\, |v| = 1\}  \subset \RR^3.
	\label{Eq:RP1S1}
\end{equation}

Note that $\mcS_*^{planar}$ is homeomorphic to $\mathbb{R} P^1$ and continuous maps from $\partial \Omega$ into $\mcS_*^{planar}$ have well-defined $\frac{1}{2}\ZZ$-valued degrees; see for instance Brezis, Coron and Lieb \cite[Section VIII, part B]{BrezisCoronLieb}. The relation between $\mcS_*^{planar}$ and $\Sphere^1$ in \eqref{Eq:RP1S1} is manifested in the fact that a map $Q \in C(\partial\Omega,\mcS_*^{planar})$ can be written in the form $Q = s_+\big(v \otimes v - \frac{1}{3}I_3\big)$ for some $v \in C(\partial\Omega,\Sphere^1)$ if and only if the degree $\frac{k}2$ of $Q$ is an integer (i.e., $k\in \Z$ is even), and, in which case, is equal to that of $v$.

In the discussion to follow, we assume that $Q_b$ has degree $\frac{k}{2}$ with $k \in \Z\setminus\{0\}$ (which does not necessarily have the form \eqref{Qbdef}-\eqref{def:n}). We consider both cases of {\bf even and odd $k$} and note that they are fundamentally different in the limit $\eps\to 0$: for odd $k$ the limiting energy is infinite, while for even $k$ the limit has finite energy.

As usual, let $H^1_{Q_b}(\Omega,\mcS_0)$ denote the set of $H^1$ maps from $\Omega$ into $\mcS_0$ equal to $Q_b$ on $\partial \Omega$.

\begin{remark} 
Let $\Omega \subset \RR^2$ be a fixed finitely-connected, bounded $C^1$ domain and $Q_b: \partial \Omega \rightarrow \mcS_*^{planar}$ be an arbitrary $C^1$ map of degree $\frac{k}{2}$ for some $k \in \ZZ$.
 
For {\bf{even }}$k \neq 0$, there exist $\eps_0>0$ and $C_1, C_2>0$ such that for all $\eps<\eps_0$
\be\label{bd-even}
C_1 \leq \min_{H^1_{Q_b}(\Omega,\mcS_0)} \mcF_\eps[\cdot; \Omega] \leq C_2.
\ee

For {\bf odd $k$}, there exist $\eps_1>0$ and $C > 0$ such that for all $\eps<\eps_1$
\begin{equation}\label{bd-odd}
\frac{\pi}{2} s_+^2 |\ln \eps | -C \leq \min_{H^1_{Q_b}(\Omega,\mcS_0)} \mcF_\eps[\cdot;\Omega] \leq \frac{\pi}{2} s_+^2 |\ln \eps | + C.
\end{equation}
\end{remark}

\begin{remark}
For the unit disk $\Omega =D$, the following can be stated in the case the boundary data
$Q_b$ is given by \eqref{Qbdef}-\eqref{def:n}. We have seen that when $k$ is {\bf even} the minimizers of $\mcF_\eps$ in $H^1_{Q_b}(D,\mcS_0)$ are $k$-fold $O(2)$-symmetric. 

When $k$ is {\bf odd} the symmetry of minimizers is more delicate and so far is {\bf unknown}. For $k=\pm 1$, we conjecture that there exists a unique minimizer and this minimizer is $\Z_2 \times O(2)$-symmetric -- see \cite{INSZ_CVPDE,INSZ18_CRAS} for some supporting evidence. In the case of odd $|k|>1$ the above energy bounds forbid minimizers to have $\Z_2 \times SO(2)$-symmetry (in view of Lemma \ref{Lem:StrongLogDiv}) as well as the configuration of $k$-vortices of degree $\pm \frac{1}{2}$ constructed in \cite{BaumanParkPhillips}. 
\end{remark}

We would like now to briefly outline how one can obtain energy bounds \eqref{bd-even} and \eqref{bd-odd}. When $D$ is a disk and $k$ is odd, these bounds were established by Canevari \cite{canevari2015biaxiality, canevari2017line}. We will see below that \eqref{bd-even} and the upper bound in \eqref{bd-odd} can be established by elementary arguments using some extension property for $\mcS_*^{planar}$-valued maps of degree zero. The proof of the lower bound in \eqref{bd-odd} is more substantial and draws on the corresponding aforementioned estimate for disks \cite{canevari2017line}.

Let us start with \eqref{bd-even} when $k$ is even. The lower bound $C_1 > 0$ can be taken to be the minimal Dirichlet energy under the given (non-constant) boundary data $Q_b$ (since $f_{\rm bulk} \geq 0$). For the upper bound, we construct a test function by splitting the domain $\Omega$ into a disk $B_R \subset \Omega$, which without loss of generality is assumed to be centered at the origin, of some small radius $R$, and the complement $\Omega \setminus B_R$. On the boundary of the disk, we impose the test function to have boundary data of the form \eqref{Qbdef}-\eqref{def:n}. We define $Q_{test}$ by joining together minimizers of $\mcF_\eps[\cdot ;B_R]$ and $\mcF_\eps[\cdot;\Omega \setminus B_R]$ with respect to the indicated boundary data on each subdomain. It is clear that the minimal energy satisfies the following bound
\be 
\min_{H^1_{Q_b}(\Omega,\mcS_0)} \mcF_\eps[\cdot ;\Omega] \leq \mcF_\eps[Q_{test};\Omega] = \min \mcF_\eps[Q;B_R] + \min \mcF_\eps[Q;\Omega \setminus B_R],
	\label{Eq:30VII19-A1}
\ee
where the two minimizations on the right hand side are under the constraint that $Q = Q_{test}$ on the respected boundary. We know (see Table \ref{Table1}) that the first term on the right hand side of \eqref{Eq:30VII19-A1} is bounded uniformly in $\eps$. Since the degree of the map $Q_{test}: \partial (\Omega \setminus B_R) \rightarrow \mcS_*^{planar}$ is zero we can use an $H^1$ extension -- see Lemma \ref{Lem:H1Ext} below -- to show that the second term on the right hand side of \eqref{Eq:30VII19-A1} is bounded uniformly in $\eps$. Estimate \eqref{bd-even} follows.

\begin{lemma}\label{Lem:H1Ext}
Let $G \subset \RR^2$ be a bounded finitely-connected $C^1$ domain. Then every map $Q \in C^1(\partial G, \mcS_*^{planar})$ of degree zero extends to a map $Q \in H^1(G,\mcS_*^{planar})$.
\end{lemma}

\begin{proof}
As $\mcS_*^{planar}$ is diffeomorphic to a circle, this result is well known (see \cite[Section I.2]{vortices}). See also \cite[p.126]{hirsch2012differential} for results on continuous extensions in any dimension.
\end{proof}

We now consider the upper bound in \eqref{bd-odd} when $k$ is odd. We select two disjoint disks $B_R(x_1)$ and $B_R(x_2)$ inside $\Omega$. On the boundary of these disks, we consider the test function $Q_{test}$: 
\begin{equation}
Q_{test}(x) = s_+\Big(n(x) \otimes n(x) - \frac{1}{3}I_3\Big) , \qquad x \in \partial B_R(x_1) \cup \partial B_R(x_2),
	\label{Eq:AppGBC}
\end{equation}
where we set for $Re^{i\varphi}=(R\cos \varphi, R\sin \varphi)\in \partial B_R(0)$, $0 \leq \varphi < 2\pi$:
$$
n(x_1 + Re^{i\varphi}) 
	:= (\cos \frac{\varphi}{2}, \sin\frac{\varphi}{2}, 0), \quad n(x_2 + Re^{i\varphi}) 
	:= (\cos \frac{(k-1)\varphi}{2}, \sin\frac{(k-1)\varphi}{2}, 0).
$$
In $\Omega \setminus (B_R(x_1) \cup B_R(x_2))$ and $B_R(x_2)$, the test function $Q_{test}$ is constructed by minimizing $\mcF_\eps$ under the indicated boundary conditions. In $B_R(x_1)$, $Q_{test}$ is a minimizer of $\mcF_\eps$ in $\Z_2 \times O(2)$-symmetry. Using Lemma~\ref{Lem:StrongLogDiv} (namely the upper bound in \eqref{Eq:alphaRLB}) and arguing as in the previous case, we arrive at the upper bound in \eqref{bd-odd}. 

We turn to the lower bound in \eqref{bd-odd} when $k$ is odd. Take some large disk $B_{R'}(0) \supset \Omega$. We impose on $\partial B_{R'}(0)$ a boundary condition of the form \eqref{Eq:AppGBC} where 
$$
n(R'\cos \varphi, R'\sin \varphi) 
	= (\cos \frac{\varphi}{2}, \sin\frac{\varphi}{2}, 0), \quad	0 \leq \varphi < 2\pi.
$$
By \cite[Proposition 15]{canevari2017line}, we have within the above boundary condition (B.C.) on $\partial B_{R'}(0)$:
\[
\min_{B.C.} \mcF_\eps[\cdot;B_{R'}(0) ] \geq \frac{\pi}{2} s_+^2 |\ln \eps | -C,
\]
and by \eqref{bd-even}, within the above boundary condition (B.C.) on $\partial (B_{R'}(0) \setminus \Omega)$:
\[
\min_{B.C.} \mcF_\eps[\cdot;B_{R'}(0) \setminus \Omega ] \leq C.
\]
The lower bound in \eqref{bd-odd} follows from the above two estimates and the inequality
\[
\min_{B.C.} \mcF_\eps[\cdot;B_{R'}(0)] \leq \min_{H^1_{Q_b}(\Omega,\mcS_0)} \mcF_\eps[\cdot;\Omega ] + \min_{B.C.} \mcF_\eps[\cdot;B_{R'}(0) \setminus \Omega].
\]

\section{The Euler-Lagrange equations near $Q_*$}\label{App:Tech}

In this appendix, we give the proof of Lemma \ref{lemma:EL} and Proposition \ref{prop:ABCest} which concern the Euler-Lagrange equations for critical points of $\mcF$ relative to the representation in Lemma \ref{lem:TubNbhd}.

\begin{proof}[Proof of Lemma~\ref{lemma:EL}] We set $v = \frac{n_* + \psi}{|n_* + \psi|}$, $\hat v = n_* + \psi$. Recall that $n_* \cdot \psi = 0$, $Q_\sharp = s_+(v \otimes v - \frac{1}{3} I_3)$ and $Q = Q_\sharp + \eps^2 P$. We calculate separately  the elastic part in $\mcF_\eps$ and then the bulk term.

\noindent {\bf 1. The elastic part:}
\be\label{rel:elasticfull}
\int_\Omega |\nabla Q|^2\,dx= \int_\Omega \Big[2s_+^2\,|\nabla v|^2 + 2\eps^2s_+ \nabla (v \otimes v):\nabla P+\eps^4 |\nabla P|^2\Big]\,dx.
\ee
(Here $\nabla X : \nabla Y = \sum_{i} (\nabla_i X \cdot \nabla_i Y)$ for two matrix-valued maps $X$ and $Y$.) 

We calculate individually the first two terms.

\medskip

\noindent {\bf a. The $|\nabla v|^2$ term.}  Using the identities
\[
\nabla v  = \frac{1}{|\hat v|} \nabla  \hat v - \frac{1}{2|\hat v|^3} \hat v \otimes \nabla |\hat v|^2,
\]
and $n_* \cdot \psi = 0$ (in particular $|\hat v|^2 = 1 + |\psi|^2$), we see that
\begin{align*}
|\nabla v|^2 
	&= \frac{1}{|\hat v|^2} ( |\nabla \hat v|^2 - \frac{1}{4|\hat v|^2} |\nabla |\hat v|^2|^2)\\
	&= \frac{1}{1 + |\psi|^2} \Big(|\nabla\psi|^2 + |\nabla n_*|^2 + 2\nabla n_* \cdot \nabla \psi - \frac{1}{4(1 + |\psi|^2)} |\nabla |\psi|^2|^2\Big).
\end{align*}
Noting that 
\begin{align*}
\int_D \frac{1}{1 + |\psi|^2}\nabla n_* \cdot \nabla \psi\,dx
	&= - \int_D \frac{1}{1 + |\psi|^2} \Big[\Delta n_* - \frac{1}{1 + |\psi|^2}\nabla n_* \cdot \nabla |\psi|^2\Big] \cdot \psi\,dx\\
	&=   \int_D \frac{1}{(1 + |\psi|^2)^2} (\nabla n_* \cdot \nabla |\psi|^2) \cdot \psi\,dx,
\end{align*}
we obtain
\be
\int_D |\nabla v|^2\,dx= \int_D |\nabla n_*|^2\,dx + \int_D \big[|\nabla\psi|^2 - |\nabla n_*|^2\,|\psi|^2\big]\,dx + \int_D g(x,\psi,\nabla\psi)\,dx,\label{rel:nablavexpansion}
\ee
where $g$ is super-cubic in $(\psi,\nabla \psi)$ at zero:
\begin{align}
g(x,\psi,\nabla\psi)
	&=\frac{1}{1 + |\psi|^2} \Big[- |\psi|^2 (|\nabla\psi|^2 - |\nabla n_*|^2\,|\psi|^2) \nonumber\\
		&\qquad\qquad + \frac{2}{1 + |\psi|^2} (\nabla n_* \cdot \nabla |\psi|^2) \cdot \psi - \frac{1}{4(1 + |\psi|^2)} |\nabla |\psi|^2|^2\Big].\non
\end{align}

\medskip
\noindent {\bf b. The gradient coupling term $\nabla (v \otimes v): \nabla P$.} We write
\begin{align}
\int_\Omega \nabla (v \otimes v): \nabla P
	&= \int_\Omega \nabla\Big(\frac{1}{1 + |\psi|^2}(n_* + \psi) \otimes (n_* + \psi)\Big) :\nabla P\non\\
	&= \int_\Omega \nabla (n_* \otimes n_*): \nabla P + \nabla(n_* \otimes \psi + \psi \otimes n_*): \nabla P + \nabla \hat g(x,\psi):\nabla P\label{rel:vpgradexpansion}
\end{align}
where $\hat g$ is super-quadratic in $\psi$ at $\psi = 0$:
\[
\hat g(x,\psi) = -\frac{|\psi|^2}{1 + |\psi|^2}(n_* + \psi) \otimes (n_* + \psi) + \psi \otimes \psi.
\]
The expression $\nabla(n_* \otimes \psi + \psi \otimes n_*): \nabla P$ on the right hand side of \eqref{rel:vpgradexpansion} contains some terms which are quadratic in the derivatives of $P$ and $\psi$. However, we can eliminate this quadratic character by using some specific geometric information as follows: we note that $P \in (T_{Q_*}\mcS_*)^\perp$ and $n_* \otimes \psi + \psi \otimes n_* \in T_{Q_*}\mcS_*$. Indeed, by \cite[Lemma 2]{NZ13-CVPDE}, $Pn_*$ is parallel to $n_*$ and so $Pn_* = (P \cdot (n_* \otimes n_*)) n_*$. Also, as $\psi \cdot n_* = 0$ and $\Delta n_* \parallel n_*$, we have $\Delta \psi \cdot n_* = -2 \nabla \psi \cdot \nabla n_*$. Thus,  integrating by parts using again $\Delta n_* \parallel n_*$ gives
\begin{align*}
\int_D \nabla(n_* \otimes \psi): \nabla P\,dx
	&= \int_D \sum_k \nabla_k P \cdot (\nabla_k n_* \otimes \psi + n_* \otimes \nabla_k \psi)\,dx\\
	&=  \int_D \Big[\sum_k\nabla_k P \cdot (\nabla_k n_* \otimes \psi) - (Pn_*) \cdot \Delta \psi - ((\nabla n_*)^t\,P) \cdot\nabla\psi \Big]\,dx\\
	&= \int_D \Big[\sum_k \nabla_k P \cdot (\nabla_k n_* \otimes \psi)\\
			&\qquad\qquad + 2(P \cdot (n_* \otimes n_*)) (\nabla n_* \cdot \nabla\psi) - ((\nabla n_*)^t\,P) \cdot\nabla\psi\Big]\,dx.
\end{align*}
As $P$ is symmetric, we hence get
\begin{align}
\int_D \nabla(n_* \otimes \psi &+ \psi \otimes n_*): \nabla P\,dx
	= 2\int_D \Big[\sum_k\nabla_k P \cdot (\nabla_k n_* \otimes \psi)\non\\
			&+ 2(P \cdot (n_* \otimes n_*)) (\nabla n_* \cdot \nabla\psi) - ((\nabla n_*)^t\,P) \cdot\nabla\psi\Big]\,dx.
			\label{rel:gradexp2}
\end{align}
It is readily seen that the integral on the right hand side is linear in the derivatives of $P$ and $\psi$. This cancellation will play a role on our later analysis: its contribution to the Euler-Lagrange equations of $\mcF_\eps$ is of first order rather than second order.

\bigskip
\noindent {\bf 2. The bulk part:} We expand $f_{\rm bulk}(Q) = f_{\rm bulk}(Q_\sharp + \eps^2 P)$ in terms of powers of $\eps$. As $\mcS_*$ is the set of minimum points of $f_{\rm bulk}$, we have $f_{\rm bulk}(Q_\sharp) = 0$ and $\nabla f_{\rm bulk}(Q_\sharp) = 0$. We have:
\begin{align}
f_{\rm bulk}(Q)
	&= f_{\rm bulk}(Q) - f_{\rm bulk}(Q_\sharp) - \eps^2\,\nabla f_{\rm bulk}(Q_\sharp) \cdot P\non\\
&= \eps^4 \Big[-\frac{a^2}{2}|P|^2-b^2 P^2 \cdot Q_\sharp+\frac{c^2}{2} |Q_\sharp|^2 |P|^2+c^2 |P \cdot Q_\sharp|^2\Big]\non\\
	&\qquad +\eps^6\big[-\frac{b^2}{3}\textrm{tr}(P^3)+c^2 P \cdot Q_\sharp \,|P|^2\big]+\eps^8\frac{c^2}{4}|P|^4\non\\
	&=\eps^4\,h(x,P) + \eps^4 \hat h(x,\psi, P)  + \eps^6\,\mathring{h}_\eps(x,\psi, P),\label{rel:fbulkexpansion}
\end{align}
where:
\begin{align*}
h(x,P)
	&=\frac{b^2\,s_+}{2} |P|^2 - b^2\,s_+ P^2 \cdot (n_* \otimes n_*) +c^2\,s_+^2 |P \cdot (n_* \otimes n_*)|^2\\
\hat h(x,\psi,P)
	&= 	-b^2\,s_+ P^2 \cdot (v \otimes v - n_* \otimes n_*) +c^2\,s_+^2 [|P \cdot (v \otimes v)|^2- |P \cdot (n_* \otimes n_*)|^2],\\
\mathring{h}_\eps(x,\psi,P)
	&= -\frac{b^2}{3}\textrm{tr}(P^3)+c^2s_+ P \cdot (v \otimes v) \,|P|^2 +\eps^2\frac{c^2}{4}|P|^4.
\end{align*}
(Here we have used the identity $-a^2 - \frac{b^2}{3}s_+ + \frac{2c^2}{3} s_+^2 = 0$.) Note also that, as $n_* \cdot \psi = 0$ and $Pn_* \parallel n_*$, $\hat h(x,\psi,P)$ and $\mathring{h}_\eps(x,\psi,P) - \mathring{h}_\eps(x,0,P)$ are super-quadratic in $\psi$ at $\psi = 0$.

We now put together all the previous expressions, to get a new form of the full energy. Using the expression of $|\nabla v|^2$ from \eqref{rel:nablavexpansion} in \eqref{rel:elasticfull} and putting  \eqref{rel:gradexp2} in \eqref{rel:vpgradexpansion} and then in \eqref{rel:elasticfull} provide the elastic part. Putting this together with the expansion of the bulk part \eqref{rel:fbulkexpansion} into \eqref{def:mcF}
 we obtain
\begin{align*}
\mcF_\epsilon[Q]
	&=  s_+^2\int_D |\nabla n_*|^2\,dx
		 + s_+^2 \int_D \big[|\nabla\psi|^2 - |\nabla n_*|^2\,|\psi|^2\big]\,dx
		 	+ \eps^2\int_D \Big[\frac{\eps^2}{2}|\nabla P|^2 + h(x,P)\Big]\,dx\\
		&\qquad + \eps^2\,s_+\int_D \Big[ \nabla (n_* \otimes n_*): \nabla P + 2\sum_k \nabla_k P : (\nabla_k n_* \otimes \psi)\\
			&\qquad\qquad \qquad\qquad + 4(P \cdot (n_* \otimes n_*)) (\nabla n_* \cdot \nabla\psi) - 2((\nabla n_*)^t\,P) \cdot\nabla\psi\Big]\,dx\\
		&\qquad +  \int_D [s_+^2 g(x,\psi,\nabla\psi) + \eps^2\,\hat h(x,\psi,P)]\,dx
	+ \eps^2\int_D \Big[s_+\nabla \hat g(x,\psi): \nabla P+ \eps^2\mathring{h}_\eps(x,\psi,P))\Big]\,dx.
\end{align*}
The Euler-Lagrange equations for $\mcF_\epsilon$ in terms of $\psi$ and $P$ are then readily found to be of the form
\begin{align}
-\Delta \psi - |\nabla n_*|^2\,\psi
	&= A[\psi] + \eps^2\,B_\eps[\psi,P] + \lambda_\eps(x)\,n_*,
		\label{Eq:psiEqA}\\
-\eps^2 \Delta P + \mathring{\nabla}_P h (x,P)
	&=  s_+ \Delta( n_* \otimes n_* )
		+  C_\eps[\psi,P] - \frac{1}{3} \tr(C_\eps[\psi,P])I_3 + F_\eps(x),
		\label{Eq:PEqA}
\end{align}
where 
\begin{align}
\mathring{\nabla}_P h (x,P) &:= \nabla_P h (x,P) - \frac{1}{3}\tr(\nabla_P h (x,P))I_3 \non \\
 &=b^2\,s_+\,P + \frac{2}{s_+^2}(-b^2\,s_+ + c^2\,s_+^2)\,(P \cdot Q_*) \,Q_*
 	\label{Eq:09V19-A11}
\end{align}
is the gradient of $h$ with respect to $P \in \mcS_0$,\footnote{In deriving \eqref{Eq:09V19-A11}, it is useful to keep in mind the relation that $Pn_* \parallel n_*$.} $\lambda_\eps$ is a Lagrange multiplier accounting for the constraint $\psi \cdot n_* = 0$, $F_\eps(x) \in T_{Q_*}\mcS_*$ is a Lagrange multiplier accounting for the constraint $P(x) \in (T_{Q_*}\mcS_*)^\perp$, and
\begin{align}
A[\psi]_j
	&= \frac{1}{2}\nabla_i \Big[\frac{\partial g}{\partial (\nabla_i \psi_j)}(x,\psi,\nabla\psi)\Big]
		- \frac{1}{2} \frac{\partial g}{\partial \psi_j} (x,\psi,\nabla\psi),\label{def:A}\\
B_\eps[\psi,P]_j
	&= -\frac{1}{2s_+^2}\frac{\partial \hat h}{\partial \psi_j}(x,\psi,P) - \frac{1}{2s_+^2}\eps^2 \frac{\partial \mathring{h}_\eps}{\partial \psi_j}(x,\psi,P)
		+ \frac{1}{2s_+}\,\frac{\partial \hat g}{\partial\psi_j}(x,\psi)\cdot \Delta P\non\\
		&\quad + \frac{1}{s_+} \nabla_i \big[2(P \cdot (n_* \otimes n_*)) \nabla_i (n_*)_j -  \nabla_i (n_*)_k\,P_{kj}] - \frac{1}{s_+}\nabla_i P_{jk} \,\nabla_i (n_*)_k\label{def:B}\\
C_\eps[\psi,P]_{ij}
	&= -\frac{\partial \hat h}{\partial P_{ij}} (x,\psi,P)
		- \eps^2\frac{\partial \mathring{h}_\eps}{\partial P_{ij}} (x,\psi,P) + s_+ \Delta \hat g_{ij}(x,\psi) \non\\
		&\quad + 2s_+ \nabla_k(\nabla_k (n_*)_i \psi_j) - 4s_+ (n_*)_i (n_*)_j (\nabla n_* \cdot \nabla\psi) + 2s_+ \nabla_k (n_*)_i  \nabla_k \psi_j.\label{def:C}
\end{align} 
This finishes the proof of Lemma~\ref{lemma:EL}.
\end{proof}

We continue with the proof of the Lipschitz-type estimates for $A,B_\eps$ and $C_\eps$:
\bigskip
\begin{proof}[Proof of Proposition~\ref{prop:ABCest}]
Using the definitions of $A[\psi]$, $B_\eps[\psi,P]$ given in \eqref{def:A}, \eqref{def:B}, \eqref{def:C} together with the fact that $\frac{\partial}{\partial\psi_j} \hat h(x,0,P) = \frac{\partial}{\partial\psi_j} \mathring{h}_\eps(x,0,P) = 0$,\footnote{Recall that this is a consequence of the relations $n_* \cdot \psi = 0$ and $Pn_* \parallel n_*$.} we obtain $A[0] = 0$,
\begin{align*}
|A[\psi] - A[\tilde\psi]|
	&\leq C(|\psi| + |\tilde\psi|)\Big[ (1 + |\Delta \psi| + |\Delta\tilde\psi|)|\psi - \tilde\psi|+ |\nabla^2(\psi - \tilde\psi)|\\
	&\qquad\qquad +  (1 +  |\nabla \psi| + |\nabla\tilde \psi|)|\nabla(\psi - \tilde\psi)|\Big]\\
	&\qquad\qquad + C(|\nabla \psi| + |\nabla\tilde\psi|)(1 + |\nabla \psi| + |\nabla\tilde\psi|)|\psi - \tilde\psi|,
\\
|B_\eps(0,P)|
	&\leq  C(|\nabla P| + |P|),\\
|B_\eps[\psi,P] - B_\eps[\tilde \psi, P]|
	&\leq C\,(|\Delta P| + |P|^2 + \eps^2 |P|^3)|\psi - \tilde \psi|,\\
	|B_\eps[\psi, P] - B_\eps[\psi,\tilde P]| 
	&\leq C|\psi|[|\Delta (P- \tilde P)| + (|P| + |\tilde P|)(1 + \eps^2 (|P| + |\tilde P|)) |P - \tilde P|]\\
		&\qquad + C[|\nabla(P- \tilde P)| +  |P - \tilde P|].
\end{align*} 
which imply the claimed estimates \eqref{est:Aest0}, \eqref{est:Aest}, \eqref{est:Best0}, \eqref{est:Best1} and \eqref{est:Best2} in view of the embedding $H^2(D) \hookrightarrow W^{1,4}(D) \hookrightarrow L^\infty(D)$.

We split $C_\eps[\psi,P] = C_\eps^{(1)}[\psi,P] + C^{(2)}[\psi]$ where
\begin{align*}
C_\eps^{(1)}[\psi,P]_{ij}
	&= -\frac{\partial \hat h}{\partial P_{ij}} (x,\psi,P)
		-  \eps^2 \frac{\partial \mathring{h}_\eps}{\partial P_{ij}}(x,\psi,P),\\
C^{(2)}[\psi]_{ij}
	&=   2s_+ \nabla_k(\nabla_k (n_*)_i \psi_j) - 4s_+ (n_*)_i (n_*)_j (\nabla n_* \cdot \nabla\psi) + 2s_+ \nabla_k (n_*)_i  \nabla_k \psi_j\non\\
		&\qquad + s_+ \Delta \hat g_{ij}(x,\psi).
\end{align*}

We have
\begin{align*}
|C^{(2)}[\psi] - C^{(2)}[\tilde \psi]|
	&\leq C(1 + |\Delta \psi| + |\Delta \tilde\psi| + |\nabla\psi|^2 + |\nabla\tilde\psi|^2)|\psi - \tilde \psi|\\
		&\qquad + C(1 + |\psi| + |\tilde\psi| + |\nabla\psi| + |\nabla\tilde\psi|)|\nabla\psi - \nabla \tilde \psi| \\
		&\qquad + C(|\psi| + |\tilde\psi|) |\Delta\psi - \Delta \tilde \psi|,
\end{align*}
which implies
\[
\|C^{(2)}[\psi] - C^{(2)}[\tilde \psi]\|_{L^2(D)} 
	\leq C(1 + \|\psi\|_{H^2(D)} + \|\tilde\psi\|_{H^2(D)}+\|\nabla\psi\|_{L^4(D)}^2+\|\nabla\tilde\psi\|_{L^4(D)}^2) \|\psi - \tilde\psi\|_{H^2(D)}.
\] 

As for $C_\eps^{(1)}$, we have
\begin{align*}
|C_\eps^{(1)}[\psi,P] - C_\eps^{(1)}[\tilde \psi, \tilde P]|
	&\leq 
		C\Big[|\psi| + |\tilde\psi|  + \eps^2(|P| + |\tilde P|)\big(1 + \eps^2 (|P| + |\tilde P|)\big)\Big]|P - \tilde P| \\
		&\qquad + C(|P| + |\tilde P|)\big(1 + \eps^2 (|P| + |\tilde P|)\big)\,|\psi - \tilde \psi|.
\end{align*}
Hence
\begin{align*}
&\|C_\eps^{(1)}[\psi,P] - C_\eps^{(1)}[\tilde \psi, \tilde P]\|_{L^2(D)}\\
	&\qquad\leq 
		C(\|\psi\|_{H^2(D)} +\|\tilde\psi\|_{H^2(D)})  \|P - \tilde P\|_{L^2(D)}\\
			&\qquad\qquad+ C(\|P\|_{L^2(D)}+\|\tilde P\|_{L^2(D)})\,\big(1 + \eps^2 (\|P\|_{H^2(D)} + \|\tilde P\|_{H^2(D)})\big) \|\psi - \tilde \psi\|_{H^2(D)}
			\\
			&\qquad \qquad  + C\eps^2(\|P\|_{L^4(D)} + \|\tilde P\|_{L^4(D)}) \big(1 + \eps^2 (\|P\|_{H^2(D)} + \|\tilde P\|_{H^2(D)})\big)
			\|P - \tilde P\|_{H^1(D)},
\end{align*}
and thus we obtain the claimed estimate \eqref{est:Cest}.
\end{proof}


\section{Proof of Proposition \ref{cor:LInverse}}\label{App:LpaInv}

Proposition~\ref{cor:LInverse} easily follows from Lax-Milgram's theorem and Lemma~\ref{Lem:StrStab} below. 

\begin{lemma}\label{lem:NSStab}
Let $n_*$ be given by \eqref{Eq:n*Def}. For any $\zeta \in H_0^1(D,\RR)$, there holds
\[
I[\zeta] := \int_D [|\nabla \zeta|^2 - |\nabla n_*|^2\,\zeta^2]\,dx \geq 0.
\]
In particular, $n_*$ is a stable harmonic map. Equality holds if and only if $\zeta = \frac{t(1 - r^k)}{1+r^k}$ for some $t \in \RR$.
\end{lemma}

\begin{proof} Let $L_\parallel = - \Delta  -|\nabla n_*|^2$. W.l.o.g., we assume $n_*:=n_*^+$. Then $n_3 = n_* \cdot e_3 = \frac{1-r^k}{1+r^k} > 0$ in $D$ and note that $L_\parallel n_3 = 0$. Decomposing $\zeta = n_3\xi$, a direct computation yields (cf. \cite[Lemma A.1]{INSZ3})
\[
I[\zeta] := \int_D [|\nabla \zeta|^2 - |\nabla n_*|^2\,\zeta^2]\,dx = \int_D n_3^2|\nabla \xi|^2\,dx \geq 0.
\]
The assertion follows.
\end{proof}

\begin{lemma}\label{Lem:StrStab}
Let $n_*$ be given by \eqref{Eq:n*Def}. Then $n_*$ is strictly stable, i.e. there exists some number $c_0 > 0$ such that for any $\zeta \in H_0^1(D,\RR^3)$ with $\zeta \cdot n_* = 0$ a.e. in $D$, there holds
\[
I[\zeta] = \int_D [|\nabla \zeta|^2 - |\nabla n_*|^2\,|\zeta|^2]\,dx \geq c_0\int_D [|\nabla \zeta|^2 + |\zeta|^2]\,dx.
\]
\end{lemma}

\begin{proof}
By Lemma \ref{lem:NSStab}, $I$ is non-negative on $H^1_0(D,\RR^3)$. Let 
\[
\lambda_1 = \inf \Big\{I[\zeta]: \zeta \in H_0^1(D,\RR^3), \|\zeta\|_{L^2(D)} = 1, \zeta \cdot n_* = 0 \text{ a.e. in } D\Big\} \geq 0.
\]
Using the smoothness of $n_*$, we can apply the direct method of the calculus of variations to show that $\lambda_1$ is achieved by some $\bar \zeta \in H_0^1(D,\RR^3)$ satisfying $\|\bar\zeta\|_{L^2(D)} = 1$ and $\bar\zeta \cdot n_* = 0$ a.e. in $D$. 

If $\lambda_1 = 0$, Lemma \ref{lem:NSStab} implies that each component of $\bar\zeta$ is proportional to $\frac{1 - r^k}{1+r^k}$. However, as $\bar\zeta \cdot n_* = 0$, this is possible only if $\bar \zeta \equiv 0$, which contradicts $\|\bar\zeta\|_{L^2(D)} = 1$.

We thus have that $\lambda_1 > 0$. Consequently, as $|\nabla n_*|$ is bounded, there exists some $\eta > 0$ such that, for any $\zeta \in H^1_0(D,\RR^3)$ with $\zeta \cdot n_* = 0$ a.e. in $D$,
\[
\int_D [|\nabla \zeta|^2 - |\nabla n_*|^2\,|\zeta|^2]\,dx \geq \lambda_1\int_D |\zeta|^2\,dx \geq \eta \int_D |\nabla n_*|^2\,|\zeta|^2\,dx,
\]
which implies
\[
\int_D [|\nabla \zeta|^2 - |\nabla n_*|^2\,|\zeta|^2]\,dx \geq \frac{\eta}{1 + \eta} \int_D |\nabla \zeta|^2\,dx.
\]
The conclusion is readily seen.
\end{proof}


\section{A calculus lemma}\label{app:Calc}

\begin{lemma}\label{Lem:gMin}
Let $g(x,y,z)= 2x^3 - 6x y^2 + 3xz^2 + 3\sqrt{3} yz^2$ for every $x,y,z\in \RR$. Then
\[
-2(x^2 + y^2 + z^2)^{3/2} \leq g(x,y,z) \leq 2(x^2 + y^2 + z^2)^{3/2}.
\]
Equality in the first inequality holds if and only if $(x,y,z) = s(\frac{1}{2}, \frac{\sqrt{3}}{2}, 0)$ for some $s \geq 0$ or $x + \sqrt{3}y = -\sqrt{x^2 + y^2 + z^2}$. Equality in the second inequality holds if and only if $(x,y,z) = s(\frac{1}{2}, \frac{\sqrt{3}}{2}, 0)$ for some $s \leq 0$ or $x + \sqrt{3}y = \sqrt{x^2 + y^2 + z^2}$.
\end{lemma}

\begin{proof} Since $g$ is three-homogeneous, it suffices consider the extremization problem
\[
\max\{g: x^2 + y^2 + z^2 = 1\} \text{ and } \min\{g: x^2 + y^2 + z^2 = 1\}.
\]
We rewrite
$
g = (x + \sqrt{3}y)(2x^2 + 3z^2 - 2\sqrt{3} xy),
$
and so when $x^2 + y^2 + z^2 = 1$,
\[
g = (x + \sqrt{3}y)(3 - (x + \sqrt{3}y)^2) = \tilde g(x + \sqrt{3}y),
\]
where $\tilde g(t) = 3t - t^3$. As $|x + \sqrt{3}y| \leq 2$ when $x^2 + y^2 + z^2 = 1$, $\max_{[-2,2]} \tilde g = 2$, which is achieved for $t \in \{-2, 1\}$, $\min_{[-2,2]} \tilde g = -2$, which is achieved for $t \in \{-1, 2\}$, the conclusion follows.
\end{proof}

\def\cprime{$'$}

\end{document}